\documentclass[review,onefignum,onetabnum]{siamart190516}

\usepackage{lineno}
\usepackage{lipsum}
\usepackage{amsfonts}
\usepackage{graphicx}
\usepackage[caption=false,farskip=0pt]{subfig}
\usepackage{epstopdf}
\usepackage{algorithmic,arydshln,etoolbox}

\ifpdf
  \DeclareGraphicsExtensions{.eps,.pdf,.png,.jpg}
\else
  \DeclareGraphicsExtensions{.eps}
\fi

\newsiamremark{remark}{Remark}
\newsiamremark{hypothesis}{Hypothesis}
\crefname{hypothesis}{Hypothesis}{Hypotheses}
\newsiamthm{claim}{Claim}
\newsiamthm{conjecture}{Conjecture}

\allowdisplaybreaks
\nolinenumbers

\headers{The best convergence rate of LP approximation}{Shuhuang Xiang, Shunfeng Yang, and Yanghao Wu}

\title{On the best convergence rates of lightning plus polynomial approximations\thanks{Submitted to the editors DATE.
\funding{This work was funded by National Science Foundation of China (No. 12271528).}}}

\author{Shuhuang Xiang\thanks{School of Mathematics and Statistics, Central South University, Changsha 410083, Hunan, People's Republic of China
  (\email{xiangsh@csu.edu.cn}).}
\and Shunfeng Yang\thanks{School of Mathematics and Statistics, Central South University, Changsha 410083, Hunan, People's Republic of China
  (\email{yangshunfeng@163.com}), corresponding author.}
  \and Yanghao Wu\thanks{School of Mathematics and Statistics, Central South University, Changsha 410083, Hunan, People's Republic of China (\email{wyanghao96@163.com}).
  }}

\usepackage{amsopn}

\ifpdf
\hypersetup{
  pdftitle={ rational approximation of exponential clustered poles},
  pdfauthor={Shuhuang Xiang, Shunfeng Yang and Yanghao Wu}
}
\fi

\begin{document}

\maketitle

% REQUIRED
\begin{abstract}
Building on introducing  exponentially clustered poles, Trefethen and his collaborators
introduced lightning algorithms for approximating functions of singularities. These schemes may achieve root-exponential convergence rates. In particular, based on a specific choice of the  %clustered
parameter of the tapered exponentially clustered poles,
the lightning approximation with either a low-degree polynomial basis
may achieve the optimal convergence rate simply  as  the best rational approximation
for prototype $x^\alpha$ on $[0,1]$, which was illustrated through delicate numerical experiments and conjectured in [SIAM J. Numer. Anal., 61:2580-2600, 2023].
By utilizing Poisson's summation formula and results akin to Paley-Wiener Theorem,
we rigorously show that  all these schemes with  a low-degree polynomial basis achieve  root-exponential convergence rates with exact orders in approximating $x^\alpha$ for arbitrary clustered parameters theoretically, and provide the best choices of the parameter to achieve the fastest convergence rate for each type of clustered poles, from which the conjecture
is confirmed as a special case.
 Ample numerical evidences
demonstrate the optimality and sharpness of the estimates.
\end{abstract}

% REQUIRED

% REQUIRED

%===============================================================================
\section{Introduction}
\label{sec:Int}
In recent years, there has been a significant amount of research dedicated to the approximation of functions with
branching singularities on the boundary. Efficient and powerful lightning schemes have been developed via rational functions \cite{Brubeck2022,GopTre2019,Gopal2019,Herremans2023,Nakatsukasa2021,Trefethen2021} by applying exponential clustered poles.

Trefethen, Nakatsukasa and Weideman \cite{Trefethen2021} introduced two pole distribution models on the interval $[-C,0]$ with $C$ a positive constant. The first is the uniform exponential clustering of $N$ poles
\begin{equation*}%\label{eq:uniform}
q_j =-C\exp\big(-\pi j/\sqrt{N}\big),\quad 0\leq j\leq N-1,
\end{equation*}
while the second is the tapered exponential clustering of the poles
\begin{equation}\label{eq:1tapered}
p_j =-C\exp\left(-\sqrt{2}\pi\big(\sqrt{N}-\sqrt{j}\big)/\sqrt{\alpha}\right),\quad 1\le j\le N.
\end{equation}
The accuracy of the lightning methods using these two exponentially clustered poles for  approximating the prototype function $x^\alpha$  with $0<\alpha<1$ on $[0,1]$  are $\mathcal{O}(e^{-\pi\sqrt{\alpha N}})$
%for specific $\sigma=\pi$
and $\mathcal{O}(e^{-\pi\sqrt{2\alpha N}})$, respectively.

It is well known from  Stahl \cite{Stahl2003} that the best rational approximant $R_{N}(x)$ to $x^\alpha$  converges at a root-exponential rate
\begin{align*}%\label{eq:newmann1}
  \lim_{N \to \infty} e^{2\pi \sqrt{\alpha N}} \max_{x\in [0,1]}\big|x^\alpha-R_N(x)\big| = 4^{1+\alpha}|\sin(\alpha\pi)|,
\end{align*}
or equivalently for $|x|^\alpha$ on $[-1,1]$
\begin{align*}%\label{eq:newmann2}
  \lim_{N \to \infty} e^{\pi \sqrt{\alpha N}}  \max_{x\in [-1,1]}\big||x|^\alpha-R_N(x)\big|
  = 4^{1+\alpha/2} \left|\sin\left(\frac{\alpha\pi}{2}\right)\right|
\end{align*}
for each $\alpha>0$.

To achieve the minimax convergence rate $\mathcal{O}(e^{-2\pi\sqrt{\alpha N}})$, Herremans, Huybrechs and Trefethen \cite{Herremans2023} introduced a lightning $+$ polynomial approximation (LP)
\begin{equation}\label{eq:rat}
x^\alpha\approx r_N(x)=\sum_{j=1}^{N_1}\frac{a_j}{x-p_j}+\sum_{j=0}^{N_2} b_jx^j:=r_{N_1}(x)+P_{N_2}(x), \quad N=N_1+N_2
\end{equation}
based upon
a new type of tapered exponential clustering:
\begin{equation}\label{eq:tapered2}
%p_j =-C\exp(-2\pi(\sqrt{N_1}-\sqrt{j})/\sqrt{\alpha}),\quad 1\leq j\leq N_1.
p_j =-C\exp\left(-\sigma\big(\sqrt{N_1}-\sqrt{j}\big)\right),\quad 1\leq j\leq N_1
\end{equation}
with $\sigma=\frac{2\pi}{\sqrt{\alpha}}$, which leads to a significant increase in the achievable accuracy as well as the convergence
rate compared with lightning methods with $N=N_1$. See  Fig. \ref{ratecomparison} and \cite{Herremans2023} for more details.
In addition,
using this clustering,
they showed that the degree $N_2$  of  polynomial $P_{N_2}$ can be chosen as $N_2=\mathcal{O}(\sqrt{N_1})$, where
the coefficients  in \eqref{eq:rat} are evaluated by  a least-squares system of the discrete best approximation.

\begin{figure}[htbp]
\centerline{\includegraphics[width=11cm]{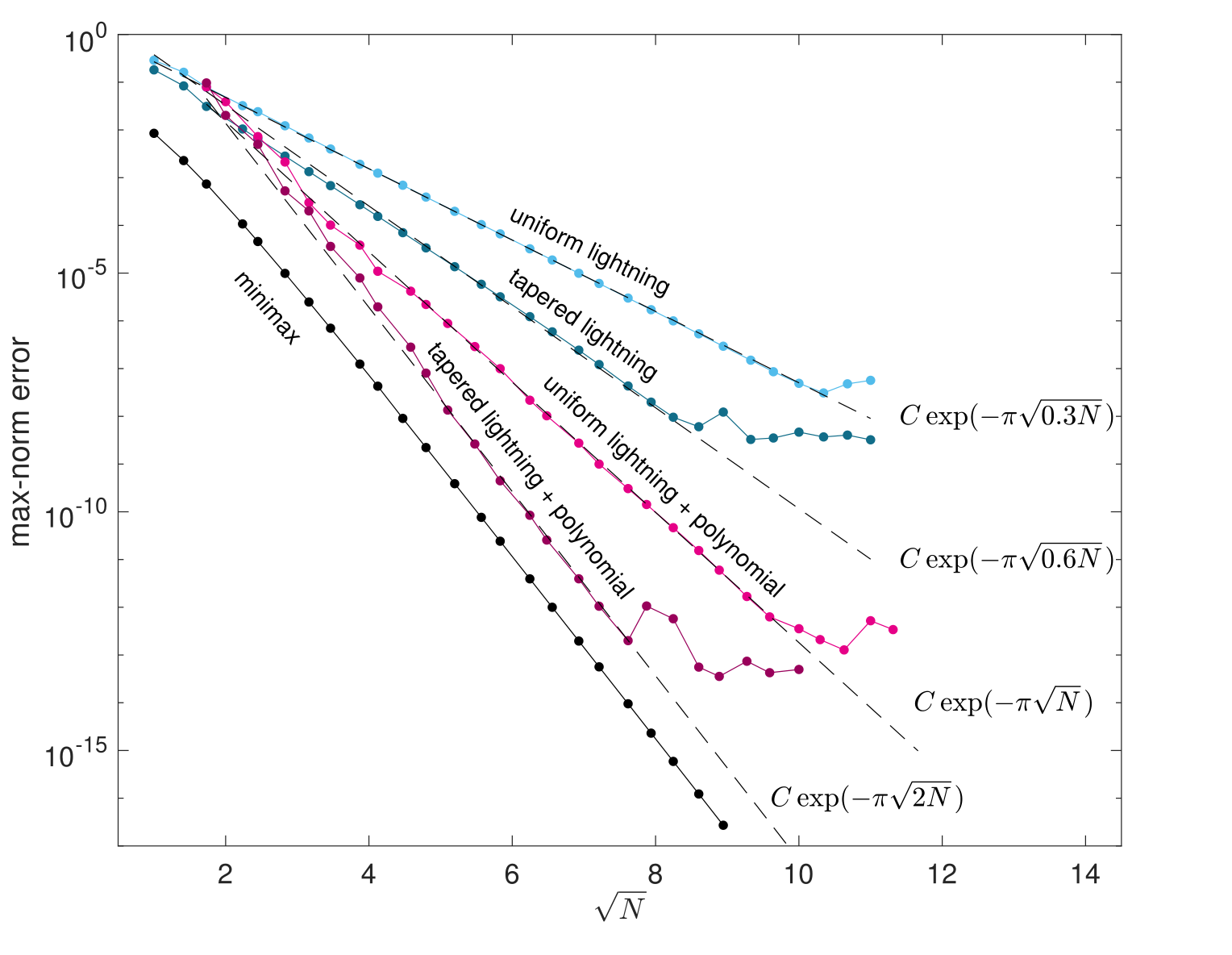}}
\caption{Convergence rates of the LPs with $\sigma=2\pi$ and $\sigma=2\sqrt{2}\pi$ compared with lightning with  $\sigma=\pi$ and $\sigma=\sqrt{2}\pi$ for $\sqrt{x}$ \cite{Herremans2023}, respectively. The lightning approximations (cf. \cite{Trefethen2021}) are equipped with $N =N_1$, while the LPSs with $N = N_1 + N_2$ and $N_2={\rm ceil}(1.3\sqrt{N_1})$. The best rational approximation of $\sqrt{x}$ based on solving a nonlinear approximation problem with free poles was studied
by Vja\v{c}eslavov \cite{1974Approximation}.}
\label{ratecomparison}
\end{figure}

Theoretical proof on the optimal convergence rate $\mathcal{O}(e^{-2\pi\sqrt{\alpha N}})$ for $\sqrt{x}$ is presented
in \cite{Herremans2023}
by an assumption (Conjecture A.7) regarding to the quadrature error of trapezoidal rule.
For $x^\alpha$ with $0<\alpha<1$ they put forth the following conjecture.

\begin{conjecture}\label{Conjecture 3.1}
%{\bf Conjecture 3.1}
\cite[Conjecture 3.1]{Herremans2023}. There exist coefficients $\{a_j\}_{j=1}^{N_1}$ and a polynomial $P_{N_2}$ with
$N_2 = \mathcal{O}(\sqrt{N_1})$, for which the LP $r_N(x)$ \eqref{eq:rat} having
tapered lightning poles \eqref{eq:tapered2} with $\sigma = \frac{2\pi}{\sqrt{\alpha}}$
satisfies:
\begin{equation*}\label{eq: brate}
|r_N(x)-x^\alpha|=\mathcal{O}\big(e^{-2\pi\sqrt{\alpha N}}\big)
\end{equation*}
as $N \rightarrow \infty$, uniformly for $x\in [0,1]$.
\end{conjecture}
%\bigskip

Conjecture \ref{Conjecture 3.1} states that based on the specific $\sigma = \frac{2\pi}{\sqrt{\alpha}}$, the LP exhibits a root-exponential convergence rate, which aligns with the best rational approximation in the sense of Stahl \cite{Stahl2003}. Compared  \eqref{eq:1tapered}  to \eqref{eq:tapered2}, the parameter $\sigma$ is increased by a factor of $\sqrt{2}$, which results in a sparser tapered clustering and a faster convergence rate from $\mathcal{O}(e^{-\pi\sqrt{2\alpha N}})$ to $\mathcal{O}(e^{-2\pi\sqrt{\alpha N}})$.

In this paper, we shall show that the LP \eqref{eq:rat} for general clustered parameter $\sigma>0$ in
\eqref{eq:tapered2} achieves  root-exponential convergence rate with an exact convergence order in approximating $x^\alpha$ with $0<\alpha<1$ and the choice of $\sigma=\frac{2\pi}{\sqrt{\alpha}}$ is the best among all $\sigma>0$ to get the fastest convergence rate. While
for  the LP
\begin{equation}\label{LPbasedonuniformclupole}
 x^\alpha\approx \bar{r}_N(x)=\sum_{j=1}^{N_1}\frac{\bar{a}_j}{x-q_j}+\sum_{j=0}^{N_2} \bar{b}_jx^j:=\bar{r}_{N_1}(x)+\bar{P}_{N_2}(x), \quad N=N_1+N_2
\end{equation}
stemmed from the uniform exponentially clustered poles
\begin{equation}\label{eq:uniform}
q_j =-C\exp\big(-\sigma j/\sqrt{N_1}\big),\quad 0\leq j\leq N_1-1,
\end{equation}
we shall show that the fastest convergence rate  is $\mathcal{O}(e^{-\pi\sqrt{2\alpha N}})$, which is achieved by choosing $\sigma=\frac{\sqrt{2}\pi}{\sqrt{\alpha}}$.

\begin{theorem}\label{mainthm}
There exist coefficients $\{a_j\}_{j=1}^{N_1}$ and a polynomial $P_{N_2}$, for which  $r_N(x)$ \eqref{eq:rat} having
tapered lightning poles \eqref{eq:tapered2} with $\sigma>0$
satisfies
\begin{equation}\label{eq: brateo}
|r_N(x)-x^\alpha|=e^{\sigma \sqrt{2M_0}}\cdot\left\{\begin{array}{ll}
\mathcal{O}(e^{-\sigma\alpha\sqrt{N}}),&\sigma\le \frac{2\pi}{\sqrt{\alpha}},\\
\frac{\mathcal{O}(1)}{e^{\frac{4\pi^2}{\sigma}\sqrt{N}}-1},&\sigma> \frac{2\pi}{\sqrt{\alpha}},
\end{array}\right.
\end{equation}
where
\begin{equation}\label{DefinitionOfM0}
M_0=2\max\left\{\frac{\sigma^2}{4},1+{\rm ceil}\left(\frac{9\pi^2}{\sigma^2}\right),\ 2{\rm ceil}\left[\left(1+\sqrt{\frac{\pi}{\sigma}}\right)^4\right]\right\},
\end{equation}
and there exist $\{\bar{a}_j\}_{j=1}^{N_1}$ and a polynomial $\bar{P}_{N_2}$ such that $\bar{r}_N(x)$
 \eqref{LPbasedonuniformclupole} having poles \eqref{eq:uniform} with $\sigma>0$ satisfies
\begin{equation}\label{eq: brateounif}
|\bar{r}_N(x)-x^\alpha|=\left\{\begin{array}{ll}
\mathcal{O}(e^{-\sigma\alpha\sqrt{N}}),&\sigma\le \frac{\sqrt{2}\pi}{\sqrt{\alpha}},\\
\frac{\mathcal{O}(1)}{e^{\frac{2\pi^2}{\sigma}\sqrt{N}}-1},&\sigma> \frac{\sqrt{2}\pi}{\sqrt{\alpha}},
\end{array}\right.
\end{equation}
%for the uniformly clustered poles \eqref{eq:uniformgen}
%\begin{equation}\label{eq:uniformgen}
%q_j =-C\exp\big(-\sigma j/\sqrt{N_1}\big),\quad 0\le q_j\leq N_1-1,
%\end{equation}
as $N =N_1+N_2\rightarrow \infty$, uniformly for $x\in [0, 1]$ and all the constants in $\mathcal{O}$ terms are independent of $\alpha\in (0,1)$, $\sigma>0$ and $N$. Both of the polynomials $P_{N_2}$ and $\bar{P}_{N_2}$ are of degree
$N_2 = \mathcal{O}(\sqrt{N_1})$.
\end{theorem}

To approximate multivariate functions with curves of singularities, Boull\'{e}, Hermanns and Huybrechs  \cite{boulle2023multivariate}  extended the LPs \eqref{eq:rat} and    \eqref{LPbasedonuniformclupole}    to explore the
construction of multivariate rational approximation by adopting the tapered exponentially clustered poles
\begin{align}\label{eq:1taperedimag}
\overline{p}_j =&i\sqrt{C}\exp\left(-\frac{\sigma}{2}\big(\sqrt{N_1}-\sqrt{j}\big)\right),\quad 1\le j\le N_1,
\end{align}
i.e., the square roots of $p_j$ \eqref{eq:tapered2} distributed along the imaginary axis.
The corresponding LP is of the form
\begin{align}
\widetilde{r}_N(x)=&\sum_{j=1}^{N_1}\left\{\frac{\widetilde{a}_j}{x-\bar{p}_j}-\frac{\widetilde{a}_j}{x+\bar{p}_j}\right\}+\sum_{j=0}^{N_2} \widetilde{b}_jx^j:=\widetilde{r}_{N_1}(x)+\widetilde{P}_{N_2}(x), \quad N=N_1+N_2.\label{LPwithimaginarypoles}
\end{align}
Analogously, equipped with the uniform exponentially clustered poles
\begin{align}%\label{eq:1taperedimag}
\overline{q}_j =&i\sqrt{C}\exp\left(-\frac{\sigma j}{2\sqrt{N_1}}\right),\quad 0\le j\le N_1-1,\label{eq:uniformgenimag_uniformal}
\end{align}
the LP is of the form
\begin{align}
\widehat{r}_N(x)=&\sum_{j=1}^{N_1}\left\{\frac{\widehat{a}_j}{x-\overline{q}_j}-\frac{\widehat{a}_j}{x+\overline{q}_j}\right\}+\sum_{j=0}^{N_2} \widehat{b}_jx^j:=\widehat{r}_{N_1}(x)+\widehat{P}_{N_2}(x), \quad N=N_1+N_2.\label{LPwithuniformalclusteringpoles}
\end{align}

By applying the integral representation of $x^\alpha$ and derivation of rational schemes of   \eqref{eq:rat}  and  \eqref{LPbasedonuniformclupole}, it directly leads  by Theorem \ref{mainthm} to
\begin{theorem}\label{mainthm2}
There exist coefficients $\{\widetilde{a}_j\}_{j=1}^{N_1}$ and a polynomial $\widetilde{P}_{N_2}$, for which  $\widetilde{r}_N(x)$ \eqref{LPwithimaginarypoles}
with
tapered lightning poles \eqref{eq:1taperedimag}
satisfies for  $0<\alpha<1$ that
\begin{equation}\label{eq: brateoimag}
\Big|\widetilde r_N(x)-|x|^{2\alpha}\Big|
=e^{\sigma \sqrt{2M_0}}\cdot
\left\{\begin{array}{ll}
\mathcal{O}(e^{-\sigma\alpha\sqrt{N}}),&\sigma\le \frac{2\pi}{\sqrt{\alpha}},\\
\frac{\mathcal{O}(1)}{e^{\frac{4\pi^2}{\sigma}\sqrt{N}}-1},&\sigma> \frac{2\pi}{\sqrt{\alpha}},
\end{array}\right.
\end{equation}
and there exist $\{\widehat{a}_j\}_{j=1}^{N_1}$ and a polynomial $\widehat{P}_{N_2}$, such that $\widehat{r}_N(x)$ \eqref{LPwithuniformalclusteringpoles}
% \begin{equation}\label{eq:taperedratimag}
%\widetilde{r}_N(x)=\sum_{j=0}^{N_1-1}\left\{\frac{a_j}{x-\overline{q}_j}-\frac{a_j}{x+\overline{q}_j}\right\}+\sum_{j=0}^{N_2} b_jx^j, \quad N=N_1+N_2
%\end{equation}
with uniform lightning poles \eqref{eq:uniformgenimag_uniformal}
satisfies for that
\begin{equation}\label{eq: brateo_uniform_imag}
\Big|\widehat r_N(x)-|x|^{2\alpha}\Big|
=\left\{\begin{array}{ll}
\mathcal{O}(e^{-\sigma\alpha\sqrt{N}}),&\sigma\le \frac{\sqrt{2}\pi}{\sqrt{\alpha}},\\
\frac{\mathcal{O}(1)}{e^{\frac{2\pi^2}{\sigma}\sqrt{N}}-1},&\sigma> \frac{\sqrt{2}\pi}{\sqrt{\alpha}},
\end{array}\right.
\end{equation}
as $N =N_1+N_2\rightarrow \infty$ uniformly for  $x\in [-1,1]$,  and the constants in $\mathcal{O}$ terms are independent of $\alpha\in (0,1)$, $\sigma>0$ and $N$.
Both of the polynomials $\widetilde{P}_{N_2}$ and $\widehat{P}_{N_2}$ are of degree
$N_2 = \mathcal{O}(\sqrt{N_1})$.
\end{theorem}

From Theorems \ref{mainthm} and \ref{mainthm2}, we see that the optimal choices of the parameter $\sigma$ can be determined readily and thus, the fastest convergence rate of the LP with
tapered lightning poles \eqref{eq:tapered2} or \eqref{eq:1taperedimag} is $\mathcal{O}(e^{-2\pi\sqrt{N\alpha}})$ while  $\mathcal{O}(e^{-\pi\sqrt{2N\alpha}})$ with uniform lightning poles \eqref{eq:uniform} or \eqref{eq:uniformgenimag_uniformal}, respectively.
In addition, the convergence rate is attainable for each  $\sigma>0$ too.

Figs. \ref{rates} - \ref{rates_uniform_imag}  illustrate  the best choices of $\sigma$ and that the convergence rates of the LPs in Theorems \ref{mainthm}  and \ref{mainthm2} are attainable for each  given $\sigma>0$, respectively.

\begin{figure}[htbp]
%\centerline{\includegraphics[height=6cm,width=15cm]{Fig_5.eps}}
\centerline{\includegraphics[height=6cm,width=16cm]{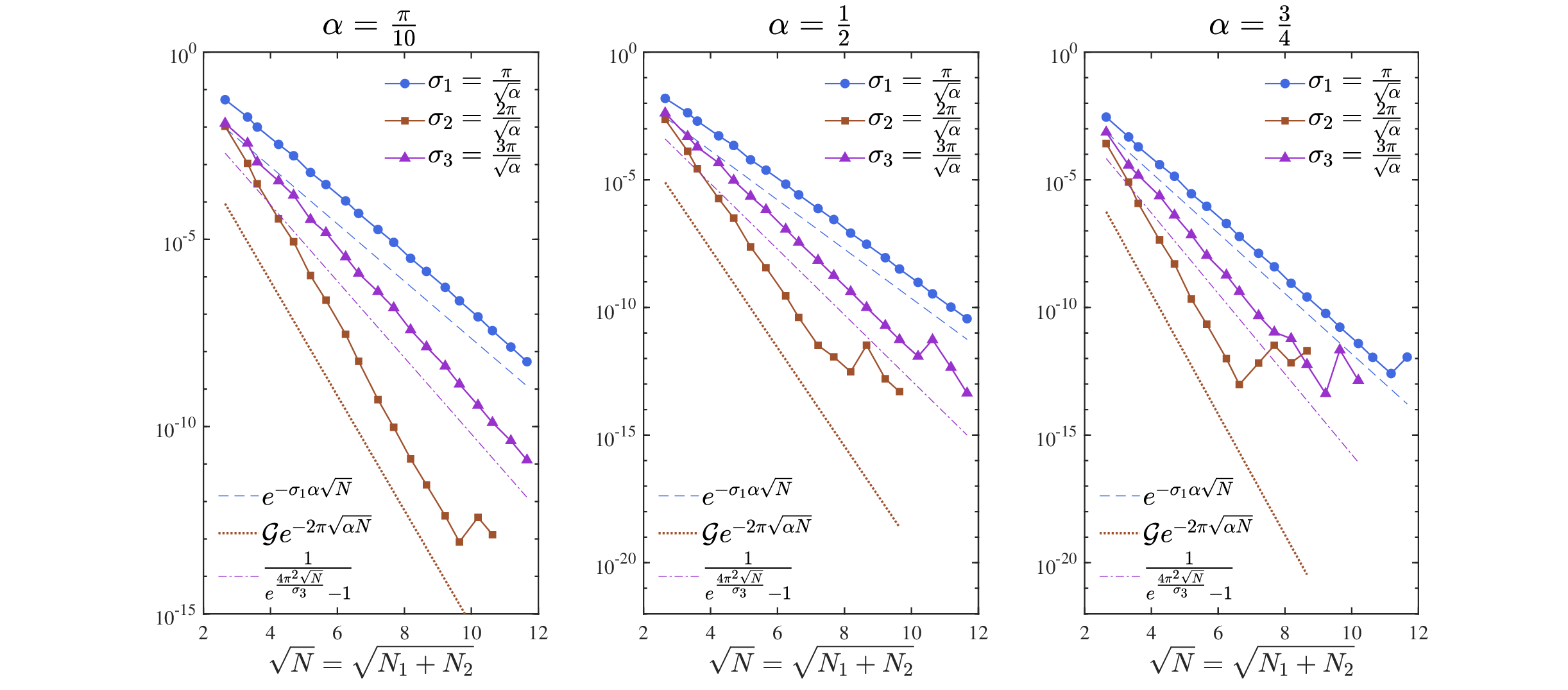}}
\caption{Convergence rates of $r_{N}(x)$ for $x^\alpha$ with various values of $\alpha$ and
$\sigma=\frac{\pi}{\sqrt{\alpha}}\ (h=\pi^2\alpha)$, $\sigma=\frac{2\pi}{\sqrt{\alpha}}\ (h=4\pi^2\alpha)$ and
$\sigma=\frac{3\pi}{\sqrt{\alpha}}\ (h=9\pi^2\alpha)$, where $\mathcal{G}=4^{1+\alpha}\sin(\alpha\pi)$ and $N=N_1+N_2$, $N_1=4:4:100$, $N_2={\rm ceil}(1.3\sqrt{N_1})$.}
\label{rates}
\end{figure}

\begin{figure}[htbp]
%\centerline{\includegraphics[height=6cm,width=15cm]{Fig_5_1.eps}}
\centerline{\includegraphics[height=6cm,width=16cm]{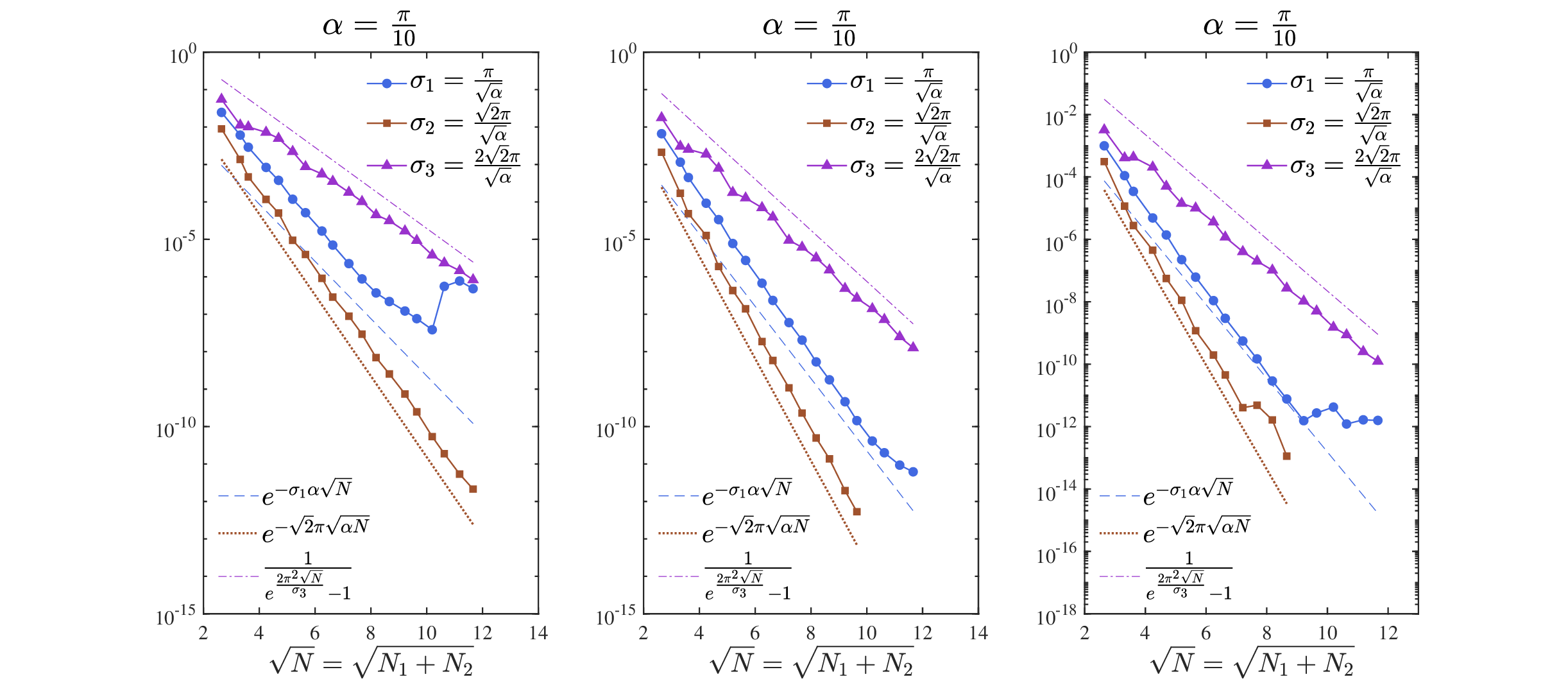}}
\caption{Convergence rates of  $\bar{r}_{N}(x)$  for $x^\alpha$ with various values of $\alpha$ and
$\sigma=\frac{\pi}{\sqrt{\alpha}}\ (\hbar=\pi\sqrt{\alpha}/\sqrt{N_1})$, $\sigma=\frac{\sqrt{2}\pi}{\sqrt{\alpha}}\ (\hbar=\sqrt{2\alpha}\pi/\sqrt{N_1})$ and
$\sigma=\frac{2\sqrt{2}\pi}{\sqrt{\alpha}}\ (\hbar=2\sqrt{2\alpha}\pi/\sqrt{N_1})$, where $N=N_1+N_2$ and $N_1=4:4:100$, $N_2={\rm ceil}(1.3\sqrt{N_1})$.}
\label{rates_uniform_real}
\end{figure}

\begin{figure}[htbp]
%\centerline{\includegraphics[height=6cm,width=15cm]{LPerroRat_t_imag}}
\centerline{\includegraphics[height=6cm,width=16cm]{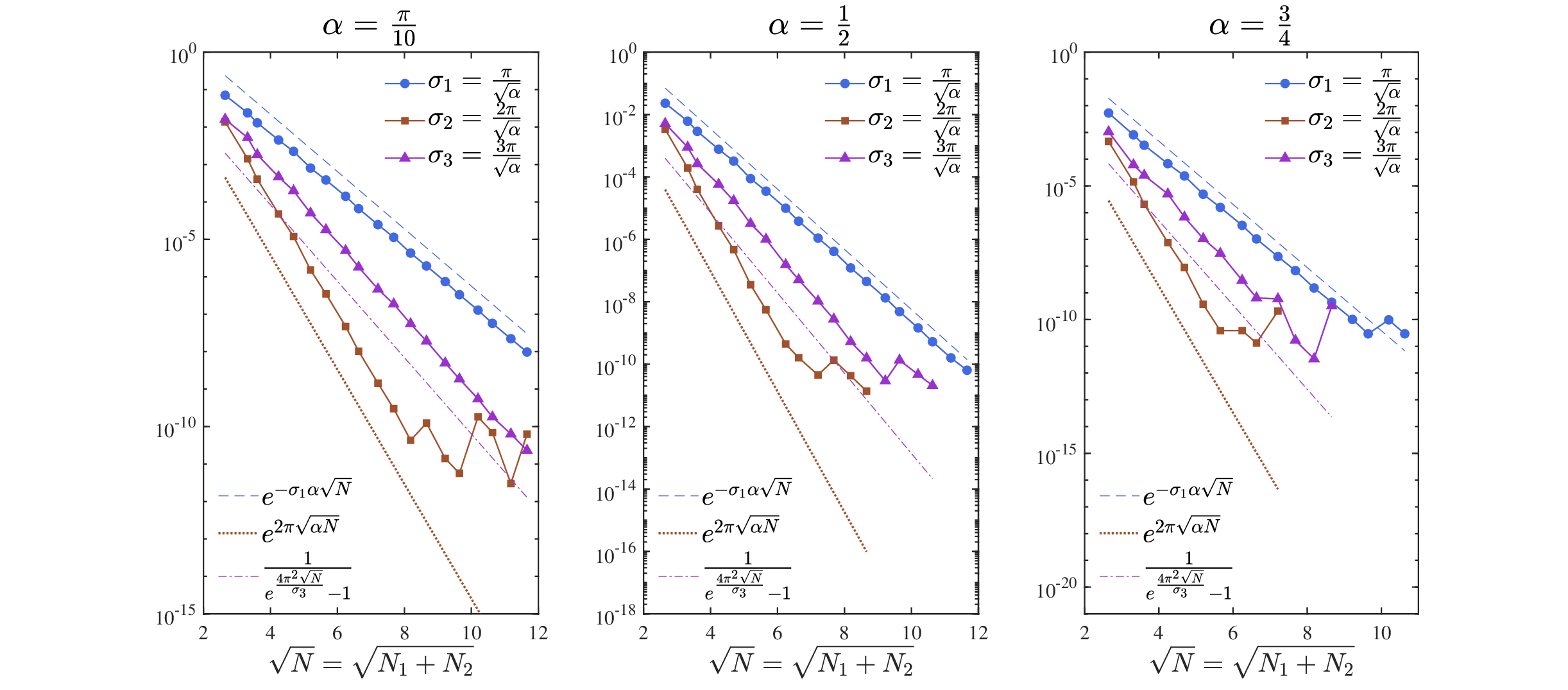}}
\caption{Convergence rates of  $\widetilde{r}_{N}(x)$  for $|x|^{2\alpha}$ with various values of $\alpha$ and
$\sigma=\frac{\pi}{\sqrt{\alpha}}\ (h=\pi^2\alpha)$, $\sigma=\frac{2\pi}{\sqrt{\alpha}}\ (h=4\pi^2\alpha)$ and
$\sigma=\frac{3\pi}{\sqrt{\alpha}}\ (h=9\pi^2\alpha)$, where $N=N_1+N_2$ and $N_1=4:4:100$, $N_2={\rm ceil}(1.3\sqrt{N_1})$.}
\label{rates_tapered_imag}
\end{figure}

\begin{figure}[htbp]
%\centerline{\includegraphics[height=6cm,width=15cm]{LPerroRat_u_imag}}
\centerline{\includegraphics[height=6cm,width=16cm]{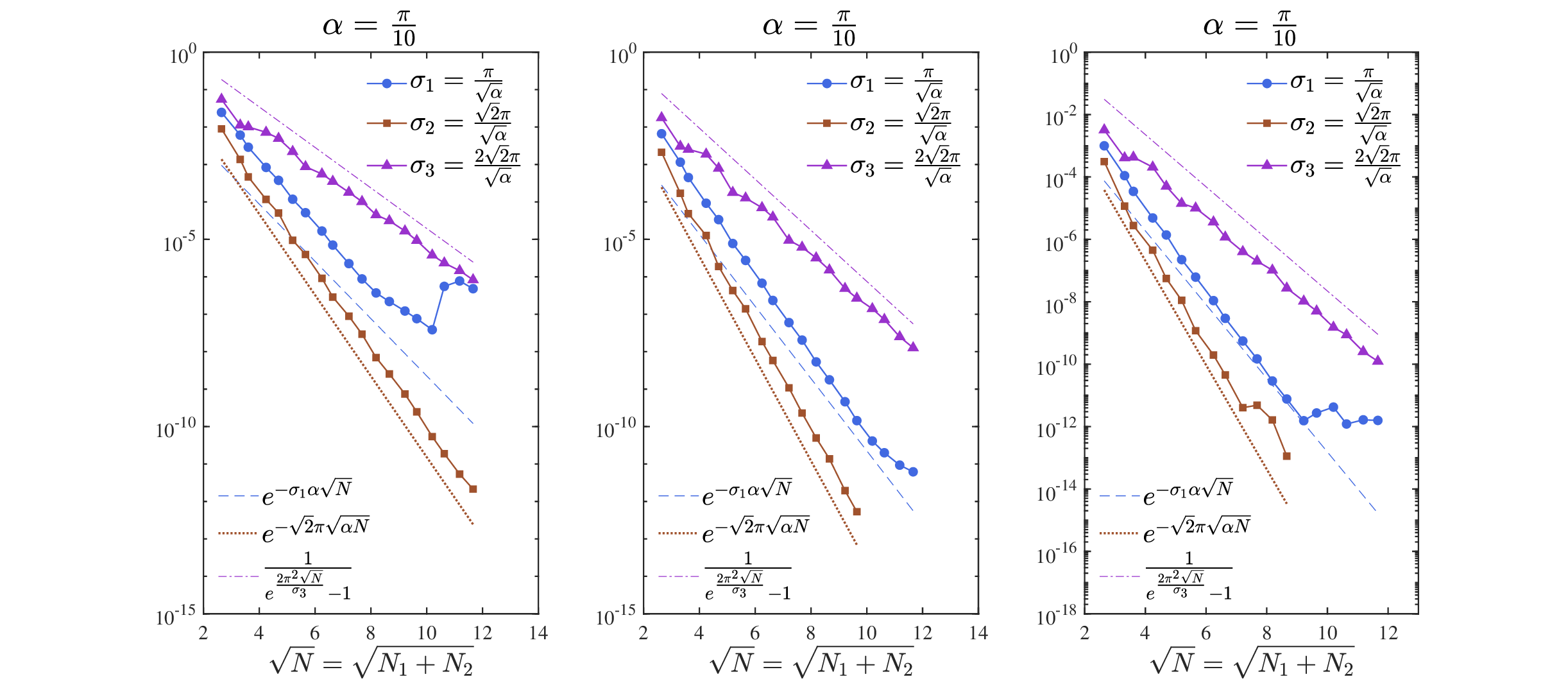}}
\caption{Convergence rates of  $\widehat{r}_{N}(x)$   for $|x|^{2\alpha}$ with various values of $\alpha$ and
$\sigma=\frac{\pi}{\sqrt{\alpha}}\ (\hbar=\pi\sqrt{\alpha}/\sqrt{N_1})$, $\sigma=\frac{\sqrt{2}\pi}{\sqrt{\alpha}}\ (\hbar=\sqrt{2\alpha}\pi/\sqrt{N_1})$ and
$\sigma=\frac{2\sqrt{2}\pi}{\sqrt{\alpha}}\ (\hbar=2\sqrt{2\alpha}\pi/\sqrt{N_1})$, where $N=N_1+N_2$ and $N_1=4:4:100$, $N_2={\rm ceil}(1.3\sqrt{N_1})$.}
\label{rates_uniform_imag}
\end{figure}

The crucial point for the analysis on the convergence rates is to utilize Poisson's summation formula (cf. \cite[(10.6-21)]{Henrici}, \cite{Stenger} and \cite{Trefethen2014SIREV}) to estimate the quadrature errors of  the composite rectangular rules for the integrals over the whole real line.

The Poisson summation formula is an important result relating Fourier transforms and Fourier series.
Based upon Poisson's formula, together with the decay behaviors of Fourier transforms and Paley-Wiener theorem \cite{Henrici,Stenger,PaleyWiener1934}, it follows that
the composite rectangular rule for  $w\in L^2(\mathbb{R})$ satisfies
\begin{align}\label{eq:trap0}
 \int_{\mathbb{R}}w(t)\mathrm{d}t-h\sum_{k=-\infty}^{+\infty}w(kh)=\mathcal{O}\left(e^{-\frac{2\pi d}{h}}\right)
\end{align}
if the Fourier transform $\mathfrak{F}[w]$ of $w$  is such that $|\mathfrak{F}[w](y)|= \mathcal{O}(e^{-d|y|})$ as $y\rightarrow \infty$ for some positive constant $d$ \cite[Theorem 1.3.2]{Stenger}.

In particular, the constant in the above $\mathcal{O}$ term \eqref{eq:trap0} can be specific suppose $w(z)$ is analytic in the strip $|\Im(z)| <d$ (or $\Im(z)>-d$, respectively) for some $d > 0$, and
$w(z)\rightarrow 0$ uniformly as $|z|\rightarrow +\infty$ in the strip (or in that half-plane, respectively), and for some $M$,
it satisfies
\begin{align*}%\label{eq:trap0}
 \int_{\mathbb{R}}|w(t+ib)|\mathrm{d}t\le M
\end{align*}
for all $b\in (-d,d)$ (or $b>-d$, respectively), then it holds
\begin{align}\label{eq:trap1}
\Big| \int_{\mathbb{R}}w(t)\mathrm{d}t-h\sum_{k=-\infty}^{+\infty}w(kh)\Big|\le \frac{2M}{e^{\frac{2\pi d}{h}}-1},
\end{align}
and the quantity $2M$ in the numerator is as small as possible \cite[Theorems 5.1-5.2]{Trefethen2014SIREV}, where $\Im(z)$ denotes the  imaginary part of $z$.

From \eqref{eq:trap0} and \eqref{eq:trap1} we see that the composite rectangular rule for the integral
over the whole real line can achieve an exponential rate as the stepsize $h\rightarrow 0$.

It is of particular interest that
both the composite rectangular rules for approximation of $x^\alpha$ involving the exponential clustering of poles \eqref{eq:tapered2} or \eqref{eq:uniform} can achieve exponential rates on $T=\sqrt{N_1h}=N_1\hbar$. One is with fixed step size $h=\sigma^2\alpha^2$ while the other is with $\hbar=\frac{\sigma\alpha}{\sqrt{N_1}}$.

%%%%%%%%%%%%%%%%%%%%%%%%%%%%%%%%%%%%%%%%%%%%%%%%%%%%%%%%%%%%%%%%%%%%%

The rest of this paper is organized as follows. In Section \ref{sec:2}, we provide preparatories on the integral formula of $x^\alpha$, the decay behaviors of  the truncated errors on the integral formula of $x^\alpha$,  and  the rational approximation forms, from which Theorem \ref{mainthm2} is proved if Theorem \ref{mainthm}  holds.
Poisson's summation formula  on the quadrature errors and results akin to Paley-Wiener Theorem are presented in Section \ref{sec:3}.
The complete proofs of Theorem \ref{mainthm} and Conjecture \ref{Conjecture 3.1} are presented in Section \ref{sec:4}. %Some remarks are concluded in Section \ref{conclusion}.
The detailed proofs of some useful lemmas in proving Theorem \ref{mainthm}  are presented in Appendix \ref{AppendixA}.

%================================================================================

\section{Preparatories}
\label{sec:2}

According to \cite[(3.222), p. 319]{GR2014}, $x^\alpha$ with $0<\alpha<1$ can be represented for $x\in [0,1]$  by
 \begin{align}\label{eq:int}
x^\alpha&=\frac{\sin(\alpha\pi)}{\alpha\pi}\int_0^{+\infty} \frac{x}{y^{\frac{1}{\alpha}}+x}\mathrm{d}y
=\frac{\sin(\alpha\pi)}{\alpha\pi}\int_{-\infty}^{+\infty} \frac{xe^t}{e^{\frac{1}{\alpha}t}+x}\mathrm{d}t\notag\\
&=\frac{\sin(\alpha\pi)}{\alpha\pi}\left\{\int_{-\infty}^{-T}+\int_{-T}^{\kappa T}+\int_{\kappa T}^{+\infty}\right\}
\frac{xe^t}{e^{\frac{1}{\alpha}t}+x}\mathrm{d}t
\end{align}
for $\kappa:=\frac{\alpha}{1-\alpha}$ and any positive real number $T$.

We first show  the two truncated errors in \eqref{eq:int} satisfy that
\begin{align*}%\label{eq:int3}
0\le \frac{\sin(\alpha\pi)}{\alpha\pi}\left\{\int_{-\infty}^{-T}+\int_{\kappa T}^{+\infty}\right\}
\frac{xe^t}{e^{\frac{1}{\alpha}t}+x}\mathrm{d}t
&\le\frac{e^{-T}\sin(\alpha\pi)}{\alpha\pi}+\frac{e^{-T}\sin((1-\alpha)\pi)}{(1-\alpha)\pi}\le 2e^{-T}
\end{align*}
which directly follows from the following inequalities for all $x\in [0,1]$
 \begin{align}\label{eq:inequ}
\frac{xe^t}{e^{\frac{1}{\alpha}t}+x}\le e^t,\quad \forall t\in (-\infty,0]  \quad\text{and}\quad \frac{xe^t}{e^{\frac{1}{\alpha}t}+x}\le e^{-\frac{1}{\kappa}t},\quad \forall t\in [0,+\infty).
 \end{align}
Together with  \eqref{eq:int} it leads to
 \begin{align}\label{eq:est}
x^\alpha
=&\frac{\sin(\alpha\pi)}{\alpha\pi}\int_{-T}^{\kappa T}
\frac{xe^t}{e^{\frac{1}{\alpha}t}+x}\mathrm{d}t+E^{(1)}_T(x),\quad \|E^{(1)}_T\|_{\infty}\le 2e^{-T}.
\end{align}

\subsection{LP schemes based on the tapered exponential clustering
poles \eqref{eq:tapered2} and \eqref{eq:1taperedimag}}
Along the way \cite{Herremans2023}, setting $t+T=\sqrt{u}$ for $t\in [-T,\kappa T]$, the integral in \eqref{eq:est} is transformed into
 \begin{align*}%\label{eq:int2}
 \frac{\sin(\alpha\pi)}{\alpha\pi}\int_{-T}^{\kappa T}\frac{xe^t}{e^{\frac{1}{\alpha}t}+x}\mathrm{d}t
 =\frac{\sin(\alpha\pi)}{\alpha\pi}\int_0^{(\kappa+1)^2T^2}
 \frac{1}{2\sqrt{u}}\frac{xe^{\sqrt{u}-T}}{e^{\frac{1}{\alpha}(\sqrt{u}-T)}+x}\mathrm{d}u.
 \end{align*}
Discretization using the rectangular rule in $N_t$ quadrature points $u = jh, \ 1 \le j \le N_t$ with step length $h=\sigma^2\alpha^2$:
\begin{align*}%\label{eq:sch}
T=\sqrt{N_1h}=\sigma\alpha\sqrt{N_1},\quad \mathcal{N}_th =(\kappa+1)^2T^2,\quad N_t={\rm ceil}(\mathcal{N}_t+1)
 \end{align*}
gives rise to the following rational approximation $r_{N_t}$ to $x^\alpha$
\begin{align}\label{eq:ECrat}
x^\alpha=&\frac{\sin(\alpha\pi)}{\alpha\pi}\int_0^{(\kappa+1)^2T^2}\frac{1}{2\sqrt{u}}
 \frac{xe^{\sqrt{u}-T}}{e^{\frac{1}{\alpha}(\sqrt{u}-T)}+x}\mathrm{d}u+E^{(1)}_T(x)\notag\\
 =&\frac{\sin(\alpha\pi)}{\alpha\pi}\int_0^{N_th}\frac{1}{2\sqrt{u}}
 \frac{xe^{\sqrt{u}-T}}{e^{\frac{1}{\alpha}(\sqrt{u}-T)}+x}\mathrm{d}u+E_T(x)\notag\\
 \approx & r_{N_t}(x) +E_T(x)
 \end{align}
with $\|E_T\|_{\infty}\le 3e^{-T}$, where  we used
\begin{align*} 0\le \frac{\sin(\alpha\pi)}{\alpha\pi}\int_{(\kappa+1)^2T^2}^{N_th}\frac{1}{2\sqrt{u}}
 \frac{xe^{\sqrt{u}-T}}{e^{\frac{1}{\alpha}(\sqrt{u}-T)}+x}\mathrm{d}u
 \le& \frac{\sin(\alpha\pi)}{\alpha\pi}\int_{\kappa T}^{+\infty}e^{-\frac{1}{\kappa}t}\mathrm{d}t\\
 =&\frac{\sin((1-\alpha)\pi)}{e^T(1-\alpha)\pi}\le e^{-T}
  \end{align*}
   by \eqref{eq:inequ}, and $r_{N_t}$ can be represented  as
\begin{align}\label{eq:ECrat1}
 r_{N_t}(x)=&\frac{\sin(\alpha\pi)}{\alpha\pi}h
 \sum_{j=1}^{N_t}\frac{1}{2\sqrt{jh}}
\frac{xe^{\sqrt{jh}-T}}{e^{\frac{1}{\alpha}(\sqrt{jh}-T)}+x}\hspace{.5cm}
\left(p_j=-e^{\frac{\sqrt{h}}{\alpha}\big(\sqrt{j}-\sqrt{N_1}\big)}\right)\notag\\
=&\frac{\sin(\alpha\pi)}{2\alpha\pi}\left[
\sum_{j=1}^{N_1}\sqrt{\frac{h}{j}}
\frac{(x-p_j+p_j)|p_j|^\alpha}{x-p_j}+\sum_{j=N_1+1}^{N_t}\sqrt{\frac{h}{j}}
\frac{x|p_j|^\alpha}{x-p_j}\right]\\
=&\frac{\sin(\alpha\pi)}{2\alpha\pi}
\sum_{j=1}^{N_1}\sqrt{\frac{h}{j}}\frac{p_j|p_j|^\alpha}{x-p_j}
+\frac{\sin(\alpha\pi)}{2\alpha\pi}\left(\sum_{j=1}^{N_1}\sqrt{\frac{h}{j}}|p_j|^\alpha+\sum_{j=N_1+1}^{N_t}\sqrt{\frac{h}{j}}\frac{x|p_j|^\alpha}{x-p_j}
\right)\notag\\
=&:r_{N_1}(x)+r_2(x).\notag
 \end{align}
In particular, from $p_j$ in \eqref{eq:ECrat1} and \eqref{eq:tapered2}, we see that  $C=1$.

Furthermore,
$r_2(x)=r_2((z+1)/2)$ ($z\in [-1,1]$) can be efficiently approximated by a polynomial $P_{N_2}(x)$ with $N_2=\mathcal{O}(\sqrt{N_1})$: Noting that $p_j\le -1$ for $j> N_1$, it is obvious that $r_2((z+1)/2)$ is analytic in disc ${\bf B}(0,2)$ with radius $2$ which contains a Bernstein ellipse $E_{\varrho_1}$\footnote{The Bernstein ellipse $E_{\varrho_1}$ with $\varrho_1=2+\sqrt{3}$ is a special ellipse with the foci at $\pm1$, whose major and minor
semiaxis lengths are $2$ and $\sqrt{3}$, respectively, summing to $\varrho_1$. See \cite{Bernstein1919,Trefethen2013} for details.} with $\varrho_1=2+\sqrt{3}$. In addition,
\begin{align*}
V:=\max_{z\in E_{\varrho_1}}|r_2((z+1)/2)|\le \max_{z\in {\bf B}(0,2)}|r_2((z+1)/2)|< 4
\end{align*}
which is established by
\begin{align*}
|r_2(z)|=&\frac{h\sin(\alpha\pi)}{\alpha\pi}\left|
\sum_{j=1}^{N_1}\frac{e^{\sqrt{jh}-T}}{2\sqrt{jh}}
+\sum_{j=N_1+1}^{N_t}\frac{1}{2\sqrt{jh}}\frac{\frac{1+re^{i\theta}}{2}e^{\sqrt{jh}-T}}{e^{\frac{1}{\alpha}(\sqrt{jh}-T)}+\frac{1+re^{i\theta}}{2}}
\right|\\
\le& \frac{\sin(\alpha\pi)}{\alpha\pi}\left(e^{-T}\int_0^{N_1h}e^{\sqrt{u}}\mathrm{d}\sqrt{u}
+\sum_{j=N_1+1}^{N_t}\frac{1}{2\sqrt{jh}}\frac{\frac{3}{2}h}{e^{\frac{1}{\kappa}(\sqrt{jh}-T)}-\frac{1}{2}e^{-(\sqrt{jh}-T)}} \right)\\
\le& 1+3 e^{\frac{1}{\kappa}T}\frac{\sin(\alpha\pi)}{\alpha\pi}\int_{N_1h}^{+\infty}e^{-\frac{1}{\kappa}\sqrt{u}}\mathrm{d}\sqrt{u}
=1+\frac{3\sin((1-\alpha)\pi)}{(1-\alpha)\pi}<4
\end{align*}
from the monotonicities of $\frac{e^{t}}{t}$ on $(0,1]$ and $[1,+\infty)$ respectively, and $\frac{e^{-\frac{1}{\kappa}t}}{t}$ on $(0,+\infty)$ together with inequality  $\frac{1}{e^{\frac{1}{\kappa}t}-\frac{1}{2}e^{-t}}\le  2e^{-\frac{1}{\kappa}t}$ for $t\ge 0$.
Then by Bernstein \cite{Bernstein1919} (also see \cite[Theorem 4.3]{Trefethen2008}), the truncated polynomial approximation $P_n$ of degree $n$ obtained as partial sums
of the Chebyshev series $r_2(x)=\sum_{k=0}^\infty b_kT_k(2x-1)$ satisfies
$$
\|r_2-P_n\|_{\infty}\le \frac{2V}{(\varrho_1-1)\varrho_1^n}.
$$
Analogous to \cite[Lemma 2.2]{Herremans2023}, it follows for $N_2\ge \frac{\sqrt{N_1h}}{\log\varrho_1}+2=\mathcal{O}(\sqrt{N_1})$ that
\begin{align}\label{polynomial app}
\|E_{PA}\|_{\infty}:=\|r_2-P_{N_2}\|_{\infty}\le \frac{2V}{(\varrho_1-1)\varrho_1^{N_2}}\le e^{-T}.
\end{align}
Consequently, it implies  that
\begin{align*}%\label{polynomial app1}
r_{N_t}(x)=r_{N_1}(x)+P_{N_2}(x)+E_{PA}(x).
\end{align*}

Therefore together with \eqref{eq:ECrat}, the estimate \eqref{eq: brateo} in Theorem \ref{mainthm} holds if
\begin{align*}
&\frac{\sin(\alpha\pi)}{\alpha\pi}\int_0^{N_th}\frac{1}{2\sqrt{u}}
 \frac{xe^{\sqrt{u}-T}}{e^{\frac{1}{\alpha}(\sqrt{u}-T)}+x}\mathrm{d}u-
 \frac{\sin(\alpha\pi)}{\alpha\pi}h
 \sum_{j=1}^{N_t}\frac{1}{2\sqrt{jh}}
\frac{xe^{\sqrt{jh}-T}}{e^{\frac{1}{\alpha}(\sqrt{jh}-T)}+x}\\
=&e^{\sigma \sqrt{2M_0}}\mathcal{O}(e^{-\sigma\alpha\sqrt{N}})+e^{\sigma \sqrt{2M_0}}\frac{\mathcal{O}(1)}{e^{\frac{4\pi^2}{\sigma}\sqrt{N}}-1}
\end{align*}
uniformly holds for $x\in [0,1]$ and the constants in the above $\mathcal{O}$ terms are independent of $N$, $x$, $\sigma$ and $\alpha$.

Moreover, the above analysis can be directly used to  construct the LP scheme based upon  \eqref{eq:1taperedimag} and prove  \eqref{eq: brateoimag} in Theorem \ref{mainthm2}.

{\bf Proof of \eqref{eq: brateoimag} in Theorem \ref{mainthm2}}: For $|x|^{2\alpha}$ and $x\in [-1,1]$, analogously by \eqref{eq:ECrat} we have
\begin{align*}%\label{eq:absint}
|x|^{2\alpha}&=\left(x^2\right)^{\alpha}=\frac{\sin(\alpha\pi)}{\alpha\pi}\int_0^{+\infty} \frac{x^2}{y^{\frac{1}{\alpha}}+x^2}\mathrm{d}y
=\frac{\sin(\alpha\pi)}{\alpha\pi}\int_{-\infty}^{+\infty} \frac{x^2e^t}{e^{\frac{1}{\alpha}t}+x^2}\mathrm{d}t\\
& \approx  r_{N_t}(x^2) +E_T(x^2) \notag
\end{align*}
and
\begin{align*}%\label{abspolynomial app1}
r_{N_t}(x^2)=r_{N_1}(x^2)+r_2(x^2):=\widetilde{r}_{N_1}(x)+\widetilde{P}_{N_2}(x)+\widetilde{E}_{PA}(x)
\end{align*}
with $\widetilde{P}_{N_2}(x)=P_{N_2}(x^2)$ of degree $\mathcal{O}(\sqrt{N_1})$ too and $\widetilde{E}_{PA}(x)=E_{PA}(x^2)$,
and
\begin{align*}%\label{abslightning}
\widetilde{r}_{N_1}(x)=r_{N_1}(x^2)
=&\frac{\sin(\alpha\pi)}{2\alpha\pi}
\sum_{j=1}^{N_1}\sqrt{\frac{h}{j}}\frac{p_j|p_j|^\alpha}{x^2-p_j}\notag\\
=&\frac{\sin(\alpha\pi)}{2\alpha\pi}
\sum_{j=1}^{N_1}\sqrt{\frac{h}{j}}\frac{p_j|p_j|^\alpha}{2\overline{p}_j}\left\{\frac{1}{x-\overline{p}_j}-\frac{1}{x+\overline{p}_j}\right\}.
\end{align*}
Therefore, if \eqref{eq: brateo} in Theorem \ref{mainthm} holds, it directly leads to \eqref{eq: brateoimag} in Theorem \ref{mainthm2}.

\begin{remark}\label{remark1}
From the proof of \eqref{eq: brateoimag}, we see that
 the LP \eqref{LPwithimaginarypoles}
with
tapered lightning poles \eqref{eq:1taperedimag}
satisfies for $x\in [0,1]$ that
\begin{equation*}%\label{eq: brateo_tapered_imag}
\Big|\widetilde r_N(x)-x^{2\alpha}\Big|=e^{\sigma \sqrt{2M_0}}\cdot
\left\{\begin{array}{ll}
\mathcal{O}(e^{-\sigma\alpha\sqrt{N}}),&\sigma\le \frac{2\pi}{\sqrt{\alpha}}\\
\frac{\mathcal{O}(1)}{e^{\frac{4\pi^2}{\sigma}\sqrt{N}}-1},&\sigma> \frac{2\pi}{\sqrt{\alpha}}
\end{array}\right.
\end{equation*}
too.
Thus, the LP \eqref{LPwithimaginarypoles} for $|x|^{2\alpha}$ on $[-1,1]$ or $x^{2\alpha}$ on $x\in [0,1]$ have the same convergence rate, quite different from the projection approximation by Jacobi orthogonal expansion, where the convergence for $x^{2\alpha}$ $(x\in [0,1])$ is two times faster than that for $|x|^{2\alpha}$ $(x\in [-1,1])$  \cite{XiangLiuNumerMath2020}.
\end{remark}

\subsection{LP schemes based on uniform exponential clustering
poles \eqref{eq:uniform} and \eqref{eq:uniformgenimag_uniformal}}
Following \eqref{eq:est} and similar to \eqref{eq:ECrat},    discretization using the rectangular rule
 in $\bar{N}_t+1$ quadrature points  with step length $\hbar=\frac{\sigma\alpha}{\sqrt{N_1}}$:
\begin{align*}%\label{eq:sch}
T=N_1\hbar=\sigma\alpha\sqrt{N_1},\quad \overline{\mathcal{N}}_t\hbar =(\kappa+1)T,\quad \bar{N}_t={\rm ceil}(\overline{\mathcal{N}}_t+1),
 \end{align*}
 gives rise to
\begin{align}\label{eq:estunif}
x^\alpha
=&\frac{\sin(\alpha\pi)}{\alpha\pi}\int_{-T}^{\kappa T}
\frac{xe^t}{e^{\frac{1}{\alpha}t}+x}\mathrm{d}t+E^{(1)}_T(x)
=\frac{\sin(\alpha\pi)}{\alpha\pi}\int_{0}^{(\kappa+1) T}
\frac{xe^{u-T}}{e^{\frac{1}{\alpha}(u-T)}+x}\mathrm{d}u+E^{(1)}_T(x)\notag\\
=& \frac{\sin(\alpha\pi)}{\alpha\pi}\int_{0}^{\bar{N}_t\hbar}
\frac{xe^{u-T}}{e^{\frac{1}{\alpha}(u-T)}+x}\mathrm{d}u+\overline{E}_T(x)
\approx\bar{r}_{\bar N_t}(x)+\overline{E}_T(x),
\end{align}
where
\begin{align}\label{eq:rat1uniform}
\bar r_{\bar{N}_t}(x)
=&\frac{\sin(\alpha\pi)}{\alpha\pi}\hbar\sum_{j=0}^{\bar{N}_t}\frac{xe^{jh-T}}{e^{\frac{1}{\alpha}(jh-T)}+x}\notag\\
=&\frac{\sin(\alpha\pi)}{\alpha\pi}
\hbar\sum_{j=0}^{N_1-1}\frac{q_j|q_j|^\alpha}{x-q_j}
+\frac{\sin(\alpha\pi)}{\alpha\pi}\left(\sum_{j=0}^{N_1-1}|q_j|^\alpha+\hbar
\sum_{k= N_1}^{\bar{N}_t}\frac{x|q_{N_1-k}|^\alpha}{x-q_{N_1-k}}
\right)\\
=&:\bar{r}_{N_1}(x)+\bar{r}_2(x)\quad  \left(q_j=:q_{N_1-k}=-e^{-\frac{N_1-k}{\alpha}\hbar}\right)\notag
 \end{align}
and
$\|\overline{E}_T\|_{\infty}\le 3e^{-T}$ since
\begin{align*}
0\le \frac{\sin(\alpha\pi)}{\alpha\pi}\int_{(\kappa+1)T}^{\bar{N}_t\hbar}
\frac{xe^{u-T}}{e^{\frac{1}{\alpha}(u-T)}+x}\mathrm{d}u
\le \frac{\sin(\alpha\pi)}{\alpha\pi}\int_{\kappa T}^{+\infty}e^{-\frac{1}{\kappa}t}\mathrm{d}t
=\frac{\sin((1-\alpha)\pi)}{e^T(1-\alpha)\pi}\le e^{-T}.
\end{align*}

Analogous to  \eqref{polynomial app}, $\bar{r}_2(x)$ can  also be approximated by a polynomial $\bar{P}_{N_2}(x)$ of degree $N_2=\mathcal{O}(\sqrt{N_1})$, with error $\|\overline{E}_{PA}\|_{\infty}\le e^{-T}$. Thus,
it implies that
\begin{align*}%\label{polynomial app1_uniform}
\bar{r}_{\bar{N}_t}(x)=\bar{r}_{N_1}(x)+\bar{P}_{N_2}(x)+\overline{E}_{PA}(x).
\end{align*}

Therefore together with \eqref{eq:rat1uniform}, the estimate \eqref{eq: brateounif} in Theorem \ref{mainthm} also holds if
\begin{align*}
&\frac{\sin(\alpha\pi)}{\alpha\pi}\int_0^{\bar{N}_th}
 \frac{xe^{u-T}}{e^{\frac{1}{\alpha}(u-T)}+x}\mathrm{d}u-\bar{r}_{\bar{N}_t}(x)
=\mathcal{O}(e^{-\sigma\alpha\sqrt{N}})
+\frac{\mathcal{O}(1)}{e^{\frac{2\pi^2}{\sigma}\sqrt{N}}-1}
\end{align*}
uniformly holds for $x\in [0,1]$ and the constants in the above $\mathcal{O}$ terms independent of $N$, $x$, $\sigma$ and $\alpha$.

Similarly, the analysis above can be directly used to  construct the LP scheme based upon   \eqref{eq:uniformgenimag_uniformal} and prove  \eqref{eq: brateo_uniform_imag} in Theorem \ref{mainthm2}.

{\bf Proof of \eqref{eq: brateo_uniform_imag} in Theorem \ref{mainthm2}}: Analogously for $|x|^{2\alpha}$ with $x\in [-1,1]$ we have
\begin{align}\label{eq:absint2}
|x|^{2\alpha}
 \approx  \bar{r}_{\bar{N}_t}(x^2) +\overline{E}_T(x^2)
 \end{align}
and
\begin{align*}%\label{abspolynomial app1}
\bar{r}_{\bar{N}_t}(x^2)=\bar{r}_{N_1}(x^2)+\bar{r}_2(x^2):=\widehat{r}_{N_1}(x)+\widehat{P}_{N_2}(x)+\widehat{E}_{PA}(x)
\end{align*}
with $\widehat{P}_{N_2}(x)=\bar P_{N_2}(x^2)$ of degree $\mathcal{O}(\sqrt{N_1})$ too and $\widehat{E}_{PA}(x)=\overline{E}_{PA}(x^2)$,
and
\begin{align*}%\label{abslightning}
\widehat{r}_{N_1}(x)=\overline{r}_{N_1}(x^2)
=&\frac{\sin(\alpha\pi)}{2\alpha\pi}
\sum_{j=1}^{N_1}\sqrt{\frac{\hbar}{j}}\frac{q_j|q_j|^\alpha}{x^2-q_j}\notag\\
=&\frac{\sin(\alpha\pi)}{2\alpha\pi}
\sum_{j=1}^{N_1}\sqrt{\frac{\hbar}{j}}\frac{q_j|q_j|^\alpha}{2\overline{q}_j}\left\{\frac{1}{x-\overline{q}_j}-\frac{1}{x+\overline{q}_j}\right\}.
\end{align*}
Thus, if \eqref{eq: brateounif} in Theorem \ref{mainthm} holds, it directly implies \eqref{eq: brateo_uniform_imag} in Theorem \ref{mainthm2}.

\begin{remark}\label{remark2}
From \eqref{eq:absint2} and Theorem \ref{mainthm2}, similar to Remark \ref{remark1}
it also follows for $x\in [0,1]$
\begin{equation*}%\label{eq: brateo_uniform_imag_remark}
|\widehat{r}_N(x)-x^{2\alpha}|=\left\{\begin{array}{ll}
\mathcal{O}(e^{-\sigma\alpha\sqrt{N}}),&\sigma\le \frac{\sqrt{2}\pi}{\sqrt{\alpha}},\\
\frac{\mathcal{O}(1)}{e^{\frac{2\pi^2}{\sigma}\sqrt{N}}-1},&\sigma> \frac{\sqrt{2}\pi}{\sqrt{\alpha}}.
\end{array}\right.
\end{equation*}
\end{remark}

For the sake of narration, we introduce the following notations:
\begin{align}
&f(u,x)=\frac{\sin(\alpha\pi)}{\alpha\pi}
\frac{1}{2\sqrt{u}}\frac{xe^{\sqrt{u}-T}}
{e^{\frac{1}{\alpha}(\sqrt{u}-T)}+x},\quad \quad\quad\,\,\,\,\,\,\overline{f}(u,x)=\frac{\sin(\alpha\pi)}{\alpha\pi}
\frac{xe^{u-T}}
{e^{\frac{1}{\alpha}(u-T)}+x},\label{eq:fun}\\
&I(x)=\frac{\sin(\alpha\pi)}{\alpha\pi}\int_0^{N_th}\frac{1}{2\sqrt{u}}
\frac{xe^{\sqrt{u}-T}}{e^{\frac{1}{\alpha}(\sqrt{u}-T)}+x}\mathrm{d}u,\quad \overline{I}(x)=\frac{\sin(\alpha\pi)}{\alpha\pi}\int_0^{\bar{N}_t\hbar}
\frac{xe^{u-T}}{e^{\frac{1}{\alpha}(u-T)}+x}\mathrm{d}u,\label{eq:quadrature}\\
&S(x)=h\sum_{j=1}^{N_t}f(jh,x)(=r_{N_t}(x)),\quad\quad\quad\quad\quad\quad\,\, \overline{S}(x)=\hbar\sum_{j=1}^{\bar{N}_t}\overline{f}(j\hbar,x)(=\bar{r}_{\bar{N}_t}(x)).\label{trapzoid_S(x)}
\end{align}

In the following, we shall focus on the proof of Theorem \ref{mainthm} mainly based on the following estimates of  quadrature errors
\begin{align}
I(x)-S(x)=&e^{\sigma \sqrt{2M_0}}\mathcal{O}(e^{-\sigma\alpha\sqrt{N_1}})+e^{\sigma \sqrt{2M_0}}\frac{\mathcal{O}(1)}{e^{\frac{4\pi^2}{\sigma}\sqrt{N_1}}-1},\label{eq:errs}\\
\overline{I}(x)-\overline{S}(x)=&\mathcal{O}(e^{-\sigma\alpha\sqrt{N_1}})
+\frac{\mathcal{O}(1)}{e^{\frac{2\pi^2}{\sigma}\sqrt{N_1}}-1},\label{eq:errs1}
\end{align}
where $N_1$ in \eqref{eq:errs} and \eqref{eq:errs1} can be replaced by $N$ due to $N_2=\mathcal{O}(\sqrt{N_1})$. See Figs. \ref{quadratureerr_tapered} and \ref{quadratureerr_uni}, and refer to Section \ref{sec:4} for more details.

\begin{figure}[htbp]
\centerline{\includegraphics[height=6cm,width=16cm]{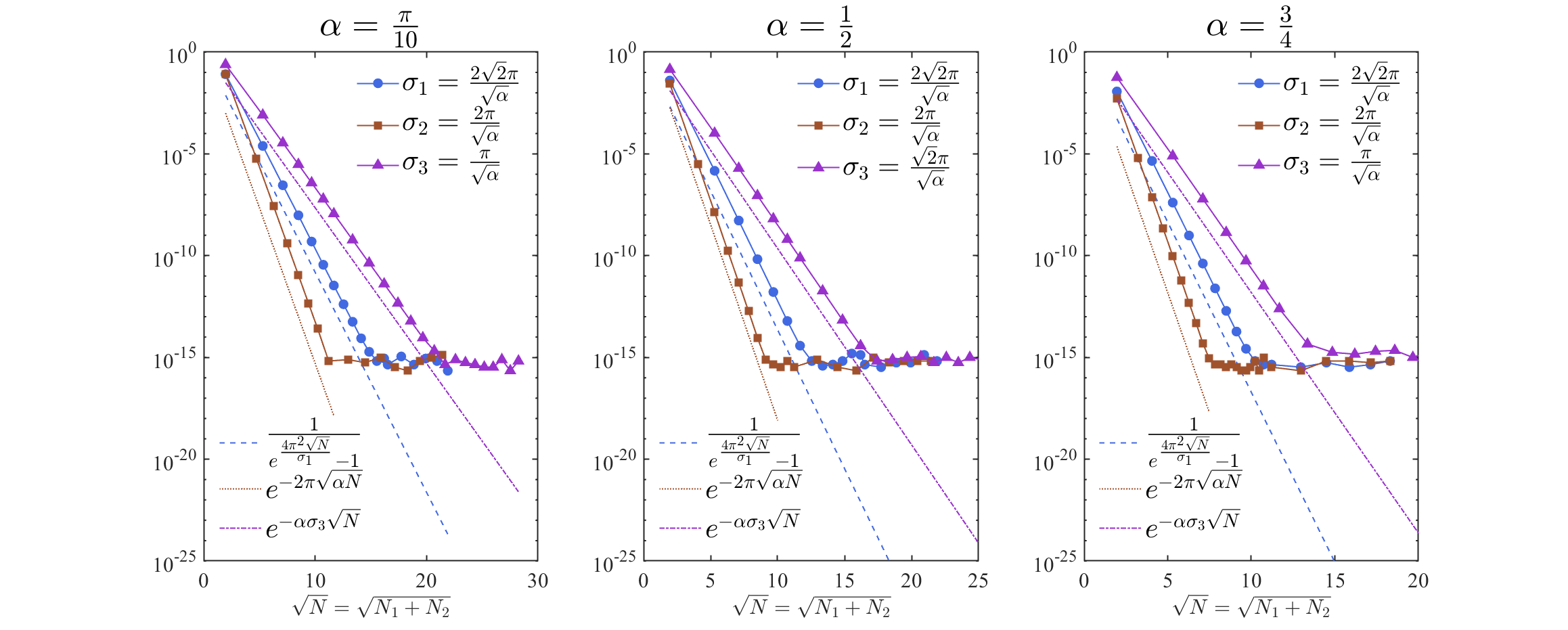}}
\caption{Quadrature errors $\|I(x)-S(x)\|_{\infty}$ of the rectangular rules for $I(x)$ with various values of $\alpha$ and $h=8\pi^2\alpha$ $(\sigma_1=\frac{2\sqrt{2}\pi}{\sqrt{\alpha}})$, $h=4\pi^2\alpha$ $(\sigma_2=\frac{4\pi}{\sqrt{\alpha}})$, $h=\pi^2\alpha$ $(\sigma_3=\frac{\pi}{\sqrt{\alpha}})$, where $N=N_1+N_2$ and $N_1=10:10:500$, $N_2={\rm ceil}(1.3\sqrt{N_1})$.}
\label{quadratureerr_tapered}
\end{figure}

\begin{figure}[htbp]
\centerline{\includegraphics[height=6cm,width=16cm]{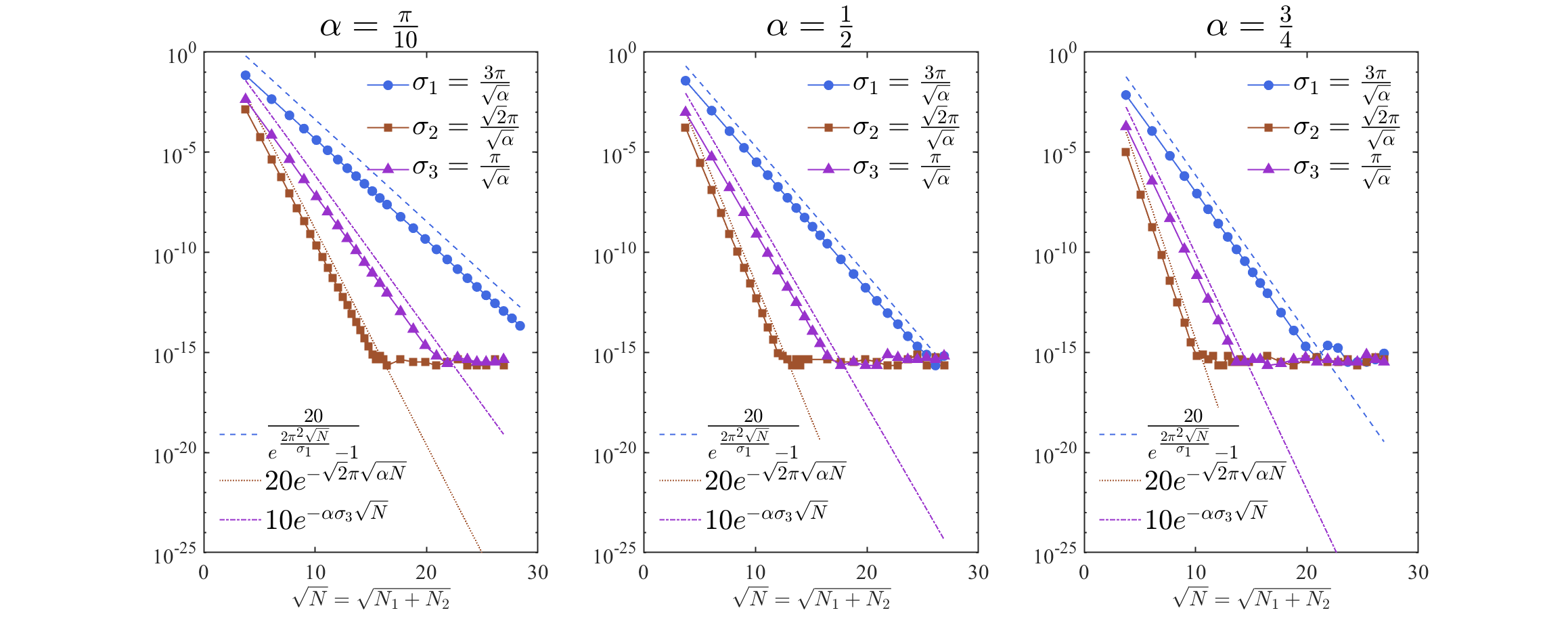}}
\caption{Quadrature errors $\|\overline{I}(x)-\overline{S}(x)\|_{\infty}$ of the rectangular rules for $\overline{I}(x)$ with various values of $\alpha$ and $\hbar=\frac{3\sqrt{\alpha}\pi}{\sqrt{N_1}}$ $(\sigma_1=\frac{3\pi}{\sqrt{\alpha}})$, $\hbar=\frac{\sqrt{2\alpha}\pi}{\sqrt{N_1}}$ $(\sigma_2=\frac{\sqrt{2}\pi}{\sqrt{\alpha}})$, $\hbar=\frac{\sqrt{\alpha}\pi}{\sqrt{N_1}}$ $(\sigma_3=\frac{\pi}{\sqrt{\alpha}})$, where $N=N_1+N_2$ and $N_1=10:10:500$, $N_2={\rm ceil}(1.3\sqrt{N_1})$.}
\label{quadratureerr_uni}
\end{figure}

\section{Possion's summation formula and quadrature errors}\label{sec:3}
In order to get the exponential convergence rates of the quadrature errors \eqref{eq:errs} and \eqref{eq:errs1}, along the way on
rectangular rule for integrals over the real line \cite[Sect. 5]{Trefethen2014SIREV},
it is necessary to introduce Poisson's summation formula (cf. \cite[(10.6-21)]{Henrici} and \cite[Theorem 1.3.1]{Stenger}).

\begin{theorem}\cite[Theorem 1.3.1]{Stenger}\label{StengerPossionFormula}
Let $w\in L^2(\mathbb{R})$ and let $w$ and its Fourier transform $\mathfrak{F}[w]$ for $t$ and $u$ in $\mathbb{R}$, satisfy the
conditions
$$
w(u)=\lim_{t\rightarrow 0^+}\frac{w(u-t)+w(u+t)}{2},\quad \mathfrak{F}[w](u)=\lim_{t\rightarrow 0^+}\frac{\mathfrak{F}[w](u-t)+\mathfrak{F}[w](u+t)}{2}.
$$
Then, for all $h > 0$,
\begin{align}\label{Possionsummationformula}
h\sum_{n=-\infty}^{+\infty}w(nh)e^{inh\mu}= \mathfrak{F}[w]\left(\frac{2n\pi}{h}+\mu\right).
\end{align}
\end{theorem}

From \eqref{Possionsummationformula} with $\mu=0$ and by $\mathfrak{F}[w]\left(0\right)=\int_{-\infty}^{+\infty}w(u)\mathrm{d}u$, it follows
\begin{align}\label{QuadratureErrorfor_w}
E^{w}_{Q}:=\int_{-\infty}^{+\infty}w(u)\mathrm{d}u
-h\sum_{j=-\infty}^{+\infty}w(jh)
=-\sum_{n\not=0}\mathfrak{F}[w]\left(\frac{2n\pi}{h}\right).
\end{align}

To establish \eqref{eq:errs} and
\eqref{eq:errs1},
in the following we are concerned with the decay asymptotics of
\begin{align}\label{eq:ffft}
\mathfrak{F}[f]\left(\frac{2n\pi}{h}\right)=\int_{-\infty}^{+\infty}f(u,x)e^{-\frac{2n\pi i u}{h}} \mathrm{d}u, \quad
\mathfrak{F}[\bar{f}]\left(\frac{2n\pi}{h}\right)=\int_{-\infty}^{+\infty}\bar {f}(u,x)e^{-\frac{2n\pi i u}{h}} \mathrm{d}u.
\end{align}
However, it is worth noticing  that  the integrand $f(u,x)=\frac{1}{2\sqrt{u}}\frac{xe^{\sqrt{u}-T}}{e^{\frac{1}{\alpha}(\sqrt{u}-T)}+x} \not\in L^2(\mathbb{R_+})$ for arbitrary fixed $x\in (0,1]$,    then we define
\begin{align}\label{eq:extension_of_f}
\hat f(u,x)=
\begin{cases}
f(u,x), & \Re(u)\ge h,\\
f(h+i\Im(u),x), & -h\le \Re(u)\le h,\\
f(-u,x), & \Re(u)\le-h,
\end{cases}
\end{align}
instead of $f$ in \eqref{eq:ffft}.
For the special case $x=0$, it is obvious that $I(0)-S(0)=0$ and $\overline{I}(0)-\overline{S}(0)=0$.

\bigskip
For readability, we firstly prove \eqref{eq:errs1}  then \eqref{eq:errs} since the proof of \eqref{eq:errs} is much more  complicated than \eqref{eq:errs1}.

\subsection{Quadrature errors with uniform exponentially clustered poles 
\eqref{eq:uniform}}\label{subsec:3.01}

\begin{theorem}\label{Quadratrue_rat_uniform}
Let $\overline{f}(u,x)$ be defined in \eqref{eq:fun} with $u\in\mathbb{R}$ and $x\in(0,1]$.
Then the summation of discrete Fourier transform decays at an exponential rate
\begin{align}\label{eq:conclusionOfFouriersum_uniform}
\sum_{n\ne0}\mathfrak{F}[\overline{f}]\big{(}\frac{2 n\pi }{\hbar}\big{)}
= \frac{\mathcal{O}(1)}{e^{\frac{2\pi^2\alpha}{\hbar}}-1},
\end{align}
and the constant in the $\mathcal{O}$ term \eqref{eq:conclusionOfFouriersum_uniform} is independent of $n$, $\hbar$, $x$, $\alpha$ and $\sigma$.
\end{theorem}
\begin{proof}
From the definition \eqref{eq:fun} and inequalities in \eqref{eq:inequ}, it is easy to verify that
$\overline{f}\in L^2(\mathbb{R})\bigcap C(\mathbb{R})$.
Moreover,
we check readily that
\begin{align*}
\int_{-\infty}^{+\infty}|\overline{f}(u,x)|\mathrm{d}u
=& \frac{\sin(\alpha\pi)}{\alpha\pi}\left\{\int_{-\infty}^0 +\int_{0}^{+\infty}\right\}\frac{xe^t}{e^{\frac{1}{\alpha}t}+x}\mathrm{d}t\\
\le&\frac{\sin(\alpha\pi)}{\alpha\pi}\left\{\int_{-\infty}^0e^t\mathrm{d}t+\int_{0}^{+\infty}e^{-\frac{1}{\kappa}t}\mathrm{d}t\right\}<2
\end{align*}
holds uniformly for all $x\in(0,1]$ and $0<\alpha<1$, thus the Fourier transform
$$\mathfrak{F}\left[\overline{f}(u,x)\right]=\int_{-\infty}^{+\infty}\overline{f}(u,x)e^{-i\xi u}\mathrm{d}u$$
exists and is continuous for all real $\xi$ \cite[(10.6-12)-(10.6-13)]{Henrici}.
Then $\overline{f}(u,x)$ and $\mathfrak{F}\left[\overline{f}(u,x)\right]$ satisfy the conditions of Theorem \ref{StengerPossionFormula}.

Furthermore, it is obvious that $\overline{f}(u,x)=\frac{\sin(\alpha\pi)}{\alpha\pi}\frac{xe^{u-T}}{e^{\frac{u-T}{\alpha}}+x}$ takes the simple poles
$$\dot{u}_k=T+\alpha\log{x}+i(2k-1)\alpha\pi,\ k=0,\pm1,\pm2,\cdots,$$
among which the closest poles to the real axis are
\begin{align}\label{Dot_u0_u1}
\dot{u}_0=T+\alpha\log{x}-i\alpha\pi \text{ and } \dot{u}_1=T+\alpha\log{x}+i\alpha\pi.
\end{align}
Thus $\overline{f}(u,x)$ is holomorphic in the strip domain
$$\left\{u\in\mathbb{C}:\ |\Im(u)|<\alpha\pi\right\}$$
with the simple poles $\dot{u}_0$ and $\dot{u}_1$ located on the lower and upper boundaries.
In addition, the second nearest poles of $\overline{f}(u,x)$ to the real line are $\dot{u}_{-1}$ and $\dot{u}_2$, with the same distance $3\alpha\pi$, which implies that $\overline{f}(u,x)$ is holomorphic in the strip domain
$\{u\in\mathbb{C}:\ |\Im(u)|\le\frac{3}{2}\alpha\pi\}$ except for the simple poles $\dot{u}_0$ and $\dot{u}_1$
\eqref{Dot_u0_u1} lying interior.

In particular, notice that \begin{align*}
\lim_{A\rightarrow+\infty}\left(
\int_{-A}^{-A-\frac{3}{2}i\alpha\pi}+\int_{A}^{A-\frac{3}{2}i\alpha\pi}\right)
\overline{f}(u,x)e^{-i\frac{2n\pi}{\hbar}u}\mathrm{d}u
=0
\end{align*}
since
\begin{align*}
\left|\frac{xe^{-A-T-it}}{e^{\frac{1}{\alpha}(-A-T-it)}+x}\right|
\le\frac{xe^{-A-T}}{x-e^{\frac{1}{\alpha}(-A-T)}}
\le\frac{xe^{-A-T}}{x/2}=2e^{-A-T}
\end{align*}
holds for sufficiently large $A$ and arbitrary $x>0$, and
\begin{align*}
\left|\frac{xe^{A-T-it}}{e^{\frac{1}{\alpha}(A-T-it)}+x}\right|
\le\frac{1}{e^{\frac{1}{\kappa}(A-T)}-e^{-(A-T)}}
\le\frac{1}{e^{\frac{1}{\kappa}(A-T)}-1}
\end{align*}
holds for $A>T$, which implies that
\begin{align*}
&\left|\left(
\int_{-A}^{-A-\frac{3}{2}i\alpha\pi}+\int_{A}^{A-\frac{3}{2}i\alpha\pi}\right)
\overline{f}(u,x)e^{-i\frac{2n\pi}{\hbar}u}\mathrm{d}u\right|\\
\le&\left[2e^{-A-T}+\frac{1}{e^{\frac{1}{\kappa}(A-T)}-1}\right]\int_{0}^{\frac{3}{2}\alpha\pi}e^{-\frac{2n\pi}{\hbar}t}\mathrm{d}t.
\end{align*}

\begin{figure}[htp]
\centerline{\includegraphics[width=12cm]{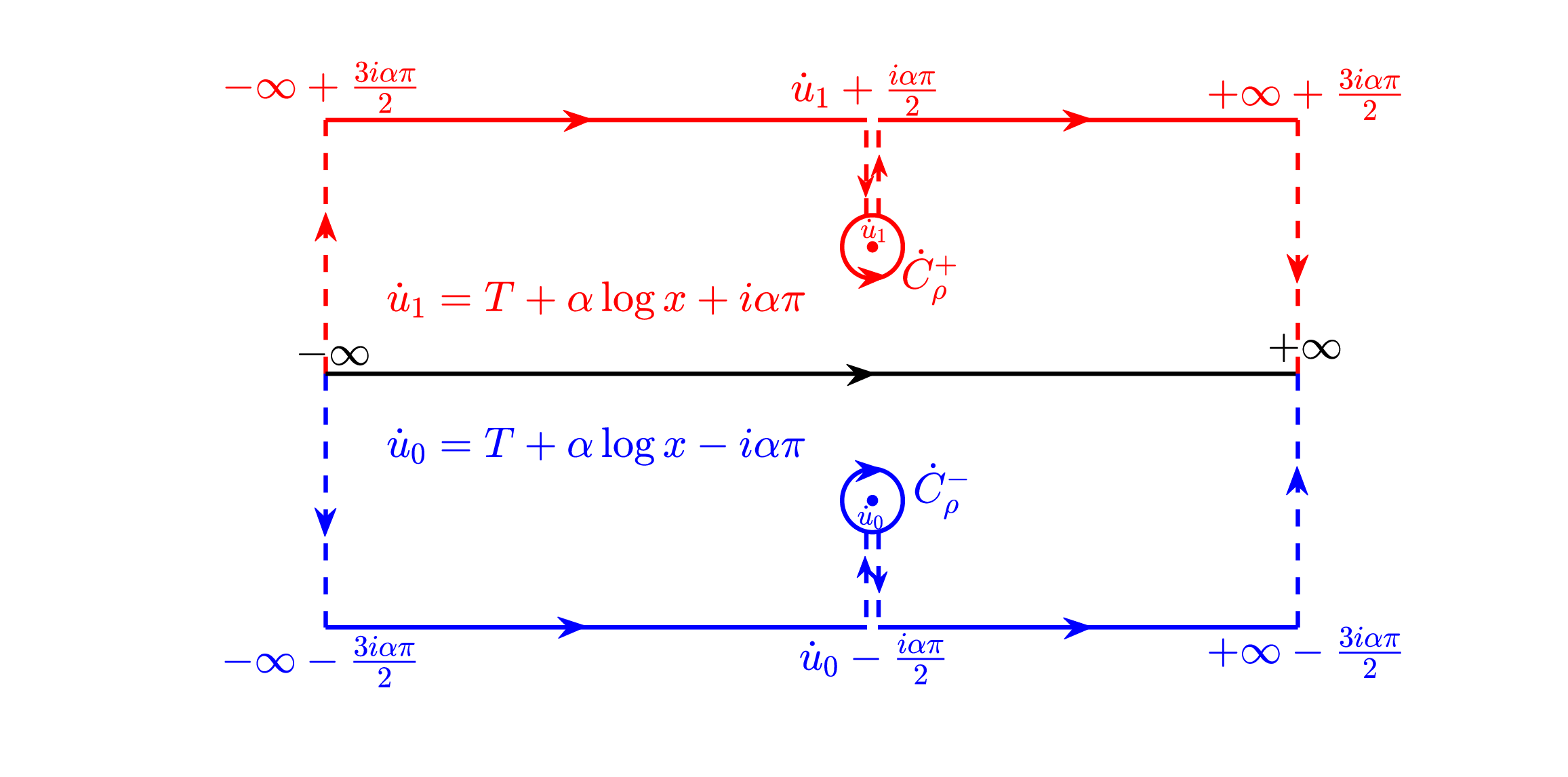}}
\caption{The contours for integrals in \eqref{ConvertIntegralToBrokenLine_Uniform}. The poles nearest to the real line of $\overline{f}(u,x)$ are $\dot{u}_0$ and $\dot{u}_1$.}
\label{integral_contour_uniform}
\end{figure}

Let $\dot{C}^-_{\rho}$  denote the circle $\{u=\dot{u}_0+\rho e^{i\theta}:\ 0\le\theta<2\pi\}$\footnote{The vary range of $\vartheta$ for $\dot{C}_{\rho}^{\pm}$ in fact is $\pm\frac{\pi}{2}\rightarrow\pm\frac{5\pi}{2}$ , it does not matter to write it as $0\rightarrow\pm2\pi$ since the invariance of the integral $\int_{\dot{C}_{\rho}^{\pm}}\overline{f}(u,z)e^{\pm i\frac{2n\pi}{\hbar}u}\mathrm{d}u$.}, where
$\rho=\min\{\frac{\alpha\pi}{2},\frac{1}{N_0}\}$ with $N_0$ being some sufficiently large integer  (see Fig. \ref{integral_contour_uniform}). Then with the aid of the Cauchy's integral Theorem and the integrals along the vertical lines $\dot{u}_0-\frac{1}{2}i\alpha\pi\rightleftharpoons \dot{C}_{\rho}^-$ (not include that on the circle $\dot{C}_{\rho}^-$) cancelling each other out,  the $n$-th $(n\ge1)$ discrete Fourier transform that can be evaluated by
\begin{align}
  \mathfrak{F}\left(\frac{2n\pi}{\hbar}\right)
  =&\int_{-\infty}^{+\infty}\overline{f}(u,x)e^{-i\frac{2n\pi}{\hbar}u}\mathrm{d}u\notag\\
  =&\int_{-\infty-\frac{3}{2}i\alpha\pi}^{+\infty-\frac{3}{2}i\alpha\pi}
  \overline{f}(u,x)e^{-i\frac{2n\pi}{\hbar}u}\mathrm{d}u
  +\int_{\dot{C}^-_{\rho}}\overline{f}(u,x)e^{-i\frac{2n\pi}{\hbar}u}\mathrm{d}u,\label{ConvertIntegralToBrokenLine_Uniform}
\end{align}
where we used the Cauchy's integral theorem to convert the integral on the real axis to the broken line (see Fig. \ref{integral_contour_uniform})
$$
-\infty\rightarrow-\infty-\frac{3}{2}i\alpha\pi\rightarrow \dot{u}_0-\frac{1}{2}i\alpha\pi\rightarrow\dot{C}^-_{\rho}
\rightarrow\dot{u}_0-\frac{1}{2}i\alpha\pi\rightarrow+\infty-\frac{3}{2}i\alpha\pi\rightarrow+\infty.
$$

In addition, the first integral on RHS of \eqref{ConvertIntegralToBrokenLine_Uniform} can be bounded by
\begin{align*}
  &\left|\int_{-\infty-\frac{3}{2}i\alpha\pi}^{+\infty-\frac{3}{2}i\alpha\pi}
  \overline{f}(u,x)e^{-i\frac{2n\pi}{\hbar}u}\mathrm{d}u\right|
  \le e^{-\frac{3n\alpha\pi^2}{\hbar}}\int_{-\infty}^{+\infty}
  \left|\overline{f}(t-\frac{3}{2}i\alpha\pi,x)\right|\mathrm{d}u\\
  =&\frac{e^{-\frac{3n\alpha\pi^2}{\hbar}}\sin(\alpha\pi)}{\alpha\pi}\int_{-\infty}^{+\infty}
  \left|\frac{xe^{t-\frac{3}{2}i\alpha\pi-T}}{e^{\frac{1}{\alpha}(t-\frac{3}{2}i\alpha\pi-T)}+x}\right|\mathrm{d}u\notag\\
  =&\frac{e^{-\frac{3n\alpha\pi^2}{\hbar}}\sin(\alpha\pi)}{\alpha\pi}\int_{-\infty}^{+\infty}
  \left|\frac{xe^{t-T}}{ie^{\frac{1}{\alpha}(t-T)}+x}\right|\mathrm{d}u\\
  \le&\frac{e^{-\frac{3n\alpha\pi^2}{\hbar}}\sin(\alpha\pi)}{\alpha\pi}
  \left(\int_{-\infty}^{T}e^{t-T}\mathrm{d}t+\int_{T}^{+\infty}e^{\frac{\alpha-1}{\alpha}(t-T)}\mathrm{d}t\right)\notag\\
  =&\frac{e^{-\frac{3n\alpha\pi^2}{\hbar}}\sin(\alpha\pi)}{\alpha\pi}\left(1+\frac{\alpha}{1-\alpha}\right)\\
  =&\frac{e^{-\frac{3n\alpha\pi^2}{\hbar}}\sin(\alpha\pi)}{(1-\alpha)\pi}
  =e^{-\frac{3n\alpha\pi^2}{\hbar}}\mathcal{O}(1)
\end{align*}
and the second integral on RHS of \eqref{ConvertIntegralToBrokenLine_Uniform} by
\begin{align*}
  &\left|\int_{\dot{C}^-_{\rho}}\overline{f}(u,x)e^{-i\frac{2n\pi}{\hbar}u}\mathrm{d}u\right|
  =\frac{\sin(\alpha\pi)}{\alpha\pi}\left|\int_{\dot{C}^-_{\rho}}
\frac{x^{\alpha}e^{u-\dot{u}_0}}{e^{\frac{1}{\alpha}(u-\dot{u}_0)}-1}e^{-i\frac{2n\pi}{\hbar}u}\mathrm{d}u\right|\\
=&\frac{e^{-\frac{2n\pi}{\hbar}(\alpha\pi-\rho\sin{\theta})}\sin(\alpha\pi)}{\alpha\pi}
\int_{0}^{2\pi}\left|\frac{x^{\alpha}e^{\rho e^{i\theta}}}{e^{\frac{\rho}{\alpha}e^{i\theta}}-1}i\rho e^{i\theta}\right|\mathrm{d}\theta\\
\le&\frac{x^{\alpha}e^{1-\frac{2n\pi}{\hbar}(\alpha\pi-\rho\sin{\theta})}\sin(\alpha\pi)}{\pi}
\int_{0}^{2\pi}\left|\frac{\frac{\rho}{\alpha}e^{i\theta}}{e^{\frac{\rho}{\alpha}e^{i\theta}}-1}\right|\mathrm{d}\theta\\
=&\mathcal{O}(1)x^{\alpha}e^{-\frac{2n\pi}{\hbar}(\alpha\pi-\rho)}.
\end{align*}
Therefore, it follows
\begin{align*}%\label{positive_DFT_decay_rat}
  \mathfrak{F}\left(\frac{2n\pi}{\hbar}\right)
  =e^{-\frac{3n\alpha\pi^2}{\hbar}}\mathcal{O}(1)
  +\mathcal{O}(1)x^{\alpha}e^{-\frac{2n\pi}{\hbar}(\alpha\pi-\rho)},\ n=1,2,\cdots
\end{align*}
and letting $\rho\rightarrow 0$, it yields \eqref{eq:conclusionOfFouriersum_uniform} for $n>0$

Similarly, by the same approach with the broken line replaced by
$$
-\infty\rightarrow-\infty+\frac{3}{2}i\alpha\pi\rightarrow \dot{u}_1+\frac{1}{2}i\alpha\pi\rightarrow\dot{C}^+_{\rho}
\rightarrow \dot{u}_1+\frac{1}{2}i\alpha\pi\rightarrow+\infty+\frac{3}{2}i\alpha\pi\rightarrow+\infty
$$
with $\dot{C}^+_{\rho}=:\{u=\dot{u}_1+\rho e^{i\theta}:\ 0\le\theta<2\pi\}$ (see Fig. \ref{integral_contour_uniform}),
the exponential bounds of the $n$-th $(n\le-1)$ discrete Fourier transform
\eqref{eq:conclusionOfFouriersum_uniform}
 also holds.

Thus, the summation of all the discrete Fourier transforms $\mathfrak{F}\left(\frac{2n\pi}{\hbar}\right)$ yields the desired conclusion \eqref{eq:conclusionOfFouriersum_uniform}
\begin{align*}
\sum_{n\ne0}\mathfrak{F}[\overline{f}]\big{(}\frac{2 n\pi }{\hbar}\big{)}
=\mathcal{O}(1)\sum_{n\ne0}e^{-\frac{3|n|\alpha\pi^2}{\hbar}}
  +\mathcal{O}(1)x^{\alpha}\sum_{n\ne0}e^{-\frac{2|n|\pi^2\alpha}{\hbar}}
= \frac{\mathcal{O}(1)}{e^{\frac{2\alpha\pi^2}{\hbar}}-1}.
\end{align*}

It is clear that all the constants in $\mathcal{O}(1)$s in this proof are independent of $n$, $\hbar$, $x$, $\alpha$ and $\sigma$.
\end{proof}

According to \eqref{eq:estunif}, \eqref{eq:quadrature}, \eqref{trapzoid_S(x)} and \eqref{QuadratureErrorfor_w}, we have for the quadrature error $\overline{E}_{Q}(x)$ that
\begin{align}\label{quadratureOfbarf}
-\sum_{n\ne0}\mathfrak{F}[\overline{f}]\big{(}\frac{2 n\pi }{\hbar}\big{)}
=&\overline{E}_Q(x)
=\int_{-\infty}^{+\infty}\overline{f}(u,x)\mathrm{d}u-\hbar\sum_{j=-\infty}^{+\infty}\overline{f}(j\hbar,x)\\
=&\int_{0}^{\bar{N}_t\hbar}\overline{f}(j\hbar,x)\mathrm{d}u-\hbar\sum_{j=0}^{\bar{N}_t}\overline{f}(j\hbar,x)
+\overline{E}_T(x)-\overline{E}_{TD}(x)\notag\\
=&\overline{I}(x)-\overline{S}(x)+\overline{E}_T(x)-\overline{E}_{TD}(x),\notag
\end{align}
where $\overline{E}_{TD}$ satisfies by \eqref{eq:inequ} that
\begin{align}\label{DiscreteTrucatedError}
0\le\overline{E}_{TD}(x)=&\hbar\sum_{j=-\infty}^{-1}\overline{f}(j\hbar,x)
+\hbar\sum_{j=\bar{N}_t+1}^{+\infty}\overline{f}(\hbar,x)\notag\\
\le&\frac{\hbar\sin(\alpha\pi)}{\alpha\pi}\left(\sum_{j=-\infty}^{-1}e^{j\hbar-T}
+\sum_{j=\bar{N}_t+1}^{+\infty}e^{-\frac{j\hbar-T}{\kappa}}\right)\notag\\
\le&e^{-T}\int_{-\infty}^{0}e^t\mathrm{d}t
+\frac{\sin(\alpha\pi)}{\alpha\pi}\int_{(\kappa+1)T}^{+\infty}e^{-\frac{t-T}{\kappa}}\mathrm{d}t\notag\\
=&e^{-T}+\frac{\sin((1-\alpha)\pi)}{(1-\alpha)\pi}e^{-T}\le2e^{-T}.
\end{align}

Then from \eqref{quadratureOfbarf}, \eqref{DiscreteTrucatedError} and Theorem \ref{Quadratrue_rat_uniform} we have
by $\|\overline{E}_T\|_{\infty}\le 3e^{-T}$ that
\begin{align}\label{errquad_uniform}
\overline{I}(x)-\overline{S}(x)=\overline{I}(x)-\bar{r}_{\bar{N}_t}(x)
=\overline{E}_Q(x)-\overline{E}_T(x)+\overline{E}_{TD}(x)
=\frac{\mathcal{O}(1)}{e^{\frac{2\alpha\pi^2}{\hbar}}-1}+\mathcal{O}(e^{-T})
\end{align}
and the constants in the above $\mathcal{O}$ terms are independent of $T$, $n$, $\hbar$, $x$, $\alpha$ and $\sigma$.

Fig. \ref{quadratureerr_uni} illustrates the quadrature errors $\max_{x\in [0,1]}|\overline{S}(x)-\overline{I}(x)|$ of the rectangular rule  with various values of $\hbar$ and $\alpha$. From Fig. \ref{quadratureerr_uni}, we see that the quadrature errors with respect to $\hbar=\frac{\sqrt{2\alpha}\pi}{\sqrt{N_1}}$ are smaller than those to $\hbar=\frac{3\sqrt{\alpha}\pi}{\sqrt{N_1}}$ and $\hbar=\frac{\sqrt{\alpha}\pi}{\sqrt{N_1}}$ for fixed $\alpha$.

%%%%%%%%%%%%%%%%%%%%%%%%%%%%%%%%%%%%%%%%%%%%%%%%%%%%%%%%%%%%%%%%%%%%%%%%%%%%%%%%%%%%%
%%%%%%%%%%%%%%%%%%%%%%%%%%%%%%%%%%%%%%%%%%%%%%%%%%%%%%%%%%%%%%%%%%%%%%%%%%%%%%%%%%%%%
\subsection{Quadrature errors with  tapered exponentially clustered poles \eqref{eq:tapered2}}\label{sec32}
As  $f(u,x)$ is not included in $L^2$ and does not satisfies the conditions of Theorem \ref{StengerPossionFormula}, thus  Poisson's summation formula \eqref{Possionsummationformula} can not be applied directly to the integrand $f(u,x)$, which stems from the singularity of factor $\frac{1}{\sqrt{u}}$ at the original point. Thus, our strategy is to separate the pole and branch singularities by dividing the interval $[0,1]$ of $x$ into two subinterval, on which the quadrature errors of $\int_0^{+\infty}f(u,x)\mathrm{d}u$ are explored meticulously in Subsections \ref{subsec:3.1} and \ref{subsec:3.2}, respectively.

\subsubsection{A uniform bound of the quadrature error of $I(x)$ for $x$ near $0$}\label{subsec:3.1}
Following \cite[Lemma A.2]{Herremans2023}, similarly we get a uniform bound of $I(x)-S(x)$ for small $x$.

\begin{theorem}\label{la1}
Let $x^*=\frac{\frac{1}{\kappa}\sqrt{u^*}+1}{\sqrt{u^*}-1}
e^{\frac{1}{\alpha}(\sqrt{u^*}-T)}$ with
\begin{align}\label{eq:ustar0}
u^{*}=\frac{1+(1-2\alpha)\sqrt{4\alpha-4\alpha^2+1}}
{2(1-\alpha)^2}.
\end{align}
Then the quadrature error holds uniformly for $x\in [0,x^*]$ that
\begin{align}\label{eq:qerr}
0\le I(x)-S(x)\le I(x)\le  \left(\frac{\frac{1}{\kappa}\sqrt{u^*}+1}{\sqrt{u^*}-1}\right)^{\alpha}e^{\sqrt{u^*}-T}.
\end{align}
Additionally, $u^{*}$
monotonically increases for $\alpha\in(0,1)$ satisfying $u^{*}\in(1,4)$ and  $\left(\frac{\frac{1}{\kappa}\sqrt{u^*}+1}{\sqrt{u^*}-1}\right)^{\alpha}e^{\sqrt{u^*}}$ is uniformly bounded for $\alpha\in (0,1)$.
\end{theorem}
\begin{proof}
It is obvious from \eqref{eq:int} and \eqref{eq:quadrature} that $0<I(x)<x^{\alpha}$ due to that the integrand $f(u,x)$ is positive for all $[0,+\infty)\times(0,1]$. By a relatively tedious but straightforward calculation, we show that $f(u,x)$ is monotonically decreasing for any fixed $x$ between $0$ and $x^*=\frac{\frac{1}{\kappa}\sqrt{u^*}+1}{\sqrt{u^*}-1}e^{\frac{1}{\alpha}(\sqrt{u^*}-T)}$, as a function of $u\in[0,+\infty)$. Its proof is sketched as follows.

For fixed $x\in[0,1]$, $f(u,x)$ is decreasing with respect to $u\in[0,+\infty)$ if and only if the  partial derivative
$$
\partial_uf(u,x)=\frac{f(u,x)}{2\sqrt{u}}
\left[\frac{x-\frac{1}{\kappa}e^{\frac{1}{\alpha}(\sqrt{u}-T)}}
{x+e^{\frac{1}{\alpha}(\sqrt{u}-T)}}-\frac{1}{\sqrt{u}}\right]\le 0
$$
which is equivalent to
$$
x(\sqrt{u}-1)\le \bigg(\frac{\sqrt{u}}{\kappa}+1\bigg)e^{\frac{1}{\alpha}(\sqrt{u}-T)}.
$$

For $0\le u\le 1$, the above inequality is satisfied obviously.
While for  $u>1$, $f(u,x)$ is decreasing with respect to $u$ if and only if
\begin{align}\label{eq:gu}
x\le \frac{\frac{1}{\kappa}\sqrt{u}+1}{\sqrt{u}-1}
e^{\frac{1}{\alpha}(\sqrt{u}-T)}=:g(u),\, u>1.
\end{align}
It is easy to check that $u^*$ in \eqref{eq:ustar0} is the minimum point of $g(u)$ by solving the equation
$$
g'(u)=\frac{e^{\frac{1}{\alpha}(\sqrt{u}-T)}}{2\sqrt{u}(\sqrt{u}-1)}
\bigg{[}\frac{1}{\alpha}\big{(}\frac{\sqrt{u}}{\kappa}+1\big{)}
-\frac{\frac{1}{\kappa}+1}{\sqrt{u}-1}\bigg{]}=0,
$$
that is,
\begin{align}\label{eq:guroot}
u+(\kappa-1)\sqrt{u}-(\alpha\kappa+\alpha+\kappa)=0.
\end{align}
By solving \eqref{eq:guroot}, readily we have
$$
u^*=\alpha+\alpha\kappa+\frac{\kappa^2+1}{2}+
\frac{1-\kappa}{2}\sqrt{(\kappa+1)(4\alpha+\kappa+1)},
$$
which is equivalent to \eqref{eq:ustar0} with substituting $\kappa$ by $\frac{\alpha}{1-\alpha}$.
Hence, for all fixed $x\in[0, g(u^*)]$ it follows that $f(u,x)$ monotonically decreases with respect to $u\in(1,+\infty)$.
The continuity of $f$ implies that $f(u,x)$ is monotonically decreasing against $u>0$.

Since $S(x)$ is a lower Riemann sum of $f(u,x)$ on $u\in[0,+\infty)$, which implies $0<S(x)\le I(x)\le x^\alpha\le(x^*)^\alpha$ for $x\in [0,x^*]$. Then it derives \eqref{eq:qerr}.

In particular, for $\alpha=\frac{1}{2}$, $u^*=2$ and $x^*=\frac{\sqrt{2}+1}{\sqrt{2}-1}e^{2\sqrt{2}-2T}$ (cf. \cite{Herremans2023}).
Additionally, from \eqref{eq:ustar0}, it is easy to verify that $u^{*}$
 is monotonically increasing in $\alpha \in (0,1)$ and
\begin{align*}
\lim_{\alpha\rightarrow0^+}u^{*}=1,\quad\quad \lim_{\alpha\rightarrow1^-}u^{*}=4.
\end{align*}

Furthermore, by using  $\sqrt{u^*}-1=\frac{\alpha\kappa+\alpha}{\kappa+\sqrt{u^*}}=\frac{\alpha}{\alpha+(1-\alpha)\sqrt{u^*}}$ derived from \eqref{eq:guroot}, $\left(\frac{\frac{1}{\kappa}\sqrt{u^*}+1}{\sqrt{u^*}-1}\right)^{\alpha}e^{\sqrt{u^*}}
=e^{\alpha\log\left(\frac{(1-\alpha)\sqrt{u^*}+\alpha}{\alpha(\sqrt{u^*}-1)}\right)}e^{\sqrt{u^*}}
=e^{2\alpha\log\left(\frac{(1-\alpha)\sqrt{u^*}+\alpha}{\alpha}\right)}e^{\sqrt{u^*}}$ is uniformly bounded for $\alpha\in (0,1)$.
\end{proof}

Fig. \ref{pointwise_errors_of_abs_x} illustrates the minimum point of $g(u)$ and the monotonicity of $f(u,x)$ for $x\in [0,x^*]$.

\begin{figure}[htbp]
\centerline{\includegraphics[width=14cm]{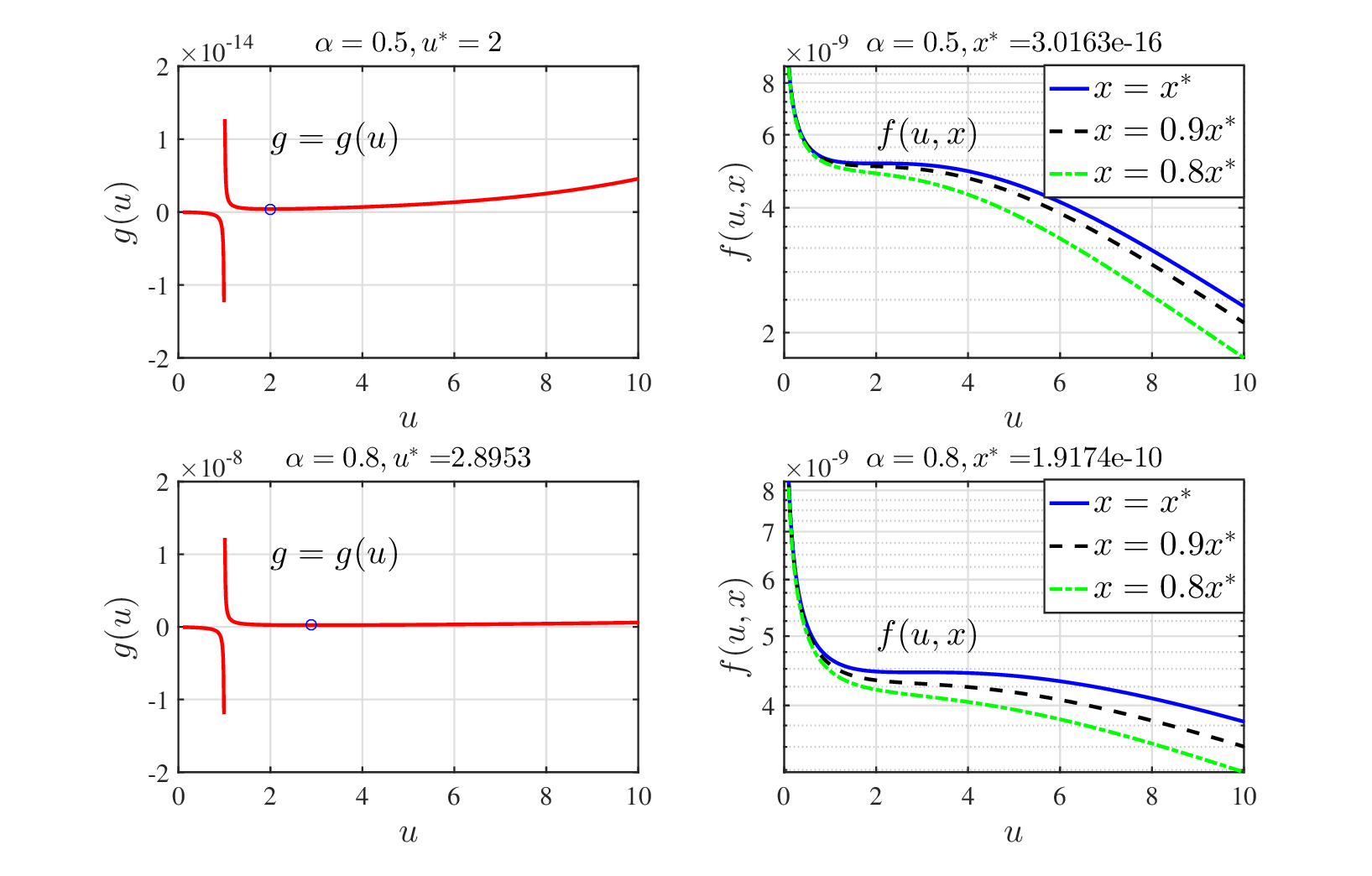}}
\caption{The function $g(u)$ in \eqref{eq:gu} is negative in $[0,1)$ and positive in $(1,\infty)$, wherein it takes a minimum value at $u^*$ (circled). Correspondingly, $f(u,x)$ monotonically decreases in
$[0,+\infty)$ for all fixed $x\in[0,x^*]$, where it takes $x=x^*,0.9x^*$ and $0.8x^*$ for illustration.}
\label{pointwise_errors_of_abs_x}
\end{figure}

\begin{remark}
Theorem \ref{la1} together with the truncated error $E_T(x)$ \eqref{eq:ECrat} and approximation error
$E_{PA}(x)$ of $P_{N_2}$ \eqref{polynomial app} leads to that for $T=\sigma\alpha\sqrt{N_1}>0$
\begin{align}\label{eq:xstar}
\left\{
\begin{array}{ll}
\|x^\alpha-r_{N}(x)\|_{\infty}=\mathcal{O}(e^{-T}),&\mbox{for all $x\in[0,1]$ if $x^*\ge 1$},\\
\|x^\alpha-r_{N}(x)\|_{\infty}=\mathcal{O}(e^{-T}),&\mbox{for all $x\in[0,x^*]$ if  $x^*< 1$.  }
\end{array}\right.
\end{align}
and the constants in the above $\mathcal{O}$ terms \eqref {eq:xstar} are independent of $T$, $x^*$, $h$, $\sigma$ and $\alpha$.
In the following, we focus on $x^*\in (0,1)$.
\end{remark}

\subsubsection{A uniform bound of the quadrature error of $I(x)$ for $x\in [x^*,1]$}
\label{subsec:3.2}
The main thought of the present part is similar to Subsection \ref{subsec:3.01} in some aspects, however, it may be much more complicated as a result of the branch singularity of $f(u,x)$. To take the advantage of the Fourier transform of $\hat f(u,x)$ and the Poisson summation \eqref{QuadratureErrorfor_w}, we will confirm \eqref{eq:errs} for $x\in [x^*,1]$.

From $1<u^*<4$ depending only on $\alpha$,  we see that $\frac{\frac{1}{\kappa}\sqrt{u^*}+1}{\sqrt{u^*}-1}>1$ due to that for $\alpha\le \frac{1}{2}$,  $\frac{1}{\kappa}=\frac{1-\alpha}{\alpha}\ge 1$, while for $\alpha>\frac{1}{2}$, $\left(1-\frac{1}{\kappa}\right)\sqrt{u^*}< 2\left(1-\frac{1}{\kappa}\right)=\frac{2(2\alpha-1)}{\alpha}<2$. Then
we may rewrite $\frac{\frac{1}{\kappa}\sqrt{u^*}+1}{\sqrt{u^*}-1}e^{\frac{1}{\alpha}(\sqrt{u^*}-T)}=e^{\frac{1}{\alpha}(\gamma-T)}$ with $\gamma>1$ uniformly bounded by Theorem \ref{la1} and independent of $T$.

In the rest of this paper, for the root-valued function $\sqrt{u}$ we consider mainly its principal branch holomorphic in the slit plane $\mathbb{C}\setminus(-\infty,0]$, and $f(u,x)$ also refers to its principal branch.
Similarly to \cite[Lemma A.3]{Herremans2023}, for any fixed $x\in [x^*,1]$ and as a function of $u$, $f(u,x)$ has the simple poles
\begin{equation*}%\label{eq:all_poles_fux}
u_k(x)=(T+\alpha\log{x})^2-\alpha^2\pi^2(2k-1)^2
+i\big{[}2\alpha\pi(2k-1)(T+\alpha\log{x})\big{]},
\end{equation*}
where $k=0,\pm1,\ldots$, among which the closest to the real axis are $u_0(x)$ and $u_1(x)$ (see Fig. \ref{decay_behavior_of_fux}).
We denote them by $u_0$ and $u_1$ for brevity, which take the same real part
$$v_0=(T+\alpha\log{x})^2-\alpha^2\pi^2,$$
and
whose imaginary parts are
$$\mp a=\mp2\alpha\pi(T+\alpha\log{x}),$$
respectively. Then $f(u,x)$ is an analytic function in the strip domain
$$\big{\{}u\in\mathbb{C}:|\Im(u)|\le a,\ \Re(u)> 0\big{\}}$$
except for two simple poles $u_0$ and $u_1$ on the boundary,
where $\Re(u)$  denotes the real  part of $u$.

\begin{figure}[htbp]
\centerline{\includegraphics[height=3.6cm,width=13cm]{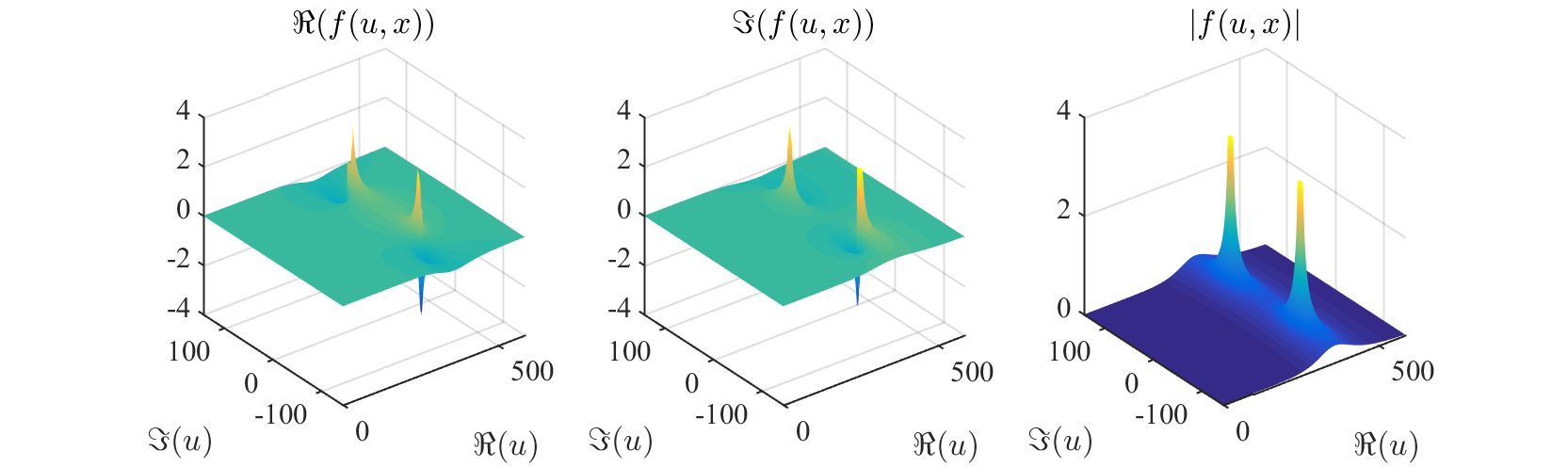}}
\caption{The decay behaviors of real, imaginary parts and the modulus of $f(u,x)$ \eqref{eq:fun} on the strip domain $\{u\in \mathbb{C}:|\Im(u)|\le 2a,\ \Re(u)>0\}$, where $\alpha=0.6$, $x=0.5$ and $T=20\sqrt{\alpha}(1-\alpha)\pi$. The singularities in the figures correspond to $u_0$ and $u_1$, the poles nearest to the real line of $f(u,x)$.}
\label{decay_behavior_of_fux}
\end{figure}

Additionally, we observe that $u_{-1}(x)$ and $u_2(x)$ locate on the boundary of the strip domain
$$\{u\in\mathbb{C}: |\Im(u)|<3a\},$$
and the remaining poles $u_k(x),k=-2,\pm3, \ldots$ of $f(u,x)$ locate outside.
Fig. \ref{decay_behavior_of_fux} shows the two poles nearest to the real axis of $f(u,x)$ in \eqref{eq:fun}, where we also observe the decay behaviors of real, imaginary parts and the modulus of $f(u,x)$ on the strip domain $\{u\in \mathbb{C}:|\Im(u)|\le 2a,\ \Re(u)>0\}$.

Set $u=v+iw=re^{i\theta}$, $r=\sqrt{v^2+w^2}$. It is easy to verify by $\cos\theta=\frac{v}{r}$  and the half angle formula that
\begin{align}\label{eq:sqrtr}
\Re(\sqrt{u})=\sqrt{\frac{\sqrt{v^2+w^2}+v}{2}},\ \Im(\sqrt{u})=\pm\sqrt{\frac{\sqrt{v^2+w^2}-v}{2}}=\pm \sqrt{\frac{w^2}{2(\sqrt{v^2+w^2}+v)}}.
\end{align}
From the definition of $u_0$, it is easy to check that
\begin{align}\label{eq:sqrtru0}
\sqrt{u_0}=T+\alpha\log x \pm i\alpha\pi.
\end{align}
In this paper, we confine to $\sqrt{u_0}=T+\alpha\log x - i\alpha\pi$.

Just as mentioned above, the existing conclusion relevant to Paley-Wiener Theorems \cite{PaleyWiener1934,Trefethen2014SIREV} cannot be directly applied to $f(u,x)$ since the distance $|a|=2\alpha\pi|T+\alpha\log x|$ of the closest poles $u_0$ and $u_1$ of $f(u,x)$ for $\Re(u)>0$ to the real axis varies in $[0,+\infty)$ for $x\in (0,1]$ while in
$(2\alpha\pi\gamma, 2\alpha\pi T]$ for $x\in [x^*,1]$. Moreover, the step size $h=\sigma^2\alpha^2$ is fixed. Thus to get exponential convergence rate on $T$ uniformly  for all $x\in [x^*,1]$, a large part of the follow-up effort
will be devoted to the establishment of Theorem \ref{eq:thm}  analogous to the Paley-Wiener result \cite{Trefethen2014SIREV} by $\hat{f}$ instead of $f$ and three additional useful lemmas in
Appendix \ref{AppendixA}, whose proofs involve meticulous analysis and exquisite techniques.

\bigskip

{\bf Case I}: By $\Upsilon$ we denote the set of all those $x\in [x^*,1]$ satisfying $v_0:=\Re(u_0)>M_0h$ with the positive integer $M_0$ defined in \eqref{DefinitionOfM0}, which implies that
\begin{equation}\label{eq:real}
v_0>\max\left\{2h+18\alpha^2\pi^2,\ 4\big(\sqrt{h}+\sqrt{\alpha\pi}\big)^4\right\}.
\end{equation}
In particular, from the definition of $a=2\alpha\pi(T+\alpha\log{x})$, we see that $a\ge 2\alpha\pi\gamma>0$ for $x\in[x^*,1]$.

\begin{theorem}\label{eq:thm}
Let $\hat f(u,x)$ be defined in \eqref{eq:extension_of_f} with $x\in\Upsilon$.
Then the sum of discrete Fourier transform decays at an exponential rate
\begin{align}\label{eq:conclusionOfFouriersum}
\sum_{n\ne0}\mathfrak{F}[\hat f]\big{(}\frac{2 n\pi }{h}\big{)}= e^{h}\mathcal{O}(e^{-T})+\mathcal{O}\left(\frac{x^\alpha}{e^{\frac{2\pi a}{h}}-1}\right),\quad a=2\pi\alpha(T+\alpha\log{x}),
\end{align}
and the constants in the $\mathcal{O}$ terms \eqref{eq:conclusionOfFouriersum} are independent of $T$, $n$, $h$, $x$, $\alpha$ and $\sigma$.
\end{theorem}

\begin{proof}
We establish the conclusion by leveraging several lemmas that are sketched in Appendix \ref{AppendixA}
to avoid unnecessary repetition.

From the definition of \eqref{eq:extension_of_f}, $\hat f(u,x)$ is continuous and piecewise smooth for $u\in (-\infty,+\infty)$ with arbitrarily fixed $x\in (0,1]$, and it is easy to verify that
$\hat f\in L^2(\mathbb{R})$.

Moreover, applying $f(h,x)=\frac{\sin(\alpha\pi)}{\alpha\pi}
\frac{1}{2\sqrt{h}}\frac{xe^{\sqrt{h}-T}}
{e^{\frac{1}{\alpha}(\sqrt{h}-T)}+x}\le \frac{\sin(\alpha\pi)}{\alpha\pi}
\frac{e^{\sqrt{h}(1-\sqrt{N_1})}}{2\sqrt{h}}$, it follows
$$\int_{-\infty}^{+\infty}|\hat{f}(u,x)|\mathrm{d}u=2\int_0^hf(h,x)\mathrm{d}u+2\int_h^{+\infty}|f(u,x)|\mathrm{d}u <M_1$$
also holds for a certain constant $M_1=4+\sqrt{h}e^{\sqrt{h}(1-\sqrt{N_1})}\le4+e^{-1}$ uniformly for all $x\in(0,1]$ as $N_1>1$. Then $\hat{f}(u,x)$ satisfies $\int_{-\infty}^{+\infty}|\hat{f}(u,x)|\mathrm{d}u<+\infty$, and thus the Fourier transform $\mathfrak{F}[\hat{f}(u,x)]=\int_{-\infty}^{+\infty}\hat{f}(u,x)e^{-i\xi u}\mathrm{d}u$
exists and continuous for all real $\xi$ (see \cite[(10.6-12)-(10.6-13)]{Henrici}).

Define an $h$-periodic function
\begin{equation*}%\label{eq:per}
F(\upsilon,x)=\sum_{k=-\infty}^{\infty}\hat{f}(kh+\upsilon,x),\ \upsilon\in[0,h],
\end{equation*}
whose uniform convergence can be checked readily by \eqref{eq:inequ}.
Then following \cite[p. 270]{Henrici}, the $n$th $(n\ge1)$ Fourier coefficient
of $F(\upsilon,x)$ satisfies that
\allowdisplaybreaks[4]
\begin{align}
c_n=&\frac{1}{h}\mathfrak{F}[\hat f]\big{(}\frac{2n\pi}{h}\big{)}
=\frac{1}{h}\int_0^hF(\upsilon,x)e^{-i\frac{2n\pi}{h}\upsilon}\mathrm{d}\upsilon\notag\\
=&\frac{1}{h}\sum_{k=-\infty}^{\infty}
\int_{kh}^{(k+1)h}\hat f(u,x)e^{-i\frac{2n\pi}{h}u}\mathrm{d}u\label{eq:fourier_c}\\
=&\frac{1}{h}
\int_{h}^{+\infty} f(u,x)e^{i\frac{2n\pi}{h}u}\mathrm{d}u
+\frac{1}{h}
\int_{h}^{+\infty} f(u,x)e^{-i\frac{2n\pi}{h}u}\mathrm{d}u\notag\\
&+\frac{2}{h}\int_{0}^{h}f(h,x)\cos{\big{(}\frac{2n\pi}{h}u\big{)}}\mathrm{d}u\notag\\
=&\frac{1}{h}
\int_{\Gamma^{+}_{\rho,h}}f(u,x)e^{i\frac{2n\pi}{h}u}\mathrm{d}u
+\frac{1}{h}
\int_{\Gamma^{-}_{\rho,h}} f(u,x)e^{-i\frac{2n\pi}{h}u}\mathrm{d}u\label{eq:fourier_positive11}\\
=&\frac{i}{h}\int_{0}^{2a}\big{[}f(h+it,x)-f(h-it,x)\big{]}
e^{-\frac{2n\pi}{h}t}\mathrm{d}t\label{eq:fourier_positive11111}\\
&+\frac{1}{h}\left\{
\int_{h-2ia}^{u_0-ia} +\int_{C^{-}_{\rho}}+\int_{u_0-ia}^{+\infty-2ia}\right\}f(u,x)e^{-i\frac{2n\pi}{h}u}\mathrm{d}u\notag\\
&+\frac{1}{h}\left\{
\int_{h+2ia}^{u_1+ia} +\int_{C^{+}_{\rho}}+\int_{u_1+ia}^{+\infty+2ia}\right\}f(u,x)e^{i\frac{2n\pi}{h}u}\mathrm{d}u,\notag\\
=&\frac{i}{h}\int_{0}^{2a}\big{[}f(h+it,x)-f(h-it,x)\big{]}
e^{-\frac{2n\pi}{h}t}\mathrm{d}t\label{eq:fourier_positive}\\
&+\frac{1}{h}\left\{
\int_{h-2ia}^{+\infty-2ia} +\int_{C^{-}_{\rho}}\right\}f(u,x)e^{-i\frac{2n\pi}{h}u}\mathrm{d}u\notag\\
&+\frac{1}{h}\left\{
\int_{h+2ia}^{+\infty+2ia} +\int_{C^{+}_{\rho}}\right\}f(u,x)e^{i\frac{2n\pi}{h}u}\mathrm{d}u,\notag
\end{align}
where we used $\frac{2}{h}\int_{0}^{h}f(h,x)\cos{\big{(}\frac{2n\pi}{h}u\big{)}}\mathrm{d}u
=\frac{2f(h,x)}{h}\int_{0}^{h}\cos{\big{(}\frac{2n\pi}{h}u\big{)}}\mathrm{d}u=0$,
and
$C^{\pm}_{\rho}=\{z=u_0+\rho e^{i\theta},\theta:0\rightarrow\pm2\pi\}$ with
$0<\rho=\frac{1}{2}\min\{\alpha^2\pi^2,2\pi\alpha\gamma,\frac{h}{N_0}\}$ for some fixed sufficiently large  $N_0$ independent of $x$ and $T$. We also used in \eqref{eq:fourier_positive11} the Cauchy's integral theorem on the analytic function $f(u,x)e^{\pm i\frac{2n\pi}{h}u}$, and the integrals on $[h,+\infty)$ are converted to those on the paths (see Fig. \ref{integral_contour}):
$$\Gamma^{-}_{\rho,h}:\ h\rightarrow(h-2ia)\rightarrow(u_0-ia)\rightarrow C_{\rho}^{-}
\rightarrow(u_0-ia)\rightarrow(+\infty-2ia)\rightarrow+\infty,$$
$$\Gamma^{+}_{\rho,h}:\ h\rightarrow(h+2ia)\rightarrow(u_1+ia)\rightarrow C_{\rho}^{+}
\rightarrow(u_1+ia)\rightarrow(+\infty+2ia)\rightarrow+\infty.$$

\begin{figure}[htp]
\centerline{\includegraphics[width=14cm]{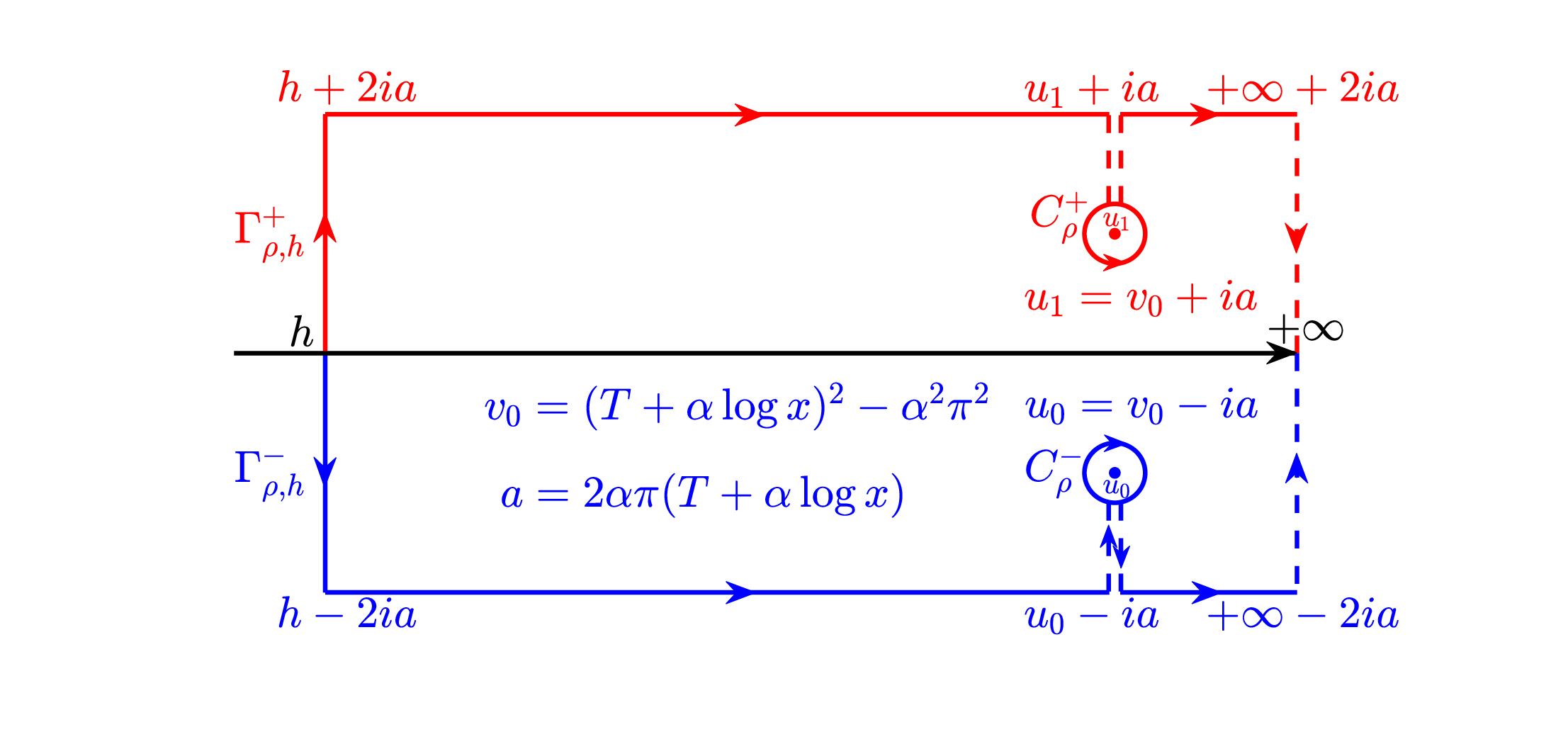}}
\caption{The integral contours $\Gamma^{-}_{\rho,h}$ (blue) and $\Gamma^{+}_{\rho,h}$ (red). The poles nearest to the real line of $f(u,x)$ are $u_0$ and $u_1$.
Together with the straight line $[h,+\infty)$ in the opposite direction, they form two closed circuits, wherein $f(u,x)$ is holomorphic.}
\label{integral_contour}
\end{figure}

\noindent
The integrals on the vertical line segments $(u_0-ia)\rightleftharpoons C^{-}_{\rho}$ and $(u_1+ia)\rightleftharpoons C^{+}_{\rho}$ (not include $C^{\pm}_{\rho}$) are canceled.

Analogously, we use in \eqref{eq:fourier_positive11111} the fact that
$$
\lim_{A\rightarrow +\infty}\int_{A\pm2ia}^{A}
f(u,x)e^{\pm i\frac{2n\pi}{h}u}\mathrm{d}u=0
$$
by $\frac{1}{|\sqrt{A\pm it}|}\le \frac{1}{\sqrt{A}}$ and \eqref{eq:sqrtr} that for $A>T$
$$
\left|\frac{xe^{\sqrt{A\pm it}-T}}{e^{\frac{1}{\alpha}(\sqrt{A\pm it}-T)}+x}\right|\le  \frac{1}{e^{\frac{1}{\kappa}\Re(\sqrt{A\pm it}-T)}-e^{-\Re(\sqrt{A\pm it}-T)}}\le
 \frac{1}{e^{\frac{1}{\kappa}(\sqrt{A}-T)}-1}
$$
and then
$$
\left|\int_{A\pm2ia}^{A}
f(u,x)e^{\pm i\frac{2n\pi}{h}u}\mathrm{d}u\right|=\left|\int_{0}^{2a}
f(A\pm it,x)e^{-\frac{2n\pi}{h}t}\mathrm{d}t\right|\le \frac{a}{\sqrt{A}(e^{\frac{1}{\kappa}(\sqrt{A}-T)}-1)}.
$$
Thus, the integrals in \eqref{eq:fourier_positive} can be bounded uniformly as follows
\begin{align*}
&\bigg{|}\int_{0}^{2a}\big{[}f(h+it,x)-f(h-it,x)\big{]}
e^{-\frac{2n\pi}{h}t}\mathrm{d}t\bigg{|}\\%\label{bound on vertical line}\notag\\
=&\left(\frac{e^{h}}{h^2}
\int_0^1te^{-\frac{2n\pi}{h}t}\mathrm{d}t
+\frac{e^{\sqrt{h}}}{\sqrt{h}}\int_1^{+\infty}\frac{e^{\sqrt{t}}}{\sqrt{t}}
e^{-\frac{2n\pi}{h}t}\mathrm{d}t\right)\mathcal{O}(e^{-T}), \hspace{1.3cm} \mbox{(Lemma \ref{la3})} \\
&\bigg{|}\int_{C^{\pm}_{\rho}}
f(u,x)e^{\pm i\frac{2n\pi}{h}u}\mathrm{d}u\bigg{|}
=e^{(\rho-a)\frac{2n\pi}{h}}x^{\alpha}\mathcal{O}(1),\hspace{3cm} \mbox{(Lemma \ref{lemma_inte_circ})}\\%\label{bound on vertical circle}\\
&\bigg{|}
\int_{h\pm 2ia}^{+\infty\pm 2ia}f(u,x)e^{\pm i\frac{2n\pi}{h}u}\mathrm{d}u\bigg{|}
=e^{-\frac{4n\pi}{h}a}x^{\alpha}\mathcal{O}(1),\hspace{2.5cm} \mbox{(Lemma \ref{infty})}%\label{bound on vertical infty}
\end{align*}
respectively, where the constants in the above $\mathcal{O}$ terms are independent of $x$, $\alpha$, $\sigma$, $\rho$, $n$ and $T$.  Then letting $\rho\rightarrow0$, it follows
for $n\ge 1$  that
\begin{align}\label{eq:fourier_minus}
&h|c_n|=\big{|}\mathfrak{F}[\hat f]\big{(}\frac{2 n\pi}{h}\big{)}\big{|}\\
=&\left(\frac{e^{\sqrt{h}}}{h^2}
\int_0^1te^{-\frac{2n\pi}{h}t}\mathrm{d}t
+\frac{e^{\sqrt{h}}}{\sqrt{h}}\int_1^{+\infty}\frac{e^{\sqrt{t}}}{\sqrt{t}}
e^{-\frac{2n\pi}{h}t}\mathrm{d}t\right)\mathcal{O}(e^{-T})+e^{-\frac{2n\pi}{h}a}x^{\alpha}\mathcal{O}(1).\notag
\end{align}

In the analogous way, from \eqref{eq:fourier_c} the estimate \eqref{eq:fourier_minus} for $n\le-1$ holds with $e^{-\frac{2n\pi}{h} t}$ and $e^{-\frac{2n\pi}{h}a}$ replaced by $e^{\frac{2n\pi}{h}t}$ and $e^{\frac{2n\pi}{h}a}$, respectively. Thus, we have
\begin{align}
&h\bigg|\sum_{n\ne 0}c_n\bigg|=\bigg|\sum_{n\ne0}\mathfrak{F}[\hat f]\big{(}\frac{2 n\pi }{h}\big{)}\bigg|\notag\\
=&\left(\frac{e^{h}}{h^2}\sum_{n=1}^{\infty}
\int_0^12te^{-\frac{2n\pi}{h}t}\mathrm{d}t
+\frac{e^{\sqrt{h}}}{\sqrt{h}}\int_1^{+\infty}\frac{e^{\sqrt{t}}}{\sqrt{t}}
\frac{2}{e^{\frac{2\pi}{h}t}-1}\mathrm{d}t\right)\mathcal{O}(e^{-T})+\frac{x^{\alpha}\mathcal{O}(1)}{e^{\frac{2\pi a}{h}}-1}\notag\\
=&e^{h}\left(\frac{-1}{2\pi}\sum_{n=1}^{\infty}\frac{2}{nhe^{\frac{2n\pi}{h}}}
-\frac{1}{4\pi^2}\sum_{n=1}^{\infty}\frac{2}{n^2e^{\frac{2n\pi}{h}}}
+\frac{1}{4\pi^2}\sum_{n=1}^{\infty}\frac{2}{n^2}\right)\mathcal{O}(e^{-T})\label{00000000000}\\
&+e^{\sqrt{h}}\left(\sqrt{\frac{1}{2\pi}}
\int_{\sqrt{\frac{2\pi}{h}}}^{+\infty}\frac{4e^{\sqrt{\frac{h}{2\pi}}y}}{e^{y^2}-1}\mathrm{d}y\right)\mathcal{O}(e^{-T})
+\frac{x^{\alpha}\mathcal{O}(1)}{e^{\frac{2\pi a}{h}}-1},\notag\\
=&e^h\mathcal{O}(e^{-T})
+\frac{x^{\alpha}\mathcal{O}(1)}{e^{\frac{2\pi a}{h}}-1},\label{111}
\end{align}
where we used in \eqref{00000000000} the convergence of the series
$\sum_{n=1}^{\infty}\frac{2}{nhe^{\frac{2n\pi}{h}}}(<\sum_{n=1}^{\infty}\frac{1}{n^2\pi}$ by $e^{\frac{2n\pi}{h}}>1+\frac{2n\pi}{h}>\frac{2n\pi}{h}$), $\sum_{n=1}^{\infty}\frac{2}{n^2e^{\frac{2n\pi}{h}}}$,
$\sum_{n=1}^{\infty}\frac{2}{n^2}$, which can be bounded for all $\alpha\in(0,1)$ and $\sigma>0$,
and the fact for the integral
$\int_{\sqrt{2\pi/h}}^{+\infty}\frac{4e^{\sqrt{h/(2\pi)}t}}
{e^{t^2}-1}\mathrm{d}t$ that
\begin{align*}
  \int_{\sqrt{\frac{2\pi}{h}}}^{+\infty}\frac{e^{\sqrt{\frac{h}{2\pi}}y}}{e^{y^2}-1}\mathrm{d}y
  \le&\int_{0}^{1}\frac{e^{\sqrt{\frac{h}{2\pi}}y}}{e^{\frac{2\pi}{h}}-1}\mathrm{d}y
  +\int_{1}^{+\infty}\frac{e^{\sqrt{\frac{h}{2\pi}}y}}{\frac{1}{2}e^{y^2}}\mathrm{d}y\\
  \le&\frac{h}{2\pi}\sqrt{\frac{2\pi}{h}}(e^{\sqrt{\frac{h}{2\pi}}}-1)
  +2e^{\frac{h}{8\pi}}\int_{-\infty}^{+\infty}e^{-\left(t-\frac{1}{2}\sqrt{\frac{h}{2\pi}}\right)^2}\mathrm{d}t\\
  =&\big(e^{\sqrt{h}}+e^{\frac{h}{8\pi}}\big)\mathcal{O}(1)
\end{align*}
to get \eqref{111} by
$e^{\sqrt{h}}\big(e^{\sqrt{h}}+e^{\frac{h}{8\pi}}\big)=e^h\mathcal{O}(1)$
for all $h>0$.
\end{proof}

\vspace{0.66cm}
It is worthy of noting that by applying \eqref{eq:inequ},
\begin{align*}
\int_{N_th}^{+\infty}f(u,x)\mathrm{d}u
\le& \int_{(\kappa+1)^2T^2}^{+\infty}f(u,x)\mathrm{d}u
\le \frac{\sin(\alpha\pi)}{\alpha\pi}\int_{\kappa T}^{+\infty}e^{-\frac{1}{\kappa}t}\mathrm{d}t\\
=&\frac{\sin((1-\alpha)\pi)}{(1-\alpha)\pi}e^{-T}<e^{-T}
\end{align*}
 and
$$
f(h,x)\le \frac{e^{\sqrt{h}}}{2\sqrt{h}}e^{-T},\quad \int_0^hf(u,x)\mathrm{d}u\le \frac{e^{-T}\sin(\alpha\pi)}{\alpha\pi}\int_0^h\frac{e^{\sqrt{u}}}{2\sqrt{u}}\mathrm{d}u\le  e^{\sqrt{h}}e^{-T},
$$
we get
\begin{align}\label{eq:connect}
\int_{-\infty}^{+\infty}\hat f(u,x)\mathrm{d}u
=2hf(h,x)+2\int_h^{+\infty}f(u,x)\mathrm{d}u
%=&\textcolor[rgb]{1.00,0.00,0.00}{2}I(x)+2hf(h,x)-2\int_0^hf(u,x)\mathrm{d}u\\
=2I(x)+ e^{h}\mathcal{O}(e^{-T}).
\end{align}

On the other hand, it is obvious that
\begin{align}\label{eq:quadratureerrorformula_tapered}
\int_{-\infty}^{+\infty}\hat f(u,x)\mathrm{d}u
=&\sum_{j=-\infty}^{+\infty}\hat{f}(jh,u)+E_Q(x)
=2\sum_{j=1}^{N_t}f(jh,x)+E_{TD}(x)+E_Q(x)\\
=&2S(x)+E_{TD}(x)+E_Q(x)\notag
\end{align}
with
\begin{align}\label{definitionofE_Q(x)}
E_{Q}(x)=\int_{-\infty}^{+\infty}\hat f(u,x)\mathrm{d}u
-h\sum_{j=-\infty}^{+\infty}\hat f(jh,x)
=-\sum_{n\ne0}\mathfrak{F}[\hat f]
\big(\frac{2n\pi}{h}\big)
\end{align}
and
\begin{align*}
0\le E_{TD}(x)=&\hat{f}(0,x)h+2h\sum_{j=N_t+1}^{+\infty}\hat{f}(jh,x)
=hf(h,x)+2h\sum_{j=N_t+1}^{+\infty}f(jh,x)\\
\le & \frac{\sqrt{h}}{2}e^{\sqrt{h}}e^{-T}
+2h\frac{\sin{(\alpha\pi)}}{\alpha\pi}
\sum_{j=N_t+1}^{+\infty}\frac{1}{2\sqrt{jh}}
e^{-\frac{1}{\kappa}(\sqrt{jh}-T)}\\
\le& \frac{\sqrt{h}}{2}e^{\sqrt{h}}e^{-T}+\frac{2\sin{(\alpha\pi)}}{\alpha\pi}
\int_{N_th}^{+\infty}
\frac{e^{-\frac{1}{\kappa}(\sqrt{u}-T)}}{2\sqrt{u}}\mathrm{d}u\notag\\
\le& \frac{\sqrt{h}}{2}e^{\sqrt{h}}e^{-T}+\frac{2\sin(\alpha\pi)}{(1-\alpha)\pi}e^{-T}\\
=&e^{h}\mathcal{O}(e^{-T})=e^{\sigma\sqrt{2M_0}}\mathcal{O}(e^{-T})
\end{align*}
since $h=\alpha^2\sigma^2\le\sigma^2\le\sigma\sqrt{2M_0}$ by the definition of $M_0$ \eqref{DefinitionOfM0}.

Thus, combining \eqref{eq:connect} and \eqref{eq:quadratureerrorformula_tapered} yields that
\begin{align*}
I(x)-S(x)=e^h\mathcal{O}(e^{-T})+\frac{E_{TD}(x)}{2}+\frac{E_Q(x)}{2}
=e^{\sigma\sqrt{2M_0}}\mathcal{O}(e^{-T})+\mathcal{O}\left(\frac{x^\alpha}{e^{\frac{2\pi a}{h}}-1}\right).
\end{align*}
from the definition of $E_Q(x)$ \eqref{definitionofE_Q(x)} and the bound of the sum of discrete Fourier transforms \eqref{eq:conclusionOfFouriersum}.
Clearly, the constants in the above $\mathcal{O}$ terms are independent of $T$, $n$, $h$, $x$, $\alpha$ and $\sigma$.

\bigskip
 {\bf Case \uppercase\expandafter{\romannumeral2}}: For $x\in[x^*,1]\setminus\Upsilon$, we see that $v_0=\Re(u_0)=(T+\alpha\log{x})^2-\alpha^2\pi^2\le M_0h$ and $e^{\frac{1}{\alpha}(\gamma-T)}\le x\le e^{\frac{1}{\alpha}(\sqrt{M_0h+\alpha^2\pi^2}-T)}\le e^{\frac{1}{\alpha}(\sqrt{2M_0h}-T)}:=x_0$, which
implies
\begin{align*}%\label{eq:intappC}
|I(x)|
\le&\int_{0}^{(\kappa+1)^2T^2}f(u,x)\mathrm{d}u+e^{-T}\notag
=\frac{\sin(\alpha\pi)}{\alpha\pi}\int_{-T}^{\kappa T}
\frac{xe^{t}}
{e^{\frac{1}{\alpha}t}+x}\mathrm{d}t+e^{-T}\\
\le&\frac{x_0\sin(\alpha\pi)}{\alpha\pi}\int_{-T}^{\kappa T}
e^{(1-\frac{1}{\alpha})t}\mathrm{d}t+e^{-T}
=e^{\frac{1}{\alpha}(\sqrt{2M_0h}-T)}\frac{\sin{\alpha\pi}}{(1-\alpha)\pi}
\big(e^{\frac{T}{\kappa}}+e^{-T}\big)+e^{-T}\\
\le& \frac{e^{-T}\sin(\alpha\pi)}{(1-\alpha)\pi}e^{\frac{1}{\alpha}\sqrt{2M_0h}}+e^{-T}
=\frac{e^{-T}\sin(\alpha\pi)}{(1-\alpha)\pi}e^{\sigma\sqrt{2M_0}}+e^{-T}.
\end{align*}

Analogously, we have
\begin{align*}
|r_{N_t}(x)|=&\frac{\sin{(\alpha\pi)}}{\alpha\pi}\left|h\sum_{j=1}^{N_t}\frac{1}{2\sqrt{jh}}
\frac{xe^{\sqrt{jh}-T}}{e^{\frac{1}{\alpha}(\sqrt{jh}-T)}+x}\right|
\le\frac{x_0\sin{(\alpha\pi)}}{\alpha\pi}h\sum_{j=1}^{N_t}
\frac{e^{-\frac{1}{\kappa}\left(\sqrt{jh}-T\right)}}{2\sqrt{jh}}\notag\\
\le&\frac{x_0\sin{(\alpha\pi)}}{\alpha\pi}\int_{0}^{+\infty}e^{-\frac{1}{\kappa}(\sqrt{u}-T)}\mathrm{d}u
=e^{\frac{1}{\alpha}(\sqrt{2M_0h}-T)}\frac{\sin{(\alpha\pi)}}{(1-\alpha)\pi}e^{\frac{T}{\kappa}}\\
\le& \frac{\sin(\alpha\pi)}{(1-\alpha)\pi}e^{\sigma\sqrt{2M_0}}e^{-T}.
%\le&\frac{(\ell+2)!x^*C^\alpha\sin{(\alpha\pi)}}{\alpha\pi\varkappa(\beta)}\int_{-T}^{\kappa T}e^{-\frac{1}{\kappa}t}\mathrm{d}t+\mathcal{O}(e^{-T})\notag\\
%\le&\frac{(\ell+2)!e^{\frac{c_0}{\alpha}}\alpha C^{\alpha}\sin(\alpha\pi)}{\alpha\pi\varkappa(\beta)(\ell+1-\alpha)}e^{-T}+\mathcal{O}(e^{-T}),
\end{align*}
and then
\begin{align*}
|I(x)-r_{N_t}(x)|=|I(x)-S(x)|\le 2\frac{\sin(\alpha\pi)}{(1-\alpha)\pi}e^{\sigma\sqrt{2M_0}}e^{-T}+e^{-T}
\end{align*}
uniformly for $x\in[x^*,1]\setminus\Upsilon$.

Thus together with Case \uppercase\expandafter{\romannumeral1} and Case \uppercase\expandafter{\romannumeral2}, from the definition of $M_0$ we see $e^h\le e^{\sigma\sqrt{2M_0}}$ and then for all $x\in [x^*,1]$
\begin{align}\label{eq:quaderrortaper}
I(x)-S(x)=e^{\sigma \sqrt{2M_0}}\mathcal{O}(e^{-\sigma\alpha\sqrt{N_1}})+\mathcal{O}\left(\frac{x^\alpha}{e^{\frac{2\pi a}{h}}-1}\right).
\end{align}

Fig. \ref{quadratureerr_tapered} illustrates the quadrature errors $\max_{x\in [0,1]}|S(x)-I(x)|$ of the rectangular rule  with various values of $h$ and $\alpha$. From Fig. \ref{quadratureerr_tapered}, we see that the quadrature errors with respect to $h=4\pi^2\alpha$ are smaller than those to $h=8\pi^2\alpha$ and $h=\pi^2\alpha$ for fixed $\alpha$.
%, which are confirmed in the following subsection theoretically.

\section{Convergence rates of the LPs}\label{sec:4}
Having obtained the exponential decay orders as stated in \eqref{errquad_uniform} and \eqref{eq:quaderrortaper}, we can now derive the convergence orders of the rational approximations $r_N(x)$ \eqref{eq:rat}
and $\bar{r}_N(x)$ \eqref{eq: brateounif} and then  complete the proofs of Theorem \ref{mainthm} and Conjecture \ref{Conjecture 3.1}.

\begin{lemma}\label{eq:Paly}
Let $h=\sigma^2\alpha^2$, $\eta:=\frac{4\pi^2\alpha}{\sigma^2\alpha^2}=\frac{4\pi^2}{\sigma^2\alpha}$,
and $Q(x)=\frac{x^\alpha}{e^{2\pi a/h}-1}$ for $x\in [e^{\frac{1}{\alpha}(\gamma-T)},1]$ with $a=2\alpha\pi(T+\alpha\log{x})$ being the distance of the nearest poles of $f(u,x)$
to the real axis. Then it holds uniformly for $x\in [e^{\frac{1}{\alpha}(\gamma-T)},1]$ that
\begin{align}\label{eq:qerr1}
\left\{\begin{array}{ll}
\frac{1}{e^{\eta T}-1}\le Q(x)\le \frac{e^\gamma e^{-T}}{e^{\eta\gamma}-1},&\eta\ge 1 \,\,({\rm i.e.,}\, h\le 4\pi^2\alpha),\\
\frac{(1-\eta)^{1-\frac{1}{\eta}}e^{-T}}{\eta}\le Q(x)\le \max\{\frac{1}{e^{\eta T}-1},\frac{e^\gamma e^{-T}}{e^{\eta\gamma}-1}\},&\eta<1\,\,({\rm i.e.,}\, h> 4\pi^2\alpha).\end{array}\right.
\end{align}
\end{lemma}
\begin{proof}
From the definition of $a$, $Q(x)$ can be written as $Q(x)=\frac{x^\alpha}{x^{\eta\alpha} e^{\eta T}-1}$, then we have
\begin{align*}
\frac{\mathrm{d}}{\mathrm{d}x}Q(x)=\frac{(1-\eta)x^{\eta\alpha} e^{\eta T}-1}{(x^{\eta\alpha} e^{\eta T}-1)^2}\alpha x^{\alpha-1},
\end{align*}
which implies $Q(1)\le Q(x)\le Q(x^*)$ if $\eta\ge1$, i.e., the first case of \eqref{eq:qerr1}, while $Q(x_m)\le Q(x)\le\max\left\{Q(x^*),\ Q(1)\right\}$ if $\eta<1$, i.e., the second case of \eqref{eq:qerr1}, where $x_m=\left[(1-\eta)^{-\frac{1}{\eta}}e^{-T}\right]^{\frac{1}{\alpha}}$ is the minimum point of $Q(x)$ on $[x^*,1]$ and $x^*=e^{\frac{1}{\alpha}(\gamma-T)}$.
\end{proof}

{\bf Proof of  Theorem \ref{mainthm}}: Note that
\begin{align*}
\frac{e^\gamma e^{-T}}{e^{\eta\gamma}-1}< \frac{e^\gamma e^{-T}}{\eta\gamma}\le \frac{\sigma^2\alpha e^\gamma e^{-T}}{4\pi^2}<\frac{\sigma^2 e^\gamma e^{-T}}{2}<e^{\sigma} e^\gamma e^{-T}<e^{\sigma\sqrt{2M_0}} e^\gamma e^{-T}
\end{align*}
and $e^{-T}<\frac{1}{e^{\eta T}-1}$ for $\eta<1$ and $T>0$,
then from Lemma \ref{eq:Paly}, we see that
\begin{align*}
\frac{x^\alpha}{e^{2\pi a/h}-1}=&\left\{\begin{array}{ll}
\mathcal{O}(e^{-T}),& h\le 4\pi^2\alpha\\
\mathcal{O}\left(\max\left\{\frac{1}{e^{\eta T}-1},e^{\sigma\sqrt{2M_0}} e^\gamma e^{-T}\right\} \right),& h> 4\pi^2\alpha \end{array}\right\}\\
=&\left\{\begin{array}{ll}
\mathcal{O}(e^{-T}),& h\le 4\pi^2\alpha,\\
e^{\sigma\sqrt{2M_0}} e^\gamma\frac{\mathcal{O}(1)}{e^{\eta T}-1},& h> 4\pi^2\alpha, \end{array}\right.
\end{align*}
which, by the uniform boundedness of $\gamma$ together with  Theorem \ref{la1} and \eqref{eq:quaderrortaper}, follows  %as $T\rightarrow +\infty$ that
\begin{align}\label{eq:prove_thm1esti}
\|I-r_{N_t}\|_{\infty}
=&e^{\sigma \sqrt{2M_0}}\mathcal{O}(e^{-T})+
\left\{\begin{array}{ll}
\mathcal{O}(e^{-T}),& h\le 4\pi^2\alpha\\
e^{\sigma \sqrt{2M_0}}\frac{\mathcal{O}(1)}{e^{\eta T}-1},& h> 4\pi^2\alpha\end{array}\right\}\notag\\
=&e^{\sigma \sqrt{2M_0}}\cdot\left\{\begin{array}{ll}\mathcal{O}(e^{- T}),& h\le 4\pi^2\alpha. \\
\frac{\mathcal{O}(1)}{e^{\eta T}-1},& h> 4\pi^2\alpha.\end{array}\right.
\end{align}
Applying $T=\sqrt{N_1h}$, $h=\sigma^2\alpha^2$ and
 \begin{align*}%\label{eq:N1}
 \sqrt{N_1}=\sqrt{N-N_2}=\sqrt{N}\left(1+\mathcal{O}\left(\frac{N_2}{N}\right)\right)=\sqrt{N}+\mathcal{O}(1)
 =:\sqrt{N}+c_N
 \end{align*}
 where $c_N(<0)$ is uniformly bounded and independent of $N$,
from \eqref{eq:prove_thm1esti} it has for $h\le 4\pi^2\alpha$ (i.e., $\sigma\le \frac{2\pi}{\sqrt{\alpha}}$) that
\begin{align*}%\label{case1thm1.1}
 \|I-r_{N_t}\|_{\infty}=e^{\sigma \sqrt{2M_0}}\mathcal{O}(e^{-T})
 =e^{\sigma \sqrt{2M_0}}\mathcal{O}(e^{-\sigma\alpha\sqrt{N}})
 \end{align*}
since
\begin{align*}
e^{- T}= e^{-\sigma\alpha\sqrt{N}-c_N\sqrt{h}}=\mathcal{O}(e^{-\sigma\alpha\sqrt{N}}),%\ \ h\le 4\pi^2\alpha,
\end{align*}
thus we prove the first case of \eqref{eq: brateo} with $C=1$ from $x^\alpha=r_{N_1}(x)+P_{N_2}(x)+E_{Q}(x)+E_T(x)+E_{PA}(x)$.

When $h>4\pi^2\alpha$ (i.e., $\eta<1$), we have by $\sqrt{h}=\sigma\alpha<\sigma\sqrt{2M_0}$ that
\begin{align*}
e^{-T}<e^{ -\eta T}= e^{-\frac{4\pi^2}{\sigma}\sqrt{N}-c_N\eta \sqrt{h}}\le \frac{e^{|c_N|}e^{\sigma \sqrt{2M_0}}}{e^{\frac{4\pi^2}{\sigma}\sqrt{N}}-1},
\end{align*}
which implies that
 \begin{align}\label{case2-1thm1.1}
 \|I-r_{N_t}\|_{\infty}=e^{\sigma \sqrt{2M_0}}\frac{\mathcal{O}(1)}{e^{\frac{4\pi^2}{\sigma}\sqrt{N}}-1},
\end{align}
and then
the second case of \eqref{eq: brateo} with $C=1$.

Subsequently, we prove \eqref{eq: brateounif}. By \eqref{errquad_uniform} and $x^\alpha=\bar{r}_{N_1}(x)+\bar{P}_{N_2}(x)+\bar{E}_{Q}(x)+\bar{E}_T(x)+\bar{E}_{PA}(x)$,
$T=N_1\hbar$, $\hbar=\frac{\sigma\alpha}{\sqrt{N_1}}$ it follows
directly for the case of $C=1$ that
\begin{equation*}
|\bar{r}_N(x)-x^\alpha|=\mathcal{O}(e^{-T})+\frac{\mathcal{O}(1)}{e^{\frac{2\alpha\pi^2}{\hbar}}-1}
=\left\{\begin{array}{ll}
\mathcal{O}(e^{-\sigma\alpha\sqrt{N}}),&\sigma\le \frac{\sqrt{2}\pi}{\sqrt{\alpha}},\\
\frac{\mathcal{O}(1)}{e^{\frac{2\pi^2}{\sigma}\sqrt{N}}-1},&\sigma> \frac{\sqrt{2}\pi}{\sqrt{\alpha}}.
\end{array}\right.
\end{equation*}

For the general cases $C>0$ in \eqref{eq:tapered2} and \eqref{eq:uniform}, we rewrite the integral \eqref{eq:int} as
\begin{align*}
x^\alpha &=\frac{e^{\delta_0}\sin(\alpha\pi)}{\alpha\pi}
\int_{-\infty}^{+\infty}
\frac{xe^{\tilde t}}{Ce^{\frac{1}{\alpha}\tilde t}+x}\mathrm{d}\tilde t
=\frac{C^{\alpha}\sin(\alpha\pi)}{\alpha\pi}
\int_{-\infty}^{+\infty}
\frac{\hat{x}e^{\tilde t}}{e^{\frac{1}{\alpha}\tilde t}+\hat{x}}\mathrm{d}\tilde t,
\end{align*}
where we make the substitutions $\hat x=x/C\in[0,1/C]$ and $t=\tilde t+\delta_0$ such that $C=e^{\frac{\delta_0}{\alpha}}$. By the same procedure
on the last integral (with $\hat x$ instead of $x$ in the integrand), we obtain the desired results \eqref{eq: brateo} and \eqref{eq: brateounif} of Theorem \ref{mainthm} in exactly the same manner.

This completes the proof.

\vspace{.5cm}
In particular, if $h=4\pi^2\alpha$ (i.e., $\sigma=\frac{2\pi}{\sqrt{\alpha}}$ in \eqref{eq: brateo}), then from Theorem \ref{mainthm} it establishes Conjecture \ref{Conjecture 3.1}.

\begin{corollary}\label{mainthm1}
There exist coefficients $\{a_j\}_{j=1}^{N_1}$ and a polynomial $P_{N_2}$ with
$N_2$ $=\mathcal{O}(\sqrt{N_1})$, for which  $r_N(x)$ \eqref{eq:rat} having
tapered lightning poles \eqref{eq:tapered2} with $\sigma = \frac{2\pi}{\sqrt{\alpha}}$
satisfies:
\begin{equation*}%\label{eq: brate1}
|r_N(x)-x^\alpha|=\mathcal{O}(e^{-2\pi\sqrt{\alpha N}})
\end{equation*}
as $N \rightarrow \infty$, uniformly for $x\in [0, 1]$ and the constant in the $\mathcal{O}$ term is independent of $\alpha$ and $N$.
\end{corollary}

Additionally, by letting $\sigma=\frac{\sqrt{2}\pi}{\sqrt{\alpha}}$ in \eqref{eq: brateounif} we have from Theorem \ref{mainthm} that
\begin{corollary}\label{optimalrat1}
There exist coefficients $\{\bar{a}_j\}_{j=1}^{N_1}$ and a polynomial $\bar{P}_{N_2}$ with
$N_2$ $= \mathcal{O}(\sqrt{N_1})$, for which $\bar{r}_N(x)$ \eqref{LPbasedonuniformclupole}
having
uniform lightning poles \eqref{eq:uniform} with $\sigma = \frac{\sqrt{2}\pi}{\sqrt{\alpha}}$
satisfies:
\begin{equation*}%\label{eq: brate1}
|\bar{r}_N(x)-x^\alpha|=\mathcal{O}(e^{-\pi\sqrt{2\alpha N}})
\end{equation*}
as $N \rightarrow \infty$, uniformly for $x\in [0, 1]$ and the constant in the $\mathcal{O}$ term is independent of $\alpha$ and
$N$.
\end{corollary}

%==========================================================================

\begin{appendix}
\section{Useful lemmas for the proof of Theorem \ref{mainthm}}\label{AppendixA}
We prove the three lemmas used in the proof of Theorem \ref{eq:thm} in Subsection \ref{subsec:3.2} only for the case $n\ge1$, and those for $n\le-1$ can be proved analogously.

\begin{lemma}\label{la3}
It holds
\begin{align*}
 &\bigg{|}\int_{0}^{2a}
\big[f(h+it,x)-f(h-it,x)\big]
e^{-\frac{2n\pi}{h}t}\mathrm{d}t\bigg{|}\\
=&\left(\frac{e^{h}}{h^2}
\int_0^1te^{-\frac{2n\pi}{h}t}\mathrm{d}t
+\frac{e^{\sqrt{h}}}{\sqrt{h}}\int_1^{+\infty}\frac{e^{\sqrt{t}}}{\sqrt{t}}
e^{-\frac{2n\pi}{h}t}\mathrm{d}t\right)\mathcal{O}(e^{-T})
\end{align*}
%$$
% \bigg{|}\int_{0}^{2a}\big{[}f(h+it,x)-f(h-it,x)\big{]}
%e^{-\frac{2n\pi}{h}t}\mathrm{d}t\bigg{|}
%=\mathcal{O}(e^{-T})\int_0^{2a}te^{\sqrt{t}}e^{-\frac{2n\pi}{h}t}\mathrm{d}t
%$$
uniformly for $x\in \Upsilon$ and the constant in the above $\mathcal{O}$ term is independent of $T$, $n$, $x$, $h$, $\alpha$ and $\sigma$.
\end{lemma}
\begin{proof}
We rewrite the first part of the integrand on the left-hand side as follows
\begin{align*}%\label{0000}
 &f(h+it,x)-f(h-it,x)\\
=&e^{-T}\frac{\sin(\alpha\pi)}{2\alpha\pi}
\bigg{[}\frac{1}{\sqrt{h+it}}
\frac{xe^{\sqrt{h+it}}}{e^{\frac{1}{\alpha}(\sqrt{h+it}-T)}+x}
-\frac{1}{\sqrt{h-it}}
\frac{xe^{\sqrt{h-it}}}{e^{\frac{1}{\alpha}(\sqrt{h-it}-T)}+x}\bigg{]}\notag.
\end{align*}

Let
\begin{align*}
\varphi(t,x)=&\frac{x}
{e^{\frac{1}{\alpha}(\sqrt{h-it}-T)}+x},\quad t\in[0,2a].
\end{align*}
Since $x\in\Upsilon$, we have from the assumption \eqref{eq:real} that $\sqrt[4]{M_0h}\ge \sqrt{2\alpha\pi}+\sqrt[4]{4h}$, and it follows
$$\sqrt{T+\alpha\log{x}}=\sqrt[4]{v_0+\alpha^2\pi^2}\ge\sqrt[4]{M_0h}
\ge\sqrt{2\alpha\pi}+\sqrt[4]{4h}
\ge\frac{\sqrt{2\alpha\pi}+\sqrt{2\alpha\pi+8\sqrt{h}}}{2}$$
and then $(\sqrt{T+\alpha\log{x}}-\sqrt{\alpha\pi/2})^2\ge \alpha\pi/2+2\sqrt{h}$,
i.e., $\sqrt{2\alpha\pi}\sqrt{T+\alpha\log{x}}+2\sqrt{h}-(T+\alpha\log{x})
\le0$,
which together with \eqref{eq:sqrtr} and \eqref{eq:sqrtru0} implies
\begin{align*}
\Re(\sqrt{h-it}-\sqrt{u_0})
=&\sqrt{\frac{\sqrt{h^2+t^2}+h}{2}}-\Re(\sqrt{u_0})
\le\sqrt{\frac{\sqrt{h^2+4a^2}+h}{2}}-\Re(\sqrt{u_0})\\
\le&\sqrt{a}+\sqrt{h}-\Re(\sqrt{u_0})
=\sqrt{2\alpha\pi}\sqrt{T+\alpha\log{x}}+\sqrt{h}-(T+\alpha\log{x})\\
\le&-\sqrt{h}=-\alpha\sigma,
\end{align*}
and
$|e^{\frac{1}{\alpha}(\sqrt{h-it}-\sqrt{u_0})}|\le e^{-\sigma}$, and
\begin{align*}
\big{|}\varphi(t,x)\big{|}
=\frac{x}{|x+e^{\frac{1}{\alpha}(\sqrt{h-it}-T)}|}=\frac{1}{|e^{\frac{1}{\alpha}(\sqrt{h-it}-\sqrt{u_0})}-1|}
\le\frac{1}{1-e^{-\sigma}}.
\end{align*}
Analogously, we have
$$
\frac{x}{|x+e^{\frac{1}{\alpha}(\sqrt{h+it}-T)}|}=\frac{1}{|e^{\frac{1}{\alpha}(\sqrt{h+it}-\sqrt{u_0})}-1|}
\le\frac{1}{1-e^{-\sigma}}.
$$

Define
$$\phi(t,x)=\frac{1}{\sqrt{h+it}}
\frac{xe^{\sqrt{h+it}}}{e^{\frac{1}{\alpha}(\sqrt{h+it}-T)}+x}=\frac{e^{\sqrt{h+it}}}{\sqrt{h+it}}
\frac{1}{e^{\frac{1}{\alpha}(\sqrt{h+it}-\sqrt{u_0})}-1},\quad t\in[-1,1],
$$
then $\phi(t,x)$ is analytic for $t\in [-1,1]$ and $\partial_t\phi$ is continuous on $[-1,1]\times [0,1]$, and from \cite{Mcleod1965} it obtains
$$
|\phi(t,x)-\phi(-t,x)|\le 2\|\partial_t\phi\|_{\infty}t,\quad t\in [-1,1].
$$
Moreover, it is easy to verify that
\begin{align}\label{eq:phibound2}
 \bigg\|\frac{e^{\sqrt{h+it}}}{\sqrt{h+it}}\bigg\|_{L^\infty[-1,1]}\le\frac{e^{\sqrt{h}+1}}{\sqrt{h}},\quad
\bigg\|\frac{\mathrm{d}}{\mathrm{d}t}\frac{1}{e^{\frac{1}{\alpha}(\sqrt{h+it}-\sqrt{u_0})}-1}\bigg\|_{L^\infty[-1,1]}\le \frac{e^{-\sigma}}{2\alpha\sqrt{h}(1-e^{-\sigma})^2},
\end{align}
 which together with
$h=\sigma^2\alpha^2$ leads to
\begin{align}\label{eq:phiderbound2}
\|\partial_t\phi\|_{\infty}&\le \bigg\|\frac{\mathrm{d}}{\mathrm{d}t}\frac{e^{\sqrt{h+it}}}{\sqrt{h+it}}\bigg\|_{L^\infty[-1,1]}\frac{1}{1-e^{-\sigma}}
+\bigg\|\frac{e^{\sqrt{h+it}}}{\sqrt{h+it}}\bigg\|_{L^\infty[-1,1]}
\frac{e^{-\sigma}}{2\alpha\sqrt{h}(1-e^{-\sigma})^2}\\
&\le \frac{e^{\sqrt{h}+1}(\sqrt{h}+1)}{2h\sqrt{h}}\frac{1}{1-e^{-\sigma}}+\frac{e^{\sqrt{h}+1-\sigma}}{2\alpha h(1-e^{-\sigma})^2}\notag\\
&=\mathcal{O}(1) h^{-2}e^h\notag
\end{align}
with $\mathcal{O}(1)$ independent of $T$, $n$, $x$, $\alpha$, $\sigma$ and $h$, where for the estimate of the first term  in the above inequality \eqref{eq:phiderbound2}, we used $\sqrt{h}=\sigma\alpha<\sigma$, $|e^{\sqrt{h+it}}|\le  e^{\sqrt{h}+1}$ for $0<h,t\le 1$ while $\sqrt{h}e^{\sqrt{h}}\le e^h$ for $h\ge 1$, while for the estimate of the second term, we also used $1-e^{-\sigma}\le \sigma$ if $0<\sigma\le 1$ while $1-e^{-\sigma}>\frac{1}{2}$ and $e^{\sigma}\ge \frac{1}{2}\sigma^2$ if $\sigma>1$.

Consequently, we get by  $\sqrt{t}\le |\sqrt{h\pm it}|\le \sqrt{h}+\sqrt{t}$ for $t\in [1,2a]$ that
\begin{align*}%\label{eq:lea1bounds}
 &\bigg{|}\int_{0}^{2a}
\big[f(h+it,x)-f(h-it,x)\big]
e^{-\frac{2n\pi}{h}t}\mathrm{d}t\bigg{|}\\
=&e^{-T}\frac{\sin(\alpha\pi)}{2\alpha\pi}\bigg{|}\int_{0}^{1}(\phi(t,x)-\phi(-t,x))e^{-\frac{2n\pi}{h}t}\mathrm{d}t\notag\\
&+\int_1^{2a}
\left(\frac{e^{\sqrt{h+it}}}{\sqrt{h+it}}\varphi(-t,x)-\frac{e^{\sqrt{h-it}}}{\sqrt{h-it}}\varphi(t,x)\right)e^{-\frac{2n\pi}{h}t}\mathrm{d}t\bigg{|}\notag\\
=&\mathcal{O}(1)\bigg(
\frac{e^{h}}{h^{2}}\int_0^1te^{-\frac{2n\pi}{h}t}\mathrm{d}t
+\frac{1}{1-e^{-\sigma}}\int_1^{2a}\frac{e^{\sqrt{t}+\sqrt{h}}}{\sqrt{t}}
e^{-\frac{2n\pi}{h}t}\mathrm{d}t\bigg)e^{-T}\notag\\
=&\mathcal{O}(e^{-T})\bigg(
\frac{e^{h}}{h^2}\int_0^1te^{-\frac{2n\pi}{h}t}\mathrm{d}t
+\frac{e^{\sqrt{h}}}{\sqrt{h}}\int_1^{+\infty}\frac{e^{\sqrt{t}}}{\sqrt{t}}e^{-\frac{2n\pi}{h}t}\mathrm{d}t\bigg)\notag
\end{align*}
if $2a\ge 1$. Otherwise, the integral can be bounded by the first term with $\mathcal{O}(e^{-T})\cdot$ $\left(\frac{e^{h}}{h^2}\int_0^1te^{-\frac{2n\pi}{h}t}\mathrm{d}t\right)$.
These together lead to the desired result.
\end{proof}

\begin{lemma}\label{lemma_inte_circ}
Let $f(u,x)$ be defined in \eqref{eq:fun} with $x\in\Upsilon$. Suppose for some fixed sufficiently large number $N_0>0$ %independent of $x$, $h$ and $T$
that
$$
C^{\pm}_{\rho}=\{z=u_0+\rho e^{i\theta},\theta:0\rightarrow\pm2\pi\},\ \ 0<\rho\le\rho_0=:\frac{1}{2}\min\bigg\{\alpha^2\pi^2,2\pi\alpha\gamma,\frac{h}{N_0}\bigg\},
$$
then for arbitrary $\rho\in(0,\rho_0]$ it holds uniformly that
\begin{align}\label{55555}
\bigg{|}\int_{C^{\pm}_{\rho}}
f(u,x)e^{-i\frac{2n\pi}{h}u}\mathrm{d}u\bigg{|}
=e^{(\rho-a)\frac{2n\pi}{h}}x^{\alpha}\mathcal{O}(1)
\end{align}
as $T$ approaches to $+\infty$, the constant in the  $\mathcal{O}(1)$ \eqref{55555} is independent of $x$, $h$, $\alpha$, $\sigma$, $\rho$ and $T$.
\end{lemma}
\begin{proof}
Without loss of generality, we consider the case of $u\in C^{-}_{\rho}$, and the same argument can be developed for case $u\in C^{+}_{\rho}$.

Setting $u=:v+iw=u_0+\rho e^{i\theta}=v_0+\rho\cos\theta+i(\rho\sin\theta-a)$, the integral in \eqref{55555} can be estimated as
\allowdisplaybreaks[3]
\begin{align}\label{integral_bound_for_cir}
&\left|\int_{C^{-}_{\rho}}
f(u,x)e^{-i\frac{2n\pi}{h}u}\mathrm{d}u\right|\\
%=e^{-\frac{2n\pi}{h}a}\bigg|\int_{C^{-}_{\rho}+ia}
%f(u-ia,x)e^{-i\frac{2n\pi}{h}u}\mathrm{d}u\bigg|\\
\le&e^{-\frac{2n\pi}{h}a}\int_{0}^{2\pi}
\left|f(u_0+\rho e^{i\theta},x)e^{-i\frac{2n\pi}{h}(v_0+\rho\cos{\theta}+i\rho\sin{\theta})}
ie^{i\theta}\right|\rho \mathrm{d}\theta\notag\\
=&e^{(\rho-a)\frac{2n\pi}{h}}\int_{0}^{2\pi}
\left|\rho e^{i\theta}f(u_0+\rho e^{i\theta},x)\right|\mathrm{d}\theta\notag\\
=&e^{(\rho-a)\frac{2n\pi}{h}}\int_0^{2\pi}\left|\frac{e^{\sqrt{u_0+\rho e^{i\theta}}-T}}{2\sqrt{u_0+\rho e^{i\theta}}}\right|
\left|\frac{\rho e^{i\theta}}{e^{\frac{1}{\alpha}(\sqrt{u_0+\rho e^{i\theta}}-\sqrt{u_0})}-1}\right|\mathrm{d}\theta\notag\\
=&\alpha e^{(\rho-a)\frac{2n\pi}{h}}\int_0^{2\pi}\left|\frac{e^{\sqrt{u_0+\rho e^{i\theta}}-T}(\sqrt{u_0+\rho e^{i\theta}}+\sqrt{u_0})}{2\sqrt{u_0+\rho e^{i\theta}}}\right|
\left|\frac{\frac{\rho e^{i\theta}/\alpha}{\sqrt{u_0+\rho e^{i\theta}}+\sqrt{u_0}}}{e^{\frac{\rho e^{i\theta}/\alpha}{\sqrt{u_0+\rho e^{i\theta}}+\sqrt{u_0}}}-1}\right|\mathrm{d}\theta\notag.
\end{align}

We first bound the last term in the integrand of the last identity \eqref{integral_bound_for_cir} by using  $e^z-1\sim z$ for $z\rightarrow 0$ as follows,
 for sufficiently large $N_0$ there is a constant $C_0$ such that
\begin{align}\label{eq:bound_circ1}
\bigg{|}\frac{\rho e^{i\theta}/\alpha}{\sqrt{u_0+\rho e^{i\theta}}+\sqrt{u_0}}\bigg{|}\bigg{/}\bigg{|}
e^{\frac{\rho e^{i\theta}/\alpha}{\sqrt{u_0+\rho e^{i\theta}}+\sqrt{u_0}}}-1\bigg{|}
\le C_0
\end{align}
holds uniformly for $\alpha\in(0,1)$ since $\rho\le\alpha^2\pi^2$ and  $\Re(\sqrt{u_0})=T+\alpha\log{x}\ge \gamma>1$.

Next we estimate $\frac{\sqrt{u_0+\rho e^{i\theta}}+\sqrt{u_0}}{\sqrt{u_0+\rho e^{i\theta}}}$: From  \eqref{eq:real} it holds that $|u_0|\ge |v_0|\ge M_0h\ge2h$, and by $\rho\le\frac{h}{N_0}$ we obtain the following bound
$$
\bigg{|}\frac{u_0}{u}\bigg{|}
=\bigg{|}\frac{u_0}{u_0+\rho e^{i\theta}}\bigg{|}
\le\frac{|u_0|}{|u_0|-|\rho e^{i\theta}|}
\le\frac{|v_0|}{|v_0|-\frac{h}{N_0}}
\le\frac{2h}{2h-\frac{h}{N_0}}
\le 2,
$$
and then
\begin{align}\label{eq:bound_circ2}
\bigg{|}\frac{\sqrt{u_0+\rho e^{i\theta}}+\sqrt{u_0}}{\sqrt{u_0+\rho e^{i\theta}}}\bigg{|}\le 1+\sqrt{\bigg{|}\frac{u_0}{u_0+\rho e^{i\theta}}\bigg{|}}<3.
\end{align}

Finally, we consider $e^{\sqrt{u_0+\rho e^{i\theta}}-T}$: From $v_0>M_0h$ and the assumption \eqref{eq:real}, we have  with $c_0=\frac{\alpha\pi}{2}$ and
$a=2\alpha\pi(T+\alpha\log{x})=2\pi\alpha\sqrt{v_0+\alpha^2\pi^2}$  by  $v_0+\rho\ge v$ that
$$
2\big{(}\sqrt{v_0+\alpha^2\pi^2}+c_0\big{)}^2-v
=2v_0+2\alpha^2\pi^2+2c_0^2+\frac{2c_0a}{\alpha\pi}-v
\ge v_0+2\alpha^2\pi^2+a+2c_0^2-\rho.
$$
While by
$$
\sqrt{v^2+w^2}\le\sqrt{(v_0+\rho)^2+(a+\rho)^2}\le v_0+a+2\rho,
$$
we get from the definition of $\rho$ for sufficiently large $N_0$ that
$$
\sqrt{v^2+w^2}-
2\big{(}\sqrt{v_0+\alpha^2\pi^2}+c_0\big{)}^2+v
\le3\rho-2\alpha^2\pi^2-2c_0^2\le0,
$$
from which it yields
$$\sqrt{\frac{\sqrt{v^2+w^2}+v}{2}}\le\sqrt{v_0+\alpha^2\pi^2}
+\frac{\alpha\pi}{2},$$
$$
\Re(\sqrt{u})-T=\sqrt{\frac{\sqrt{v^2+w^2}+v}{2}}-T\le \sqrt{v_0+\alpha^2\pi^2}-T+\frac{\alpha\pi}{2}
=\alpha\log{x}+\frac{\alpha\pi}{2},
$$
and then
\begin{align}\label{eq:bound_circ3}
\bigg{|}e^{\sqrt{u_0+\rho e^{i\theta}}-T}\bigg{|}=e^{\Re(\sqrt{u})-T}\le x^\alpha e^{\frac{\alpha\pi}{2}}.
\end{align}

Substituting \eqref{eq:bound_circ1}, \eqref{eq:bound_circ2} and \eqref{eq:bound_circ3} into \eqref{integral_bound_for_cir},
we have that
\begin{align*}
&\bigg{|}\int_{C^{-}_{\rho}}
f(u,x)e^{-i\frac{2n\pi}{h}u}\mathrm{d}u\bigg{|}
%\\\le&\alpha e^{\rho-\frac{2n\pi}{h}a}\int_0^{2\pi}\bigg{|}\frac{e^{\sqrt{u_0+\rho e^{i\theta}}-T}(\sqrt{u_0+\rho e^{i\theta}}+\sqrt{u_0})}{2\sqrt{u_0+\rho e^{i\theta}}}\bigg{|}
%\bigg{|}\frac{\frac{\rho/\alpha}{\sqrt{u_0+\rho e^{i\theta}}+\sqrt{u_0}}}{e^{\frac{\rho/\alpha}{\sqrt{u_0+\rho e^{i\theta}}+\sqrt{u_0}}}-1}\bigg{|}d\theta\notag\\
=e^{(\rho-a)\frac{2n\pi}{h}}x^\alpha \mathcal{O}(1).
\end{align*}
\end{proof}

\bigskip
In the following, we turn to prove the boundedness of
$$
\bigg{|}
\int_{h\pm 2ia}^{+\infty\pm 2ia}f(u,x)e^{\pm i\frac{2n\pi}{h}u}\mathrm{d}u\bigg{|}=e^{-\frac{4n\pi a}{h}}\bigg{|}\int_{h}^{+\infty}f(t\pm 2ia,x)e^{\pm i\frac{2n\pi}{h}t}\mathrm{d}t\bigg{|}.
$$
The key strategy here is to divide the integral interval $[h,+\infty)$ into three
subintervals
$$[h,v_L],\ [v_L,v_R],\ [v_R,+\infty),$$
on which the integrals of $|f(t-2ia,x)|$ will be bounded by $x^{\alpha}\mathcal{O}(1)$, where the dividing points satisfies $h<v_L<v_0<v_R<+\infty$ (see Fig. \ref{locationofhatu}). The proof can be directly applied to the integrals of $|f(t+2ia,x)|$.

We seek two points $u_L=v_L-2ia$ and $u_R=v_R-2ia$ locating on the left and right sides of $u_0-ia$, respectively, such that
\begin{align}\label{bounds of Im(t-2ia)}
-\frac{5}{2}\alpha\pi=\Im(\sqrt{u_L})
\le\Im(\sqrt{u})\le\Im(\sqrt{u_R})
=-\frac{3}{2}\alpha\pi,\hspace{.3cm}u=t-2ia,\ v_L\le t\le v_R.
\end{align}
%Then we specify $v_1=\Re(u_L)$ and $v_2=\Re(u_R)$.

\begin{figure}[htp]
\centerline{\includegraphics[width=16cm]{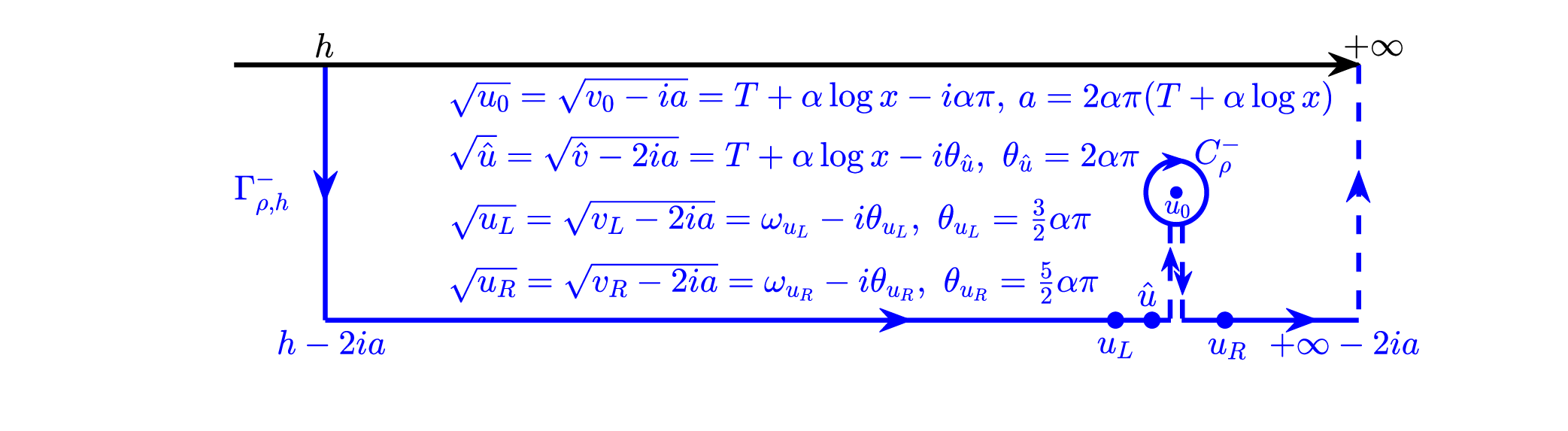}}
\caption{The locations of $u_L$ and $u_R$, which divide the integral path $[h-2ia,+\infty-2ia)$ into three parts $[h-2ia,u_L],\ [u_L,u_R]$ and $[u_R,+\infty-2ia)$
satisfying \eqref{bounds of Im(t-2ia)}, \eqref{monotonicity on h-baru}, \eqref{monotonicity on baru-infty}, \eqref{eq:lowerboundforsqruh} and \eqref{eq:upperboundforsqruh}, where $\hat u$ lies between the points $u_L$ and $u_R$, that is, $v_L=\Re(u_L)<\Re(\hat u)=\hat v<\Re(u_R)=v_R$. Additionally, $\Re(\sqrt{u_0})=\Re(\sqrt{\hat{u}})$.}
\label{locationofhatu}
\end{figure}

The following observations are much important for the choosing of $u_L$ and $u_R$.
Denote $u=t-2ia$ with $t\in[h,+\infty)$, then by \eqref{eq:sqrtr} it follows
\begin{align}\label{presention of sqrt u}
\sqrt{u}=\sqrt{\frac{\sqrt{t^2+4a^2}+t}{2}}-i\sqrt{\frac{2a^2}{\sqrt{t^2+4a^2}+t}}
=:\omega_u-i\theta_u,
\end{align}
which implies that both of its real part $\Re(\sqrt{u})$ and imaginary part $\Im(\sqrt{u})$ are strictly monotonically increasing with respect to $t\in[h,+\infty)$, and $\Re(\sqrt{u})$ is  positive and $\Im(\sqrt{u})$ negative.

At first, we show that
\begin{align}\label{bounds of Im(u0-ia)}
-\frac{5}{2}\alpha\pi<\Im\big(\sqrt{v_0-2ia}\big)
=\Im\big(\sqrt{u_0-ia}\big)<-\frac{3}{2}\alpha\pi.
\end{align}

Set $\hat{u}=\hat{v}-2ia$ such that $\Re(\sqrt{\hat{u}})=\Re(\sqrt{u_0})$, which is equivalent to $\theta_{\hat{u}}=2\alpha\pi$
by $\Re(\sqrt{u_0})=\frac{a}{2\alpha\pi}$ and $\Re(\sqrt{\hat{u}})=\frac{a}{\theta_{\hat{u}}}$ (according to \eqref{presention of sqrt u}).
Then according to the monotonicity of $\Re(\sqrt{u})$ and $\Im(\sqrt{u})$ for $t\in [h,+\infty)$ and $u$ on the line segment $[h-2ia,+\infty-2ia)$ , we have
\begin{align}
&\Re(\sqrt{u}-\sqrt{u_0})<0\hspace{.5cm}\text{ for }\hspace{.5cm}u=t-2ia {\rm \, \,and\,\,} t\in[h,\hat{v}),\label{monotonicity on h-baru}\\
&\Re(\sqrt{u}-\sqrt{u_0})>0\hspace{.5cm}\text{ for }\hspace{.5cm}u=t-2ia  {\rm \, \, and\,\,} t\in(\hat{v},+\infty),\label{monotonicity on baru-infty}
\end{align}
which together with
$$\Re(\sqrt{\hat{u}})=\Re(\sqrt{u_0})=\sqrt{\frac{\sqrt{v_0^2+a^2}+v_0}{2}}
<\sqrt{\frac{\sqrt{v_0^2+4a^2}+v_0}{2}}=\Re(\sqrt{v_0-2ia})$$  
implies that
\begin{align}\label{eq:lowerboundforsqruh}
-\frac{5}{2}\alpha\pi<-2\alpha\pi=\Im\big(\sqrt{\hat u}\big)<
\Im\big(\sqrt{v_0-2ia}\big).
\end{align}
%according to the monotonicity of $\Re(\sqrt{u})$ and $\Im(\sqrt{u})$.

Moreover, we have by \eqref{eq:real} that $v_0=(T+\alpha\log{x})^2-\alpha^2\pi^2>M_0h\ge 18\pi^2\alpha^2$, thus it is easy to verify that
\begin{align}\label{ratio of a v0}
\frac{a}{v_0}
=&\frac{2\alpha\pi\sqrt{v_0+\alpha^2\pi^2}}{v_0}
<\frac{1}{2}.
\end{align}
Furthermore, from \eqref{presention of sqrt u} and \eqref{ratio of a v0} it holds that
\begin{align*}
4-\frac{\Im^2(\sqrt{v_0-2ia})}{\Im^2(\sqrt{u_0})}
=&4-\frac{4\big(\sqrt{v_0^2+a^2}+v_0\big)}{\sqrt{v_0^2+4a^2}+v_0}
=\frac{4\big(\sqrt{v_0^2+4a^2}-\sqrt{v_0^2+a^2}\big)}{\sqrt{v_0^2+4a^2}+v_0}\\
=&\frac{12a^2}{\big(\sqrt{v_0^2+4a^2}+v_0\big)
\big(\sqrt{v_0^2+4a^2}+\sqrt{v_0^2+a^2}\big)}\\
\le&\frac{12a^2}{\big(\sqrt{4a^2+4a^2}+2a\big)
\big(\sqrt{4a^2+4a^2}+\sqrt{4a^2+a^2}\big)}\hspace{1cm}\\
=&\frac{6}{(\sqrt{2}+1)(2\sqrt{2}+\sqrt{5})},
\end{align*}
and then
\begin{align*}
\bigg|\frac{\Im(\sqrt{v_0-2ia})}{\Im(\sqrt{u_0})}\bigg|=\frac{-\Im(\sqrt{v_0-2ia})}{-\Im(\sqrt{u_0})}
\ge \sqrt{4-\frac{6}{(\sqrt{2}+1)(2\sqrt{2}+\sqrt{5})}}>\frac{18}{10},
\end{align*}
which implies that
\begin{align}\label{eq:upperboundforsqruh}
\Im(\sqrt{v_0-2ia})
<\frac{18}{10}\Im(\sqrt{u_0})
=-\frac{18}{10}\alpha\pi<-\frac{3}{2}\alpha\pi.
\end{align}

Inspirited by \eqref{bounds of Im(u0-ia)}, we may choose $u_L:=v_L-2ia$ satisfying $\Im(\sqrt{u_L})=-\frac{5}{2}\alpha\pi$. Consequently we have
from \eqref{presention of sqrt u} and $\Im(\sqrt{u_0})=-\alpha\pi$ that
\begin{align*}
\sqrt{\frac{\sqrt{v_L^2+4a^2}-v_L}{2}}=-\Im(\sqrt{u_L})=-\frac{5}{2}\Im(\sqrt{u_0})
=\frac{5}{2}\sqrt{\frac{\sqrt{v_0^2+a^2}-v_0}{2}},
\end{align*}
and then
\begin{align}\label{bound for su-su0 on h-v1}
\sqrt{\frac{4a^2}{\sqrt{v_L^2+4a^2}+v_L}}
=\frac{5}{2}\sqrt{\frac{a^2}{\sqrt{v_0^2+a^2}+v_0}},\quad
\frac{\Re(\sqrt{u_0})}{\Re(\sqrt{u_L})}
=\frac{\sqrt{\frac{\sqrt{v_0^2+a^2}+v_0}{2}}}{\sqrt{\frac{\sqrt{v_L^2+4a^2}+v_L}{2}}}
=\frac{5}{4}.
\end{align}
Subsequently, we have $u_L=\frac{16}{25}(T+\alpha\log x)^2-\frac{25}{4}\alpha^2\pi^2-2ia$ with $v_L>h$ by the assumption \eqref{eq:real}, and
from \eqref{monotonicity on h-baru} and \eqref{bound for su-su0 on h-v1} we get
\begin{align}\label{real_thefirstsegment}
\Re(\sqrt{u}-\sqrt{u_0})\le
\Re(\sqrt{u_L}-\sqrt{u_0})=-\frac{1}{5}\Re(\sqrt{u_0})
=-\frac{1}{5}(T+\alpha\log{x})\le-\frac{1}{5}
\end{align}
for $u=t-2ia$ and $t\in[h,v_L]$.

Similarly,
by choosing $u_R:=v_R-2ia$ satisfying $\Im(\sqrt{u_R})=-\frac{3}{2}\alpha\pi$, we have $\Re(\sqrt{u_R})=\frac{4}{3}\Re(\sqrt{u_0})$ and $u_R=\frac{16}{9}(T+\alpha\log x)^2-\frac{9}{4}\alpha^2\pi^2-2ia$ which yields
\begin{align}\label{real_thesecondsegment}
\Re(\sqrt{u}-\sqrt{u_0})\ge
\Re(\sqrt{u_R}-\sqrt{u_0})
=\frac{1}{3}\Re(\sqrt{u_0})
=\frac{1}{3}(T+\alpha\log{x})
\ge\frac{1}{3}
\end{align}
for $u=t-2ia$ and $t\in[v_R,+\infty)$.

Thus now we can choose $v_L$ and $v_R$ as the dividing points, which are the real parts of $u_L$ and $u_R$, respectively. Furthermore, $u_L$ and $u_R$ satisfy well the condition \eqref{bounds of Im(t-2ia)}.% since both of $\Re(\sqrt{u})$ and $\Im(\sqrt{u})$ strictly monotonically increase with respect to $t>0$.

\bigskip
\begin{lemma}\label{infty}
Let $f(u,x)$ be defined in \eqref{eq:fun} with $x\in\Upsilon$. Then
\begin{align}\label{bound_path_h-inf}
\bigg{|}
\int_{h\pm 2ia}^{+\infty\pm 2ia}f(u,x)e^{\pm i\frac{2n\pi}{h}u}\mathrm{d}u\bigg{|}
=&e^{-\frac{4n\pi a}{h}}\bigg{|}\int_{h}^{+\infty}f(t\pm 2ia,x)e^{\pm i\frac{2n\pi}{h}t}\mathrm{d}t\bigg{|}\\
=&e^{-\frac{4n\pi a}{h}}x^{\alpha}\mathcal{O}(1)\notag
\end{align}
holds uniformly for $x\in\Upsilon$, and the constant in $\mathcal{O}(1)$ \eqref{bound_path_h-inf} is independent of $x$, $h$, $\alpha$, $\sigma$ and $T$.
\end{lemma}
\begin{proof}
We only prove the case $f(t-2ia)$ of \eqref{bound_path_h-inf}, and the other can be proved in the same way.

At first, we estimate  the integrand of \eqref{bound_path_h-inf}
\begin{align*}%\label{integrand}
\big|f(u,x)\big|
=\frac{\sin{(\alpha\pi)}}{2\alpha\pi}
\frac{\big|xe^{\sqrt{u}-T}\big|}
{\big|\sqrt{u}\big|\big|e^{\frac{1}{\alpha}(\sqrt{u}-T)}+x\big|}
=\frac{\sin{(\alpha\pi)}}{2\alpha\pi}
\frac{\big|e^{\sqrt{u}-T}\big|}
{\big|\sqrt{u}\big|\big|e^{\frac{1}{\alpha}(\sqrt{u}-\sqrt{u_0})}-1\big|}
\end{align*}
on the subinterval $[h,v_L]$, where $u=t-2ia$ and $t\in[h,v_L]$.

From $\sqrt{u_0}=T+\alpha\log{x}-i\alpha\pi$, we have
\begin{align}\label{e^{sqrt_u-T}}
\big{|}e^{\sqrt{u}-T}\big{|}
=\big{|}e^{(\sqrt{u_0}-T)+(\sqrt{u}-\sqrt{u_0})}\big{|}
=x^{\alpha}e^{\Re(\sqrt{u}-\sqrt{u_0})},
\end{align}
and the exponent can be estimated by $v_0\ge2a$ \eqref{ratio of a v0} and $v_0>18\alpha^2\pi^2$ \eqref{eq:real}
as follows
\begin{align}\label{boundsfor_realpar_sqru-u0}
\Re(\sqrt{u}-\sqrt{u_0})
=&\sqrt{\frac{\sqrt{t^2+4a^2}+t}{2}}
-\sqrt{\frac{\sqrt{v^2_0+a^2}+v_0}{2}}\notag\\
=&\frac{1}{2}\frac{\sqrt{t^2+4a^2}-\sqrt{v_0^2+a^2}+(t-v_0)}
{\sqrt{\frac{\sqrt{t^2+4a^2}+t}{2}}
+\sqrt{\frac{\sqrt{v_0^2+a^2}+v_0}{2}}}\notag\\
=&\frac{1}{2}\frac{(t-v_0)
\big(\frac{t+v_0}{\sqrt{t^2+4a^2}+\sqrt{v_0^2+a^2}}+1\big)
+\frac{3a^2}{\sqrt{t^2+4a^2}+\sqrt{v_0^2+a^2}}}
{\sqrt{\frac{\sqrt{t^2+4a^2}+t}{2}}
+\sqrt{\frac{\sqrt{v_0^2+a^2}+v_0}{2}}}\\
\le&\frac{\frac{t-v_0}{2}
\big(\frac{t+v_0}{t+2a+v_0+a}+1\big)}
{\sqrt{t+a}+\sqrt{v_0+\frac{a}{2}}}
+\frac{\frac{3a^2}{\sqrt{t^2+4a^2}
+\big(\frac{a^2}{4\alpha^2\pi^2}+\alpha^2\pi^2\big)}}{2(\sqrt{t}+\sqrt{v_0})}\notag\hspace{.5cm}({\rm by}\,t-v_0<0)\\
\le&\frac{t-v_0}{6(\sqrt{t}+\sqrt{v_0})}
\bigg[\frac{t+v_0}{3(t+v_0)}+1\bigg]
+\frac{3a^2}{
\frac{2a^2}{4\alpha^2\pi^2}\sqrt{v_0}}\notag\hspace{.5cm}({\rm by}\,2a<v_0)\\
\le&\frac{2}{9}\big(\sqrt{t}-\sqrt{v_0}\big)
+\sqrt{2}\alpha\pi\notag\hspace{.5cm}({\rm by}\,v_0\ge18\alpha^2\pi^2)
%\le&\frac{2}{3}(\sqrt{t}-\sqrt{v_0})
%+\frac{3\sqrt{2}}{2}\notag
\end{align}
for $u=t-2ia$ and $t\in[h,v_0]$.
%where $\eta_0=3\alpha h$ is a positive number independent of $u_0$.

Based on the estimations \eqref{real_thefirstsegment}, \eqref{e^{sqrt_u-T}} and \eqref{boundsfor_realpar_sqru-u0} and by noticing $\big|\sqrt{t-2ia}\big|>\sqrt{t}$, the integral of $\big|f(t-2ia,x)\big|$ on the first subinterval $[h,v_L]$ satisfies that
\begin{align}\label{bound for first interval}
\int_{h}^{v_L}\big|f(t-2ia,x)\big|\mathrm{d}t
\le&\frac{x^{\alpha}e^{\sqrt{2}\alpha\pi}\sin{(\alpha\pi)}}{\big(1-e^{-\frac{1}{5\alpha}}\big)\alpha\pi}
\int_{h}^{v_L}\frac{e^{\frac{2}{9}(\sqrt{t}-\sqrt{v_0})}}{2\sqrt{t}}\mathrm{d}t
=x^{\alpha}\mathcal{O}(1)
\end{align}
holds uniformly for $x\in\Upsilon$ by $h<v_L<v_0$. Particularly, the constant in $\mathcal{O}(1)$ is independent of $x$,  $\alpha$, $\sigma$ and $T$.

In order to estimate the integral on the third subinterval $[v_R,+\infty)$, we rewrite the integrand
$|f(u,x)|$ as
\begin{align}\label{estimation of integrand}
|f(u,x)|
=&\frac{\sin{(\alpha\pi)}}{\alpha\pi}
\frac{x^{\alpha}}{2\big|\sqrt{t-2ia}\big|}
\frac{\big|e^{\sqrt{u}-\sqrt{u_0}}\big|}
{\big|e^{\frac{1}{\alpha}(\sqrt{u}-\sqrt{u_0})}\big|
\big{|}e^{-\frac{1}{\alpha}(\sqrt{u}-\sqrt{u_0})}-1\big{|}}\\
=&\frac{\sin{(\alpha\pi)}}{\alpha\pi}
\frac{x^{\alpha}}{2\sqrt[4]{t^2+4a^2}
e^{\frac{1}{\kappa}\Re(\sqrt{u}-\sqrt{u_0})}
\big{|}1-e^{-\frac{1}{\alpha}(\sqrt{u}-\sqrt{u_0})}\big{|}},\notag
\end{align}
where $u=t-2ia$ and $t\in[v_R,+\infty)$.

Analogous to \eqref{boundsfor_realpar_sqru-u0}, by using
$t-v_0>0$ and $\frac{3a^2}{\sqrt{t^2+4a^2}+\sqrt{v_0^2+a^2}}>0$ we have for the exponent $\Re(\sqrt{u}-\sqrt{u_0})$ that
%$$
%\frac{3a^2}{\bigg(\sqrt{t^2+4a^2}+\sqrt{v_0^2+a^2}\bigg)
%\bigg(\sqrt{\frac{\sqrt{t^2+4a^2}+t}{2}}
%+\sqrt{\frac{\sqrt{v_0^2+a^2}+v_0}{2}}\bigg)}\ge 0
%$$
\begin{align}\label{boundsfor_realpar_sqru-u0_v1-inf}
\frac{2}{9}(\sqrt{t}-\sqrt{v_0})\le\Re(\sqrt{u}-\sqrt{u_0}),
\end{align}
where $u=t-2ia$ and $t\in[v_0,+\infty)$.
With the bounds \eqref{real_thesecondsegment} and \eqref{boundsfor_realpar_sqru-u0_v1-inf}, we have
\begin{align*}
\big{|}1-e^{-\frac{1}{\alpha}(\sqrt{u}-\sqrt{u_0})}\big{|}
&\ge 1-e^{-\frac{1}{\alpha}\Re(\sqrt{u}-\sqrt{u_0})}
\ge  1-e^{-\frac{1}{3\alpha}},
\end{align*}
and then
\begin{align}\label{bound for third interval}
\int_{v_R}^{+\infty}\big|f(t-2ia,x)\big|\mathrm{d}t
\le&\frac{x^{\alpha}\sin{(\alpha\pi)}}
{\big(1-e^{-\frac{1}{3\alpha}}\big)\alpha\pi}
\int_{v_R}^{+\infty}
\frac{e^{-\frac{2}{9\kappa}(\sqrt{t}-\sqrt{v_0})}}{2\sqrt{t}}\mathrm{d}t\\
=&\frac{9x^{\alpha}\sin{(\alpha\pi)}e^{-\frac{2}{9\kappa}(\sqrt{v_R}-\sqrt{v_0})}}
{2\left(1-e^{-\frac{1}{3\alpha}}\right)(1-\alpha)\pi}
=x^{\alpha}\mathcal{O}(1)\notag
\end{align}
holds uniformly for $x\in\Upsilon$ by $v_R>v_0$, and the constant in $\mathcal{O}(1)$ is independent of $x$, $h$,  $\alpha$, $\sigma$ and $T$.

Now, we turn to the middle subinterval $[v_L,v_R]$. Since $\Im(\sqrt{u_0})=-\alpha\pi$, it is easy to check by \eqref{bounds of Im(t-2ia)}, \eqref{real_thefirstsegment} and \eqref{real_thesecondsegment} that
\begin{align*}
-\frac{3\pi}{2}\le\frac{1}{\alpha}\Im(\sqrt{u}-\sqrt{u_0})
\le-\frac{\pi}{2},
\hspace{0.5cm}
-\frac{1}{5}\le\Re(\sqrt{u}-\sqrt{u_0})\le\frac{1}{3}
\end{align*}
hold for $u=t-2ia$ and $t\in[v_L,v_R]$,
which by denoting $\theta_d:=\Im(\sqrt{u}-\sqrt{u_0})$ implies that
\begin{align}\label{bound for sqrt-sign1}
\big|e^{\frac{1}{\alpha}(\sqrt{u}-\sqrt{u_0})}-1\big|
=&\sqrt{e^{\frac{2}{\alpha}\Re(\sqrt{u}-\sqrt{u_0})}
-2e^{\frac{1}{\alpha}\Re(\sqrt{u}-\sqrt{u_0})}
\cos{\bigg(\frac{\theta_d}{\alpha}\bigg)}+1}\\
\ge&\sqrt{1+e^{\frac{2}{\alpha}\Re(\sqrt{u}-\sqrt{u_0})}}
\ge\sqrt{1+e^{-\frac{2}{5\alpha}}}>1\notag
\end{align}
for $u=t-2ia$ and $t\in[v_L,\hat v]$.
Similarly, we have
\begin{align}\label{bound for sqrt-sign2}
\big|1-e^{-\frac{1}{\alpha}(\sqrt{u}-\sqrt{u_0})}\big|
\ge\sqrt{1+e^{-\frac{2}{3\alpha}}}>1
\end{align}
for $u=t-2ia$ and $t\in[\hat v,v_R]$.

Thus, according to \eqref{monotonicity on h-baru} and \eqref{monotonicity on baru-infty}
it establishes
\begin{align}\label{bounds middle interval}
\int_{v_L}^{v_R}\big|f(t-2ia,x)\big|\mathrm{d}t
=&\frac{\sin{\alpha\pi}}{\alpha\pi}
\int_{v_L}^{\hat v}\frac{x^{\alpha}e^{\Re(\sqrt{u}-\sqrt{u_0})}}
{2\big|\sqrt{u}\big|\big|e^{\frac{1}{\alpha}(\sqrt{u}-\sqrt{u_0})}-1\big|}\mathrm{d}t\\
&+\frac{\sin{\alpha\pi}}{\alpha\pi}
\int_{\hat v}^{v_R}\frac{x^{\alpha}e^{-\frac{1}{\kappa}\Re(\sqrt{u}-\sqrt{u_0})}}
{2\big|\sqrt{u}\big|
\big|1-e^{-\frac{1}{\alpha}\Re(\sqrt{u}-\sqrt{u_0})}\big|}\mathrm{d}t
\notag\\
\le&\frac{x^{\alpha}e^{\sqrt{2}\alpha\pi}\sin{(\alpha\pi)}}{\alpha\pi}
\int_{v_L}^{\hat v}\frac{e^{\frac{2}{9}(\sqrt{t}-\sqrt{v_0})}}
{2\sqrt{t}}\mathrm{d}t
\hspace{.3cm}(\text{by \eqref{boundsfor_realpar_sqru-u0} {\rm and} \eqref{bound for sqrt-sign1}})\notag\\
&+\frac{x^{\alpha}\sin{(\alpha\pi)}}
{\alpha\pi}
\int_{\hat v}^{v_R}\frac{e^{-\frac{2}{9\kappa}(\sqrt{t}-\sqrt{v_0})}}
{2\sqrt{t}}\mathrm{d}t\notag\hspace{.3cm}(\text{by \eqref{boundsfor_realpar_sqru-u0_v1-inf} {\rm and} \eqref{bound for sqrt-sign2}})\\
=&\frac{9x^{\alpha}e^{\sqrt{2}\alpha\pi}\sin{(\alpha\pi)}}{2\alpha\pi}
\left[e^{\frac{2}{9}(\sqrt{\hat{v}}-\sqrt{v_0})}-e^{\frac{2}{9}(\sqrt{v_L}-\sqrt{v_0})}\right]\notag\\
&+\frac{9x^{\alpha}\sin{(\alpha\pi)}}{2(1-\alpha)\pi}
\left[e^{-\frac{2}{9\kappa}(\sqrt{\hat{v}}-\sqrt{v_0})}-e^{-\frac{2}{9\kappa}(\sqrt{v_R}-\sqrt{v_0})}\right]\notag\\
=&x^{\alpha}\mathcal{O}(1)\notag
\end{align}
holds uniformly for $x\in\Upsilon$, and the constant in $\mathcal{O}(1)$ is independent of $x$,  $\alpha$, $\sigma$ and $T$, where we used in the last step in \eqref{bounds middle interval} the inequalities $v_L<\hat v<v_0<v_R$ and the fact
\begin{align*}
0>\sqrt{\hat v}-\sqrt{v_0}=&\sqrt{(T+\alpha\log{x})^2-4\alpha^2\pi^2}
-\sqrt{(T+\alpha\log{x})^2-\alpha^2\pi^2}\\
=&\frac{-3\alpha^2\pi^2}{\sqrt{(T+\alpha\log{x})^2-4\alpha^2\pi^2}
+\sqrt{(T+\alpha\log{x})^2-\alpha^2\pi^2}}\\
>&-\frac{3}{8}\alpha\pi
\end{align*}
followed from $(T+\alpha\log{x})^2-\alpha^2\pi^2\ge(T+\alpha\log{x})^2-4\alpha^2\pi^2\ge16\alpha^2\pi^2$ by \eqref{eq:real}, and thus $e^{-\frac{2}{9\kappa}(\sqrt{\hat{v}}-\sqrt{v_0})}\le e^{\frac{(1-\alpha)\pi}{12}}$.

Adding \eqref{bound for first interval}, \eqref{bound for third interval} and \eqref{bounds middle interval} all up, it completes the proof of the case of
\begin{align*}
\bigg{|}
\int_{h-2ia}^{+\infty-2ia}f(u,x)e^{-i\frac{2n\pi}{h}u}\mathrm{d}u\bigg{|}
=&e^{-\frac{4n\pi a}{h}}\bigg{|}\int_{h}^{+\infty}f(t-2ia,x)e^{- i\frac{2n\pi}{h}t}\mathrm{d}t\bigg{|}\\
=&e^{-\frac{4n\pi a}{h}}x^{\alpha}\mathcal{O}(1),
\end{align*}
where the constant in the above $\mathcal{O}(1)$ is independent of $x$,  $\alpha$, $\sigma$ and $T$.
\end{proof}

\end{appendix}
%%%%%%%%%%%%%%%%%%%%%%%%%%%%%%%%%%%%%%%%%%%

%\vspace{0.16cm}
%\noindent {\bf Ethics declarations}\\
%\noindent {\bf Conflict of interest}\\
%\noindent The authors declare no competing interests.

\section*{Acknowledgement}
%The authors would like to thank Guidong Liu and Kelong Zhao  for  helpful discussion and constructive suggestions.
This work was supported by National Science Foundation of China 
(No. 12271528).

\bibliographystyle{siamplain}
\bibliography{references}
\end{document}